\RequirePackage{fix-cm}
\documentclass[smallextended]{svjour3}       
\smartqed  
\usepackage{graphicx}
\usepackage{mathptmx}      
\usepackage[usenames,dvipsnames,svgnames,table]{xcolor}
\usepackage[english]{babel}
	
\usepackage[hidelinks]{hyperref}
\usepackage{amsmath,amssymb} 
\usepackage{mathtools}
\usepackage{graphicx}
\usepackage{enumitem}

\usepackage{tikz,pgfplots}
\pgfplotsset{compat=1.10}

\usepackage{thmtools}
\usepackage{thm-restate}
 
\makeatletter
\newcommand{\mypm}{\mathbin{\mathpalette\@mypm\relax}}
\newcommand{\@mypm}[2]{\ooalign{%
  \raisebox{.1\height}{$#1+$}\cr
  \smash{\raisebox{-.6\height}{$#1-$}}\cr}}
\makeatother

\DeclareMathOperator{\Gram}{Gram}
\DeclareMathOperator{\rank}{rank}
\newcommand{\CSP}{\mathrm{CS}_{+}}
\newcommand{\CSPvN}{\mathrm{CS}_{+,\mathrm{vN}}^n}
\newcommand{\cpsd}{\operatorname{\mathrm{cpsd-rank}}}

\newcommand{\ignore}[1]{}
\newcommand{\CP}{\mathrm{CP}}
\newcommand{\cprank}{\text{\rm cp-rank}}
\newcommand{\psdrank}{\text{\rm psd-rank}}
\newcommand{\cpsdrank}{\text{\rm cpsd-rank}}
\newcommand{\xib}[2]{{\xi_{#2}^{\mathrm{#1}}}}

\newcommand{\N}{\mathbb{N}}

\newcommand{\R}{\mathbb{R}}
\newcommand{\K}{\mathbb{K}}
\newcommand{\C}{\mathbb{C}}
\newcommand{\oS}{\mathbb{S}}

\newcommand{\T}{{\sf T}}
\DeclareMathOperator{\Tr}{Tr}

\renewcommand{\sup}{\mathrm{sup}}
\renewcommand{\min}{\mathrm{min}}

\renewcommand{\max}{\mathrm{max}}

\renewcommand{\epsilon}{\varepsilon}

\newcommand{\Diag}{\text{\rm Diag}}

\DeclareMathOperator{\cpsdr}{cpsd-rank}
\DeclareMathOperator{\hcpsd}{cpsd-rank_{\mathbb C}}

\newcommand{\DAS}{\mathcal D_{\hspace{-0.08em}\mathcal A}(S)}
\newcommand{\DA}{\mathcal D_{\hspace{-0.08em}\mathcal A}}

\newcommand{\DS}{\mathcal D(S)}

\newcommand{\CDS}{D(S)}

\newcommand{\cx}{[{\bf x}]}
\newcommand{\ncx}{\langle {\bf x}\rangle}
\newcommand{\bx}{\mathbf{x}}
\newcommand{\bX}{\mathbf{X}}
\newcommand{\psdcone}{\mathrm{S}}
\newcommand{\Hermitian}{\mathrm{H}}
\newcommand{\Hilbert}{\mathcal H}
\newcommand{\elliptope}{\mathrm{E}}
\newcommand{\MD}{{\mathcal D}}
\newcommand{\MM}{{\mathcal M}}
\newcommand{\MI}{{\mathcal I}}
\newcommand{\taucpsos}{   \tau_{\mathrm{cp}}^{\mathrm{sos}}}
\newcommand{\taucp}{\tau_{\mathrm{cp}}}

\newcommand{\fmin}{f_{*}}

\newcommand{\MB}{{\mathcal B}}

\newcommand{\ftr}{f^{\text{tr}}}

\newcommand{\MA}{{\mathcal A}}

\newcommand{\SAcpsd}{S_A^{\hspace{0.1em}\mathrm{cpsd}}}
\newcommand{\SAplus}{S_A^{+}}
\newcommand{\SApsd}{S_A^{\hspace{0.1em}\mathrm{psd}}}
\newcommand{\SAVcpsd}{S_{A,V}^{\hspace{0.1em}\mathrm{cpsd}}}
\newcommand{\SAVkcpsd}{S_{A,V_k}^{\hspace{0.1em}\mathrm{cpsd}}}
\newcommand{\SAScpsd}{S_{A,\oS^{n-1}}^{\hspace{0.1em}\mathrm{cpsd}}}
\newcommand{\SABcpsd}{S_{A\oplus B}^{\hspace{0.1em}\mathrm{cpsd}}}
\newcommand{\SBcpsd}{S_B^{\hspace{0.1em}\mathrm{cpsd}}}
\newcommand{\SAcp}{S_A^{\hspace{0.1em}\mathrm{cp}}}
\newcommand{\SAVcp}{S_{A,V}^{\hspace{0.1em}\mathrm{cp}}}

\providecommand{\qedsymbol}{$\square$}
\newcommand{\mathqed}{\quad\hbox{\qedsymbol}}
\DeclareRobustCommand{\qed}{%
  \ifmmode \mathqed
  \else
    \leavevmode\unskip\penalty9999 \hbox{}\nobreak\hfill
    \quad\hbox{\qedsymbol}%
  \fi
}

\begin{document}
 
\title{Lower bounds on matrix factorization ranks via noncommutative polynomial optimization\thanks{The first and second authors are supported by the Netherlands Organization for Scientific Research, grant number 617.001.351, and the second author by the ERC Consolidator Grant QPROGRESS 615307.}
}


\author{Sander Gribling \and David de Laat \and Monique Laurent }

\authorrunning{Gribling, de Laat, Laurent} 

\institute{S. Gribling \at
              CWI, Amsterdam 
           \and
           D. de Laat \at
              CWI, Amsterdam
              \and 
              M. Laurent \at
              CWI, Amsterdam, and Tilburg University
}

\date{Received: date / Accepted: date}

\maketitle

\begin{abstract}
We use techniques from (tracial noncommutative) polynomial optimization to formulate hierarchies of semidefinite programming lower bounds on matrix factorization ranks. In particular, we consider the nonnegative rank, the positive semidefinite rank, and their symmetric analogues: the completely positive rank and the completely positive semidefinite rank. We study convergence properties of our hierarchies, compare them extensively to known lower bounds, and provide some (numerical) examples.
\end{abstract}

\section{Introduction}

\subsection{Matrix factorization ranks}

A factorization of a matrix $A \in \R^{m \times n}$ over a sequence $\{K^d\}_{d\in\N}$ of cones that are each equipped with an inner product $\langle\cdot,\cdot\rangle$ is a decomposition of the form $A=(\langle X_i,Y_j\rangle)$ with  $X_i, Y_j \in K^d$ for all $(i,j)\in [m]\times [n]$, for some integer $d\in \N$. Following \cite{GPT13}, the smallest integer $d$ for which such a factorization exists is called the \emph{cone factorization rank} of $A$ over $\smash{\{K^d\}}$.

The cones $K^d$ we use in this paper are the nonnegative orthant $\R^d_+$ with the usual inner product and the cone $\psdcone^d_+$ (resp., $\Hermitian^d_+$) of $d\times d$ real symmetric (resp., Hermitian) positive semidefinite matrices with the trace inner product $\smash{\langle X, Y \rangle = \mathrm{Tr}(X^\T Y)}$ (resp., $\smash{\langle X, Y \rangle = \mathrm{Tr}(X^* Y)}$). We obtain the \emph{nonnegative rank}, denoted $\rank_+(A)$, which uses the cones $K^d=\R^d_+$, and the \emph{positive semidefinite rank}, denoted $\smash{\psdrank_\K(A)}$, which uses the cones $K^d=\psdcone^d_+$ for $\K = \R$ and $K^d=\Hermitian^d_+$ for $\K=\C$. Both the nonnegative rank and the positive semidefinite rank are defined whenever $A$ is entrywise nonnegative.

The study of the nonnegative rank is largely motivated by the groundbreaking work of Yannakakis~\cite{Yan91}, who showed that the linear extension complexity of a polytope $P$ is given by the nonnegative rank of its slack matrix. The \emph{linear extension complexity} of $P$ is the smallest 
 integer $d$ for which $P$ can be obtained as the linear image of an affine section of the nonnegative orthant $\R^d_+$. The {\em slack matrix} of $P$ is given by the matrix $(b_i-a_i^{\sf T}v)_{v\in V,i\in I}$, where $P= \text{conv}(V)$ and $P= \{x: a_i^{\sf T}x\le b_i\ (i\in I)\}$ are the point and hyperplane representations of $P$. Analogously, the {\em semidefinite extension complexity} of $P$ is  the smallest $d$ such that $P$ is the linear image of an affine section of the cone $\psdcone^d_+$ and it is given by the (real) positive semidefinite rank of its slack matrix~\cite{GPT13}.

The motivation to study the linear and semidefinite extension complexities is that polytopes with small extension complexity admit efficient algorithms for linear optimization. Well-known examples include spanning tree polytopes~\cite{KippM} and permutahedra~\cite{Goe15}, which have polynomial linear extension complexity, and the stable set polytope of perfect graphs, which has polynomial semidefinite extension complexity~\cite{GLS}  (see, e.g., the surveys~\cite{CGZ10,FGPRT15}).
The above connection to the nonnegative rank and to the positive semidefinite rank of the slack matrix can be used to show that 
a polytope does not admit a small extended formulation. Recently this connection was used to show that the linear extension complexities of the traveling salesman, cut, and stable set polytopes are exponential in the number of nodes~\cite{FMPTdW12}, and this result was extended to their semidefinite extension complexities in~\cite{LRS15}.
Surprisingly, the linear extension complexity of the matching polytope is also exponential~\cite{Roth14}, even though linear optimization over this set is polynomial time solvable~\cite{Edmonds65}. It is an open question whether the semidefinite extension complexity of the matching polytope is exponential. 
 
Besides this link to extension complexity, the nonnegative rank also finds applications in probability theory and in communication complexity, and the positive semidefinite rank has applications in quantum information theory and in quantum communication complexity (see, e.g.,~\cite{MSvS03,FFGT11,JSWZ13,FMPTdW12}).
	
For square symmetric matrices ($m=n$) we are also interested in {\em symmetric} analogues of the above matrix factorization ranks, where we require the same factors for the rows and  columns (i.e., $X_i = Y_i$ for all $i\in [n]$). 
The symmetric analog of  the nonnegative rank is the  \emph{completely positive rank}, denoted $\cprank(A)$, which uses the cones $K^d = \smash{\R_+^d}$, and the symmetric analog of the positive semidefinite rank is the \emph{completely positive semidefinite rank}, denoted $\cpsdr_\K(A)$, which uses  the cones $K^d=\psdcone^d_+$ if $\K=\R$ and $K^d=\Hermitian^d_+$ if $\K=\C$. These symmetric factorization ranks are not always well defined since  not every symmetric nonnegative matrix admits a symmetric factorization by nonnegative vectors or postive semidefinite matrices. The symmetric matrices  for which these parameters are well defined form convex cones known as the {\em completely positive cone}, denoted $\CP^n$, and the {\em completely positive semidefinite cone}, denoted $\CSP^n$. We have the inclusions  $\CP^n \subseteq \CSP^n \subseteq \psdcone_+^n$, which are known to be strict for $n\ge 5$. For details on these cones see~\cite{BSM03,BLP17,LP15} and references therein.

Motivation for the cones  $\CP^n$ and $\CSP^n$ comes in particular from their use to model classical and quantum information optimization problems.
For instance, graph parameters such as the stability number and the chromatic number can be written as linear optimization problems over the completely positive cone~\cite{dKP02}, and the same holds, more generally, for  quadratic problems with mixed binary variables~\cite{Burer09}. 
The $\cprank$ is widely studied in the linear algebra community; see, e.g.,~\cite{BSM03,SMBJS13,SMBBJS15,BSU14}. 

The  completely positive semidefinite cone was first studied in~\cite{LP15} to describe quantum analogues of the stability number and of the chromatic number of a graph. This was later extended to general graph homomorphisms in~\cite{SV16} and to graph isomorphism in~\cite{AMRSS}. In addition, as shown in~\cite{MR14,SV16}, there is a close connection
between the completely positive semidefinite cone and the set of quantum correlations. This also gives a relation between the completely positive semidefinite rank and the minimal entanglement dimension necessary to realize a quantum correlation. 
This connection has been used in~\cite{PSVW16,GdLL17,PV17} to construct matrices whose completely positive semidefinite rank is exponentially large in the matrix size. For the special case of synchronous quantum correlations the minimum entanglement dimension is directly given by the completely positive semidefinite rank of a certain matrix (see~\cite{GdLL17b}).

The following inequalities hold for the nonnegative rank and the positive semidefinite rank: We have
\[
\psdrank_\C(A)\leq \psdrank_\R(A) \leq \rank_+(A) \leq \min \{m,n\}
\]
for any $m\times n$ nonnegative matrix $A$ and $\cprank(A)\le \binom{n+1}{2}$ for any $n\times n$ completely positive matrix $A$. However, the situation for the cpsd-rank is very different. Exploiting the connection between the completely positive semidefinite cone and quantum correlations it follows from results in~\cite{Slofstra17} that the cone $\smash{\CSP^n}$ is not closed for $n\ge 1942$. The results in~\cite{DPP17} show that this already holds for $n\ge 10$. As a consequence there does not exist an upper bound on the $\cpsdrank$ as a function of the matrix size. For small matrix sizes very little is known. It is  an open problem whether $\CSP^5$ is closed, and we do not even know how to construct a $5 \times 5$ matrix whose cpsd-rank exceeds~$5$. 

The $\rank_+$, $\cprank$, and $\text{psd-rank}$ are known to be computable; this follows  using results from~\cite{Ren92} since upper bounds exist on these factorization ranks that depend only on the matrix size, see~\cite{BR06} for a proof for the case of the $\cprank$.  But computing the nonnegative rank is NP-hard~\cite{Vav09}. In fact, determining the $\rank_+$ and $\psdrank$ of a matrix are both equivalent to the existential theory of the reals~\cite{Shi16a,Shi16b}.
For the cp-rank and the cpsd-rank no such results are known, but there is no reason to assume they are any easier. In fact it is not even clear whether the cpsd-rank is computable in general. 

To obtain upper bounds on the factorization rank of a given matrix one can employ heuristics that try to construct small factorizations. Many such heuristics exist for the nonnegative rank (see the overview~\cite{Gillis17} and references therein),  factorization algorithms exist for completely positive matrices (see the recent paper~\cite{GP18}, also~\cite{DD12} for structured completely positive matrices), and algorithms to compute positive semidefinite factorizations are presented in the recent work~\cite{VGG17}.  In this paper we want to compute {\em lower bounds} on matrix factorization ranks, which we achieve by employing a relaxation approach based on (noncommutative) polynomial optimization. 

\subsection{Contributions and connections to existing bounds}

In this work we provide a unified approach to obtain lower bounds on the four matrix factorization ranks mentioned above,  based on tools from (noncommutative) polynomial optimization. 

We sketch the main ideas of our approach in Section~\ref{secsketch} below, after having introduced some necessary notation and preliminaries about (noncommutative) polynomials in Section~\ref{sec:prelim}.  We then indicate in Section~\ref{seclink}  how our approach relates to the more classical use of polynomial optimization dealing with the minimization of polynomials over basic closed semialgebraic sets. The main body of the paper consists of four sections each dealing with one of the four matrix factorization ranks. We start  with  presenting our approach for the completely positive semidefinite rank and then explain how to adapt this to the other ranks. 

For our results we need several  technical tools about linear forms on spaces of polynomials, both in the commutative and noncommutative setting. To ease the readability of the paper we group these technical tools in Appendix~\ref{sec:background}. Moreover, we provide full proofs, so that our paper is self-contained. In addition,  some of the proofs might differ from the customary ones in the literature since our treatment  in this paper is consistently on the `moment' side rather than using real algebraic results about sums of squares. 

In Section~\ref{sec:lower bounds on cpsd rank} we introduce our approach for the completely positive semidefinite rank. We start by defining a hierarchy of lower bounds
\[
\xib{cpsd}{1}(A) \leq \xib{cpsd}{2}(A) \leq \ldots\le \xib{cpsd}{t}(A)\le \ldots\leq \cpsdr_\C(A),
\] 
where $\xib{cpsd}{t}(A)$, for $t \in \N$, is given as the optimal value of a semidefinite program whose size increases with $t$. Not much is known about lower bounds for the cpsd-rank in the literature. 
The inequality $\sqrt{\rank(A)} \le \cpsdr_\C(A)$ is known, which follows by viewing a Hermitian $d\times d$ matrix as a $d^2$-dimensional real vector, and an analytic lower bound is given in~\cite{PSVW16}. 
We show that the new parameter $\smash{\xib{cpsd}{1}(A)}$ is at least as good as this analytic lower bound and we give a small example where a strengthening of $\smash{\xib{cpsd}{2}(A)}$ is strictly better then both above mentioned generic lower bounds.
Currently we lack evidence that the lower bounds $\smash{\xib{cpsd}{t}(A)}$ can be larger than, for example, the matrix size, but this could be because small matrices with large cpsd-rank are hard to construct or might even not exist. 
We also introduce several ideas leading to strengthenings of the basic bounds $\xib{cpsd}{t}(A)$. 

We then adapt these ideas to the other three matrix factorization  ranks discussed above, where for each of them we obtain analogous hierarchies of bounds. 

For the nonnegative rank and the completely positive rank much more is known about lower bounds. The best known generic lower bounds are due to Fawzi and Parrilo~\cite{FP15,FP16}. In~\cite{FP16} the parameters $\tau_+(A)$ and $\tau_{\mathrm{cp}}(A)$ are defined, which, respectively,  lower bound the nonnegative rank and the $\cprank$, along with their computable semidefinite programming relaxations $\tau_\mathrm{+}^\mathrm{sos}(A)$ and $\tau_\mathrm{cp}^\mathrm{sos}(A)$. In~\cite{FP16} it is also shown that $\tau_+(A)$ is at least as good as certain norm-based lower bounds. In particular, $\tau_+(\cdot)$ is at least as good as the $\ell_\infty$ norm-based lower bound, which was used by Rothvo{\ss}~\cite{Roth14} to show that the matching polytope has exponential linear extension complexity. In~\cite{FP15} it is shown that for the Frobenius norm, the square of the norm-based bound is still a lower bound on the nonnegative rank, but it is not known how this lower bound compares to $\tau_+(\cdot)$.

Fawzi and Parrilo~\cite{FP16} use the atomicity of the nonnegative and completely positive ranks to derive the parameters $\tau_+(A)$ and $\tau_{\mathrm{cp}}(A)$; i.e., they use the fact that the nonnegative rank (cp-rank) of $A$  is equal to the smallest $d$ for which $A$ can be written as the sum of $d$ nonnegative (positive semidefinite) rank one matrices. As the $\psdrank$ and $\cpsdrank$ are not known to admit atomic formulations, the techniques from~\cite{FP16} do not extend directly 
to these factorization ranks. However, our approach via polynomial optimization captures these factorization ranks as well.

In Sections~\ref{sec:lowercp} and~\ref{sec:lowernnr} we construct semidefinite programming hierarchies of lower bounds $\xib{cp}{t}(A)$ and $\xib{+}{t}(A)$ on $\cprank(A)$ and $\rank_+(A)$. We show that the bounds $\xib{+}{t}(A)$ converge to $\tau_+(A)$ as $t \to \infty$. The basic hierarchy $\{\xib{cp}{t}(A)\}$ for the cp-rank does not converge to $\tau_{\mathrm{cp}}(A)$ in general, but we provide two types of additional constraints that can be added to the program defining $\xib{cp}{t}(A)$ to ensure convergence to $\tau_{\mathrm{cp}}(A)$. First, we show how a generalization of the tensor constraints that are used in the definition of the parameter $\smash{\tau_{\mathrm{cp}}^\mathrm{sos}(A)}$ can be used for this, and we also give a more efficient (using smaller matrix blocks) description of these constraints. This strengthening of $\xib{cp}{2}(A)$ is then at least as strong as $\tau_{\mathrm{cp}}^\mathrm{sos}(A)$, but requires matrix variables of roughly half the size. Alternatively, we show that for every $\epsilon >0$ there is a finite number of additional linear constraints that can be added to the basic hierarchy $\{\xib{cp}{t}(A)\}$ so that the limit of the sequence of these new lower bounds $\xib{+}{t}(A)$ is at least $\tau_{\mathrm{cp}}(A)-\epsilon$.  We give numerical results on small matrices studied in the literature, which show that $\xib{+}{3}(A)$ can improve over $\tau_{+}^\mathrm{sos}(A)$. 

Finally, in Section~\ref{sec:psdrank} we derive  a hierarchy $\{\smash{\xib{psd}{t}(A)}\}$ of lower bounds on the psd-rank. We compare the new bounds $\smash{\xib{psd}{t}(A)}$ to a bound from~\cite{LWdW17} and we provide some numerical examples illustrating their performance. 

We provide two implementations of all the lower bounds introduced in this paper, at the arXiv submission of this paper. One implementation uses Matlab and the CVX package~\cite{cvx14}, and the other one uses Julia~\cite{Julia17}. The implementations support various semidefinite programming solvers, for our numerical examples we used Mosek~\cite{mosek}.

\subsection{Preliminaries} \label{sec:prelim}

In order to explain our basic approach in the next section, we  first need to introduce some notation. We denote the set of all  words in the symbols $x_1,\ldots,x_n$ by $\langle {\bf x}\rangle = \langle x_1, \ldots, x_n \rangle$, where the empty word is denoted by $1$. This is a semigroup with involution, where the binary operation is concatenation, and the involution of a word $w\in \ncx$ is the word  $w^*$ obtained by  reversing the order of the symbols in $w$. The $*$-algebra of all real linear combinations of these words is denoted by $\R\langle {\bf x} \rangle$, and its elements are called \emph{noncommutative polynomials}. The involution extends to $\R\ncx$ by linearity. A polynomial $p\in \R\ncx$ is called {\em symmetric} if $p^*=p$  and $\mathrm{Sym} \, \R\ncx$ denotes the set of symmetric polynomials.
The degree of a word $w\in\ncx$ is the number of symbols composing it, denoted as $|w|$ or $\deg(w)$, and the degree of a polynomial $p=\sum_wp_ww\in \R\ncx$ is the maximum degree of a word $w$ with $p_w\ne 0$. Given $t\in \N \cup \{\infty\}$, we let $\langle {\bf x} \rangle_t$ be the set of words $w$ of degree $|w| \leq t$, so that $\langle  {\bf x} \rangle_\infty=\langle  {\bf x}\rangle$, and $\R\langle {\bf x} \rangle_t$ is the real vector space of noncommutative polynomials $p$ of degree $\mathrm{deg}(p) \leq t$. Given $t \in \N$, we let $\langle {\bf x} \rangle_{=t}$ be the set of words of degree exactly equal to $t$.
 
For a set $S\subseteq \mathrm{Sym} \,\R\ncx$  and $t\in \N\cup\{\infty\}$, the \emph{truncated quadratic module} at degree $2t$ associated to $S$ is defined as the cone generated by all polynomials $p^*g p \in \R\ncx_{2t}$  with $g\in S\cup\{1\}$:
\begin{equation}\label{eq:quadratic module}
\MM_{2t}(S)=\mathrm{cone}\Big\{p^*gp: p\in \R\ncx, \ g\in S\cup\{1\},\ \deg(p^*gp)\le 2t\Big\}.
\end{equation}
Likewise, for a set $T \subseteq \R\ncx$, we can define the \emph{truncated ideal} at degree $2t$, denoted by $\mathcal I_{2t}(T)$, as the vector space spanned by all polynomials $p h \in \R\ncx_{2t}$ with $h \in T$:
\begin{equation} \label{eq:ideal}
\mathcal I_{2t}(T) = \mathrm{span}\big\{ ph : p \in \R\langle \bx \rangle, \, h \in T, \, \mathrm{deg}(ph) \leq 2t \big\}.
\end{equation}
We say that $\MM(S) + \MI(T)$ is {\em Archimedean} when there exists a scalar $R>0$ such that 
\begin{equation}\label{eq:Arch}
R-\sum_{i=1}^n x_i^2\in \mathcal M(S)+ \mathcal I(T).
\end{equation}

Throughout we are interested in the space $\R\langle {\bf x} \rangle_t^*$ of real-valued linear functionals on $\R\langle {\bf x} \rangle_t$. We list some basic definitions: A linear functional $L \in \R\langle {\bf x} \rangle_t^*$ is \emph{symmetric} if $L(w) = L(w^*)$ for all $w \in \langle {\bf x} \rangle_t$ and \emph{tracial} if $L(ww') = L(w'w)$ for all  $w,w' \in \langle {\bf x} \rangle_t$. A linear functional $L \in \R\langle {\bf x} \rangle_{2t}^*$ is said to be \emph{positive} if $L(p^*p) \geq 0$ for all $p \in \R\langle {\bf x} \rangle_t$. Many properties of a linear functional $L \in \R\ncx_{2t}^*$ can be expressed as properties of its associated moment matrix (also known as its {\em Hankel matrix}). For $L \in \R\ncx_{2t}^*$ we define its associated \emph{moment matrix}, which has rows and columns indexed by words in $\ncx_t$, by 
\[
M_t(L)_{w,w'} = L(w^* w') \quad \text{for} \quad w,w' \in \ncx_t,
\]
and as usual we set $M(L) = M_\infty(L)$. It then follows that $L$ is symmetric if and only if $M_t(L)$ is symmetric, and $L$ is positive if and only if $M_t(L)$ is positive semidefinite. In fact, one can even express nonnegativity of a  linear form $L\in \R\ncx_{2t}^*$ on $\MM_{2t}(S)$  in terms of certain associated positive semidefinite moment matrices. For this, given a polynomial $g\in \R\ncx$, define the linear form $gL \in \smash{\R\langle{\bf x}\rangle_{2t-\deg(g)}^*}$ by $(gL)(p)=L(gp)$. 
Then we have $$L(p^*gp)\ge 0 \text{ for all } p\in \R\ncx_{t-d_g} \iff M_{t-d_g}(gL)\succeq 0, \quad (d_g = \lceil \deg(g)/2\rceil),$$
and thus $L\ge 0$ on $\MM_{2t}(S)$ if and only if $M_{t-d_{g}}(gL) \succeq 0$ for all $g\in S \cup \{1\}$. Also, the condition $L=0$ on $\MI_{2t}(T)$ corresponds to linear equalities on the entries of $M_t(L)$. 

The moment matrix also allows us to define a property called {\em flatness}. For $t \in \N$, a linear functional $L \in \R\ncx_{2t}^*$ is called \emph{$\delta$-flat} if the rank of $M_t(L)$ is equal to that of its principal submatrix indexed by the words in $\ncx_{t-\delta}$, that is, 
\begin{equation}\label{eq:flat}
\rank (M_t(L))=\rank (M_{t-\delta}(L)).
\end{equation}
We call $L$ \emph{flat} if it is $\delta$-flat for some $\delta \geq 1$. When $t=\infty$, $L$ is said to be {\em flat} when $\mathrm{rank}(M(L))<\infty$, which is equivalent to $\rank (M(L))=\rank(M_s(L))$ for some $s\in \N$.

A key example of a flat symmetric tracial positive linear functional on $\R\ncx$ is given by the \emph{trace evaluation} at a given matrix  tuple $\bX = (X_1,\ldots,X_n) \in (\Hermitian^d)^n$:
\[
p \mapsto \mathrm{Tr}(p(\bX)).
\]
Here $p(\bX)$ denotes the matrix obtained by substituting $x_i$ by $X_i$ in $p$, and throughout $\mathrm{Tr}(\cdot)$ denotes the usual matrix trace, which satisfies $\mathrm{Tr}(I) = d$ where $I$ is the identity matrix in $\Hermitian^d$. We mention in passing that we use $\mathrm{tr}(\cdot)$ to denote the \emph{normalized matrix trace}, which satisfies $\mathrm{tr}(I) = 1$ for $I \in \Hermitian^d$. 
Throughout, we use $L_\bX$ to denote the real part of the above functional, that is, $L_\bX$ denotes the linear form on $\R\ncx$ defined by 
\begin{equation} \label{eq:lx}
L_{\bf X}(p) = \mathrm{Re}(
\mathrm{Tr}(p(X_1,\ldots,X_n))) \quad \text{for} \quad p \in \R\langle{\bf x}\rangle.
\end{equation}
Observe that $L_\bX$ too is a symmetric tracial positive linear functional on $\R\ncx$. Moreover, $L_\bX$ is nonnegative on $\MM(S)$ if the matrix tuple $\bX$ is taken from the {\em matrix positivity domain} $\DS$ associated to the finite set $S \subseteq \mathrm{Sym} \, \R\ncx$, defined as 
\begin{equation}\label{eqDSnc}
\DS =\bigcup_{d\ge 1} \Big\{\bX=(X_1,\ldots,X_n)\in (\Hermitian^d)^n: g(\bX)\succeq 0 \text{ for } g\in S\Big\}.
\end{equation}
Similarly, the linear functional $L_\bX$ is zero on $\MI(T)$ if the matrix tuple $\bX$ is taken from the {\em matrix variety} 
$\mathcal V(T)$ associated to the finite set $T \subseteq \mathrm{Sym} \, \R\ncx$, defined as 
\[
\mathcal V(T) = \bigcup_{d\ge 1} \big\{\bX \in (\Hermitian^d)^n : h(\bX) = 0 \text{ for all } h \in T\big\},
\]

To discuss convergence properties of our lower bounds for matrix  factorization ranks we will need to consider infinite dimensional analogs of matrix algebras, namely $C^*$-algebras admitting a tracial state. 
Let us introduce some basic notions we need about $C^*$-algebras; see, e.g.,~\cite{Blackadar06} for details. For our purposes we define a {\em $C^*$-algebra} to be a norm closed $*$-subalgebra of the complex algebra $\MB(\Hilbert)$ of bounded operators on a complex Hilbert space $\Hilbert$. In particular, we have $\|a^*a\| = \|a\|^2$ for all elements $a$ in the algebra. Such an algebra $\MA$ is said to be {\em unital} if it contains the identity operator (denoted $1$). For instance, any full complex matrix algebra $\C^{d\times d}$ is a unital $C^*$-algebra. Moreover, by a fundamental result of Artin-Wedderburn, any finite dimensional $C^*$-algebra (as a vector space) is $*$-isomorphic to a direct sum $\bigoplus_{m=1}^M \C^{d_m\times d_m}$ of full complex matrix algebras~\cite{BEK78,Wed}. In particular, any finite dimensional $C^*$-algebra is unital. 

An element $b$ in a $C^*$-algebra $\MA$ is called {\em positive}, denoted $b\succeq 0$, if it is of the form $b=a^*a$ for some $a\in\MA$. For finite sets $S \subseteq \mathrm{Sym} \,\R\ncx$ and $T \subseteq \R\ncx$, the $C^*$-algebraic analogs of the matrix positivity domain and matrix variety are the sets
\begin{align*}
\DAS & = \big\{{\bf X}=(X_1,\ldots,X_n) \in \mathcal{A}^n : X_i^* = X_i \text{ for } i \in [n], \, g({\bf X}) \succeq 0 \text{ for } g \in S \big\},\\
\mathcal V_{\mathcal A}(T) & = \big\{{\bf X}=(X_1,\ldots,X_n) \in \mathcal{A}^n : X_i^* = X_i \text{ for } i \in [n], \, h({\bf X}) = 0 \text{ for } h \in T \big\}.
\end{align*}

A \emph{state} $\tau$ on a unital $C^*$-algebra $\MA$ is a linear form on $\MA$ that is {\em positive}, i.e., $\tau(a^*a)\ge 0$ for all $a\in \MA$, and satisfies $\tau(1)=1$. Since $\MA$ is a complex algebra, every state $\tau$ is Hermitian: $\tau(a) = \tau(a^*)$ for all $a \in \MA$. We say that that a state is \emph{tracial} if $\tau(ab) = \tau(ba)$ for all $a,b \in \mathcal A$ and {\em faithful} if $\tau(a^*a)=0$ implies $a=0$. A useful fact is that on a full matrix algebra $\C^{d\times d}$ the normalized matrix trace is the unique tracial state (see, e.g.,~\cite{BK12}). Now, given a tuple $\bX=(X_1,\ldots,X_n)\in \MA^n$ in a $C^*$-algebra $\MA$ with tracial state $\tau$, the second key example of a symmetric tracial positive linear functional on $\R\ncx$ is given by the {\em trace evaluation map}, which we again denote by $L_\bX$  and is defined by 
\[
L_\bX(p)=\tau (p(X_1,\ldots,X_n)) \quad \text{for all} \quad p\in\R\ncx.
\]

\subsection{Basic approach} \label{secsketch}
To explain the basic idea of how we obtain lower bounds for matrix factorization ranks we consider the case of the completely positive semidefinite rank. Given a minimal factorization $A=(\mathrm{Tr}(X_i,X_j))$, with $d=\cpsdr_\C(A)$ and $\bX=(X_1,\ldots,X_n)$ in $(\Hermitian_+^d)^n$, consider the linear form $L_{\bf X}$ on  
$\R\langle{\bf x}\rangle$ as defined in~\eqref{eq:lx}:
\[
L_{\bf X}(p) = \mathrm{Re}(
\mathrm{Tr}(p(X_1,\ldots,X_n))) \quad \text{for} \quad p \in \R\langle{\bf x}\rangle.
\]
Then we have $A=(L_{\bX}(x_ix_j))$ and  $\cpsdr_\C(A) = d=L_{\bX}(1)$. To obtain lower bounds on $\cpsdr_\C(A)$ we minimize $L(1)$ over a set of linear functionals $L$ that satisfy certain computationally tractable properties of $L_{\bf X}$. Note that this idea of minimizing $L(1)$ has recently been used in the works~\cite{TS15,Nie16} in the commutative setting to derive a  hierarchy of lower bounds converging to the nuclear norm of a symmetric tensor.

The above linear functional $L_{\bf X}$ is symmetric and tracial. Moreover it satisfies some positivity conditions, since we have $L_{\bf X}(q) \geq 0$ whenever $q({\bf X})$ is positive semidefinite. 
It follows that $L_{\bf X}(p^*p) \geq 0$ for all $p\in\R\ncx$ and, as we explain later, $L_{\bf X}$ satisfies the {\em localizing} conditions $L_{\bf X}(p^*(\sqrt{A_{ii}} x_i - x_i^2)p) \geq 0$ for all $p$ and $i$. Truncating the linear form yields the following hierarchy of lower bounds: 
\begin{align*}
\xib{cpsd}{t}(A) = \mathrm{min} \Big\{ L(1) : \; & L \in \R\langle x_1,\ldots,x_n \rangle_{2t}^* \text{ tracial and symmetric},\\[-0.2em]
&L(x_ix_j) = A_{ij} \quad \text{for} \quad i,j \in [n],\\[-0.2em]
&L \geq 0\quad \text{on} \quad \mathcal M_{2t}\big( \{\sqrt{A_{11}} x_1-x_1^2, \ldots,\sqrt{A_{nn}} x_n-x_n^2 \}\big)\Big\}.
\end{align*}
The bound $\xib{cpsd}{t}(A)$ is computationally tractable (for small $t$). Indeed, as was explained in Section \ref{sec:prelim}, the localizing constraint ``$L\ge 0$ on $\mathcal M_{2t}(S)$" can be enforced by requiring certain matrices, whose entries are determined by $L$, to be positive semidefinite. This makes the  problem defining $\smash{\xib{cpsd}{t}(A)}$ into a semidefinite program. The localizing conditions ensure the Archimedean property of the quadratic module, which permits to show certain convergence properties of the bounds $\smash{\xib{cpsd}{t}(A)}$.

The above approach extends naturally to the other matrix factorization ranks, using the following two basic ideas. First, since the cp-rank and the nonnegative rank deal with factorizations by {\em diagonal} matrices, we use linear functionals acting on classical {\em commutative} polynomials. Second, the {\em asymmetric} factorization ranks (psd-rank and nonnegative rank) can be seen as analogs of the symmetric ranks in the {\em partial matrix} setting, where we know only the values of $L$ on the quadratic monomials corresponding to entries in the off-diagonal blocks (this will require scaling of the factors in order to be able to define localizing constraints ensuring the Archimedean property). A main advantage of our approach is that it applies to all four matrix factorization ranks, after easy suitable adaptations.  

\subsection{Connection to polynomial optimization} \label{seclink}

In classical polynomial optimization the problem is to find the global minimum of a commutative polynomial $f$ over a semialgebraic set of the form 
\[
D(S) = \{x \in \R^n : g(x) \geq 0 \text{ for } g \in S\}, 
\]
where $S \subseteq \R[\bx] = \R[x_1,\ldots,x_n]$ is a finite set of polynomials.\footnote{Here, and throughout the paper, we use $\cx$ as the commutative analogue of $\ncx$.}
Tracial polynomial optimization is a noncommutative analog, where the problem is to minimize the normalized trace 
$\mathrm{tr}(f(\bX))$ of a symmetric polynomial $f$ over a matrix positivity domain $\DS$ where $S \subseteq \mathrm{Sym} \, \R\ncx$ is a finite set of symmetric polynomials.\footnote{In fact, one could consider optimization over $\DS \cap \mathcal V(T)$ for some finite set $T \subseteq \R\ncx$, the results below still hold in that setting, see Appendix~\ref{sec:background}.} Notice that the distinguishing feature here is the dimension independence: the optimization is over all possible matrix sizes. Perhaps counterintuitively, in this paper we use techniques similar to those used for the tracial polynomial optimization problem to compute lower bounds on factorization dimensions.

For classical polynomial optimization Lasserre~\cite{Las01} and Parrilo~\cite{Par00} have proposed hierarchies of semidefinite programming relaxations based on the theory of moments and the dual theory of sums of squares polynomials. These can be used to compute successively better lower bounds converging to  the global minimum (under the Archimedean condition). This approach has been used in a wide range of applications and there is an extensive literature (see, e.g.,~\cite{AL12,Las09,Lau09}). Most relevant to this work, it is used in~\cite{Las14} to design conic approximations of the completely positive cone and in~\cite{Nie14} to check membership in the completely positive cone. This approach has also been extended to the noncommutative setting, first to the eigenvalue optimization problem~\cite{NPA10,NPA12} (which will not play a role in this paper), and later to tracial optimization~\cite{BCKP13,KP16}. 

For our paper the moment formulation of the lower bounds is most relevant: For all $t \in \N \cup \{\infty\}$ we can define the bounds 
\begin{align*}
f_t &=\mathrm{inf}_{}\big\{L(f) : L\in \R\cx_{2t}^*,\, L(1)=1,\, L\ge 0  \text{ on }  \MM_{2t}(S)\big\}, \\
f_t^\mathrm{tr} &=\mathrm{inf}_{}\big\{L(f) : L\in \R\langle {\bf x}\rangle_{2t}^* \text{ tracial and symmetric},\, L(1)=1,\, L \ge 0 \text{ on } \MM_{2t}(S)\big\},
\end{align*}
where $f_t$ (resp., $f_t^\mathrm{tr}$) lower bounds the (tracial) polynomial optimization problem. 

The connection between the parameters $\xib{cpsd}{t}(A)$ and $f_t^\mathrm{tr}$ is now clear: in the former we do not have the normalization property ``$L(1)=1$" but we do have the additional affine constraints ``$L(x_i x_j) = A_{ij}$". This close relation to (tracial) polynomial optimization allows us to use that theory to understand the convergence properties of our bounds. 
Since  throughout the paper we use  (proof) techniques from (tracial) polynomial optimization, we will  state the main convergence results we need, with full proofs,  in Appendix~\ref{sec:background}. Moreover, we give all proofs from the ``moment side", which is most relevant to our treatment. Below we give a short summary of the convergence results for the hierarchies $\{f_t\}$ and $\{f_t^\mathrm{tr}\}$ that are relevant to our paper. We refer to Appendix \ref{sec:hierarchies} for details.

Under the condition that $\MM(S)$ is Archimedean we have asymptotic convergence: $f_t \rightarrow f_\infty$ and $f_t^\mathrm{tr} \to f_\infty^\mathrm{tr}$ as $t \to \infty$. In the commutative setting one can moreover show that $f_\infty$ is equal to the global minimum of $f$ over the set $D(S)$.  However, in the noncommutative setting, the parameter $f_\infty^\mathrm{tr}$ is in general not equal to the minimum of $\mathrm{tr}(f(\bX))$ over $\bX \in \mathcal D(S)$. Instead we need to consider the $C^*$-algebraic version of the tracial polynomial optimization problem: one can show that 
\[
f_\infty^\mathrm{tr}= \mathrm{inf} \big\{ \tau(f({\bf X})) : {\bf X} \in \DAS, \, \mathcal A \text{ is a unital $C^*$-algebra with tracial state } \tau\big\}.
\]
An important additional convergence result holds under flatness. If the program defining the bound $f_t$ (resp., $f_t^\mathrm{tr}$) admits a sufficiently flat optimal solution, then equality holds: $f_t = f_\infty$ (resp., $f_t^\mathrm{tr} = f_\infty^\mathrm{tr}$). Moreover, in this case, the parameter $f_t^\mathrm{tr}$ is equal to the minimum value of $\mathrm{tr}(f(\bX))$ over the matrix positivity domain $\mathcal D(S)$.

\section{Lower bounds on the completely positive semidefinite rank}
\label{sec:lower bounds on cpsd rank}

Let $A$ be a completely positive semidefinite $n \times n$ matrix. For  $t \in \N\cup\{\infty\}$ we consider the following semidefinite program,
which, as we see below, lower bounds the complex completely positive semidefinite rank of $A$:
\begin{align*}
\xib{cpsd}{t}(A) = \mathrm{min} \big\{ L(1) : \; & L \in \R\langle x_1,\ldots,x_n \rangle_{2t}^* \text{ tracial and symmetric},\\
&L(x_ix_j) = A_{ij} \quad \text{for} \quad i,j \in [n],\\[-0.2em]
&L \geq 0\quad \text{on} \quad \mathcal M_{2t}(\SAcpsd) \big\},
\end{align*}
where we set 
\begin{equation}\label{eqSA}
\SAcpsd = \big\{\sqrt{A_{11}} x_1 - x_1^2, \ldots, \sqrt{A_{nn}} x_n - x_n^2\big\}.
\end{equation}
Additionally, define the parameter $\smash{\xib{cpsd}{*}(A)}$, obtained by adding the rank constraint $\rank(M(L)) < \infty$ to the program defining $\smash{\xib{cpsd}{\infty}(A)}$, where we consider the infimum instead of the minimum since we do not know whether the infimum is always attained. (In Proposition~\ref{prop:lowerbound} we show the infimum is attained in $\smash{\xib{cpsd}{t}(A)}$ for $t\in\N\cup\{\infty\}$). This gives a hierarchy of monotone nondecreasing lower bounds on the completely positive semidefinite rank: 
\[
\xib{cpsd}{1}(A)  \leq \ldots \le\xib{cpsd}{t}(A)\le \ldots \leq \xib{cpsd}{\infty}(A) \leq \xib{cpsd}{*}(A)\leq \hcpsd(A).
\]
The inequality $\xib{cpsd}{\infty}(A)\le \xib{cpsd}{*}(A)$ is clear and monotonicity as well: If $L$ is feasible for $\smash{\xib{cpsd}{k}(A)}$ with $t \leq k \leq \infty$, then its restriction to $\R\ncx_{2t}$ is feasible for $\smash{\xib{cpsd}{t}(A)}$. 

The following notion of {\em localizing} polynomials will be useful. 
A set $S\subseteq \R\ncx$ is said to be {\em localizing} at a matrix tuple $\bX$ if $\bX \in \mathcal D(S)$ (i.e., $g(\bX)\succeq 0$ for all $g\in S$) and we say that $S$ is {\em localizing for $A$} if $S$ is localizing at some factorization $\bX\in (\Hermitian_+^d)^n$ of $A$ with $d=\hcpsd(A)$.
The set $\smash\SAcpsd$ as defined in~\eqref{eqSA} is  localizing for $A$, and, in fact, it is localizing at {\em any} factorization $\bX$ of $A$ by Hermitian positive semidefinite matrices. Indeed, since 
\[
A_{ii}=\Tr(X_i^2)\ge \lambda_{\max}(X_i^2) = \lambda_{\max} (X_i)^2
\]
we have $\sqrt{A_{ii}} X_i - X_i^2 \succeq 0$ for all $i\in [n]$.

We can now use this to show the inequality $\xib{cpsd}{*}(A) \le \hcpsd(A)$. For this set $d = \hcpsd(A)$,  let $\bX \in (\Hermitian_+^d)^n$ be a Gram factorization of $A$, and 
consider the linear form $L_\bX \in \R\langle \bx \rangle^*$ defined by
\[
L_\bX(p) = \mathrm{Re}(\mathrm{Tr}(p(\bX))) \quad \text{for all} \quad p \in \R\langle\bx\rangle.
\]
By construction $L_\bX$ is symmetric and tracial, and we have $A=(L(x_ix_j))$. Moreover, since the set of polynomials $\smash{\SAcpsd}$ is localizing for $A$, 
the linear form $L_\bX$ is nonnegative on $\smash{\mathcal M(\SAcpsd)}$. Finally, we have $\rank(M(L_\bX))<\infty$, since the algebra generated by $X_1, \ldots, X_n$ is finite dimensional. Hence, $L_\bX$ is feasible for $\smash{\xib{cpsd}{*}(A)}$ with $L_\bX(1)=d$, which shows $\smash{\xib{cpsd}{*}(A)} \le \hcpsd(A)$.

The inclusions in~\eqref{eq:inclusion} below show the quadratic module $\MM(\SAcpsd)$ is Archimedean (recall the definition in~\eqref{eq:Arch}). Moreover, although there are other possible choices for the localizing polynomials to use in $\smash{\SAcpsd}$, these inclusions also show that the choice made in~\eqref{eqSA} leads to the largest truncated quadratic module and thus to the best bound. For any scalar $c > 0$, we have the inclusions
\begin{equation}\label{eq:inclusion}
\MM_{2t}(x,c-x) \subseteq \MM_{2t}(x,c^2-x^2) \subseteq \MM_{2t}(cx-x^2) \subseteq \MM_{2t+2}(x,c-x),
\end{equation}
which hold in light of the following identities:
\begin{align}
c-x&= \big((c-x)^2 + c^2-x^2\big)/(2c), \label{eq:cx} \\
c^2 - x^2 &= (c-x)^2 + 2(cx - x^2), \label{eq:x^2}\\
cx - x^2 &= \big((c-x) x (c-x) + x(c-x)x\big)/c, \label{eq:x-x^2}\\
x &= \big( (cx - x^2) + x^2\big)/c. \label{eq:x}
\end{align} 

In the rest of this section we investigate properties of the  hierarchy $\{\xib{cpsd}{t}(A)\}$ as well as some variations on it. We discuss convergence properties, asymptotically and under flatness, and we give another formulation for the parameter $\smash{\xib{cpsd}{*}(A)}$. Moreover, as the inequality $\smash{\xib{cpsd}{*}(A)} \leq \hcpsd(A)$ is typically strict, we present an approach to strengthen the bounds in order to go beyond $\smash{\xib{cpsd}{*}(A)}$. Then we propose some techniques to simplify the computation of the bounds, and we illustrate the behaviour of the bounds on some examples.

\subsection{The parameters \texorpdfstring{$\xib{cpsd}{\infty}(A)$}{xi infinity} and \texorpdfstring{$\xib{cpsd}{*}(A)$}{xi star}}
\label{seccpsdconv}

In this section we consider convergence properties of the hierarchy $\xib{cpsd}{t}(\cdot)$, both asymptotically and under flatness. 
We also give equivalent reformulations of the limiting parameters $\smash{\xib{cpsd}{\infty}(A)}$ and $\smash{\xib{cpsd}{*}(A)}$ in terms of $C^*$-algebras  with a tracial state, which  we will use in Sections \ref{secstronger}-\ref{sec:addproperties} to show properties of these parameters.

\begin{proposition} \label{prop:lowerbound}
Let $A \in \CSP^n$. For  $t \in \N \cup \{\infty\}$ the optimum in $\xib{cpsd}{t}(A)$ is attained, and 
\[
\lim_{t \to \infty} \xib{cpsd}{t}(A) = \xib{cpsd}{\infty}(A).
\]
Moreover, $\smash{\xib{cpsd}{\infty}(A)}$ is equal to the  smallest  scalar $\alpha \geq 0$ for which there exists a unital $C^*$-algebra $\MA$ with tracial state $\tau$ and $(X_1,\ldots,X_n) \in \mathcal D_{\mathcal A}(\SAcpsd)$ such that $A = \alpha \cdot (\tau(X_iX_j))$.
\end{proposition}

\begin{proof}
The sequence $(\smash{\xib{cpsd}{t}(A)})_t$ is monotonically nondecreasing and upper bounded by $\smash{\xib{cpsd}{\infty}(A)} <\infty$, which implies its limit exists 
and is at most $\smash{\xib{cpsd}{\infty}(A)}$. 

As $\xib{cpsd}{t}(A)\le \xib{cpsd}{\infty}(A)$, we may add the redundant constraint $L(1) \leq \smash{\xib{cpsd}{\infty}(A)}$ to the problem $\xib{cpsd}{t}(A)$ for every $t \in \N$. By~\eqref{eq:x^2} we have $\mathrm{Tr}(A) -\smash{\sum_ix_i^2} \in \MM_2(\SAcpsd)$. Hence, using the result of Lemma~\ref{lemma:upperboundLw}, the feasible region of $\smash{\xib{cpsd}{t}(A)}$ is compact, and thus it has an optimal solution $L_t$. Again by Lemma~\ref{lemma:upperboundLw}, the sequence $\smash{(L_t)}$ has a pointwise converging subsequence with limit $L \in \R \langle\bx\rangle^*$. This pointwise limit $L$ is symmetric, tracial, satisfies $(L(x_ix_j)) = A$, and is nonnegative on $\smash{\mathcal M(\SAcpsd)}$. Hence $L$ is feasible for $\smash{\xib{cpsd}{\infty}(A)}$. This implies  that 
$L$ is optimal for $\smash{\xib{cpsd}{\infty}(A)}$ and we have $\lim_{t \to \infty} \smash{\xib{cpsd}{t}(A)} = \smash{\xib{cpsd}{\infty}(A)}$. 

The reformulation of $\xib{cpsd}{\infty}(A)$ in terms of $C^*$-algebras with a tracial state  follows directly using Theorem~\ref{propLinfinitedim}. \qed
\end{proof}

Next we give some equivalent reformulations for the parameter $\xib{cpsd}{*}(A)$, which follow as a direct application of Theorem \ref{propLfinitedim}. In general we do not know whether the infimum in $\smash{\xib{cpsd}{*}(A)}$ is attained. However, as a direct application of Corollary~\ref{thm:flat-nonc}, we see that this infimum is attained if there is an integer $t \in \N$ for which $\smash{\xib{cpsd}{t}(A)}$ admits a flat optimal solution.

\begin{proposition}\label{prop:cpsd*}
Let $A \in \CSP^n$. The parameter $\smash{\xib{cpsd}{*}(A)}$ is given by the infimum of $L(1)$ taken over all conic combinations $L$ of trace evaluations at elements in $\smash{\mathcal D_{\mathcal A}(\SAcpsd)}$ for which $A=(L(x_ix_j))$. The parameter $\xib{cpsd}{*}(A)$ is also equal to the infimum over all $\alpha \geq 0$ for which there exist a finite dimensional $C^*$-algebra $\MA$ with tracial state $\tau$ and $(X_1,\ldots,X_n) \in \smash{\mathcal D_{\mathcal A}(\SAcpsd)}$ such that $A = \alpha \cdot (\tau(X_iX_j))$.

In addition, if $\xib{cpsd}{t}(A)$ admits a flat optimal solution, then $\xib{cpsd}{t}(A)= \xib{cpsd}{*}(A)$.
\end{proposition}

Next we show a formulation for $\xib{cpsd}{*}(A)$ in terms of factorization by block-diagonal matrices, which helps explain why the inequality $\smash{\xib{cpsd}{*}(A)} \leq \smash{\cpsdr_\C(A)}$ is typically strict. Here $\| \cdot \|$ is the operator norm, so that $\|X\| = \lambda_\max(X)$ for $X \succeq 0$.
\begin{proposition}\label{lem:cpsd*}
For $A \in \CSP^n$ we have 
\begin{align}\label{eq:reformulationxi*}
\xib{cpsd}{*}(A) = \mathrm{inf} \Big\{ \sum_{m=1}^M d_m \cdot \underset{i \in [n]}{\max} \frac{\|X^m_{i}\|^2}{A_{ii}} : \;
& M \in \N,\, d_1,\ldots,d_M \in \N,\\[-0.7em]\nonumber
& X_i^m \in \Hermitian_+^{d_m} \text{ for } i \in [n], m \in [M],\\[-0.1em]\nonumber
& A = \mathrm{Gram}\big(\oplus_{m=1}^M X_1^m,\ldots,\oplus_{m = 1}^M X_n^m\big)\Big\}.
\end{align}
Note that using matrices from $\psdcone_+^{d_m}$ instead of $\Hermitian_+^{d_m}$ does not change the  optimal value.
\end{proposition}

\begin{proof}
The proof uses the formulation of $\smash{\xib{cpsd}{*}(A)}$ in terms of conic combinations of trace evaluations at matrix tuples in $\smash{\mathcal D(\SAcpsd)}$ as given in Proposition~\ref{prop:cpsd*}. We first show the inequality $\beta \leq \smash{\xib{cpsd}{*}(A)}$, where $\beta$ denotes the optimal value of the program in~\eqref{eq:reformulationxi*}.

For this, assume $L\in \R\ncx^*$ is a conic  combination of  trace evaluations at elements of $\smash{\MD(\SAcpsd)}$ such that $A=(L(x_ix_j))$. We will construct a feasible solution for~\eqref{eq:reformulationxi*} with objective value $L(1)$. The linear functional $L$ can be written as
\[
L=\sum_{m=1}^M \lambda_m L_{\mathbf Y^m}, \text{ where } \lambda_m > 0 \text{ and } \mathbf Y^m=(Y^m_1,\ldots,Y^m_n) \in \MD(\SAcpsd) \text{ for } m \in [M].
\]
Let $d_m$ denote the size of the matrices $Y_1^m, \ldots, Y_n^m$, so that $L(1)=\smash{\sum_m} \lambda_m d_m$. 
Since $\mathbf Y^m \in \smash{\MD(\SAcpsd)}$, we have $Y^m_i \succeq 0$ and $A_{ii}I-(Y^m_i)^2\succeq 0$ by identities~\eqref{eq:x^2} and~\eqref{eq:x}. This implies $\|Y^m_i\|^2 \leq A_{ii}$ for all $i\in [n]$ and $m \in [M]$.
Define $\mathbf X^m = \sqrt{\lambda_m} \, \mathbf Y^m$.
Then, $L(x_ix_j)= \sum_m \Tr(X^m_iX^m_j)$, so that the matrices $\smash{\oplus_m X^m_1,\ldots,\oplus_m X^m_n}$ form a Gram decomposition of $A$. 
This gives a feasible solution to~\eqref{eq:reformulationxi*} with value 
\[
\sum_{m=1}^M d_m \cdot \underset{i\in [n]}\max \frac{\|X^m_i\|^2}{A_{ii}} =\sum_{m=1}^M  d_m \lambda_m \, \underset{i\in [n]}\max \frac{\|Y^m_i\|^2}{A_{ii}} \le \sum_{m=1}^M d_m\lambda_m =L(1),
\]
which shows  $\beta \le L(1)$, and hence $\beta \leq \xib{cpsd}{*}(A)$.

For the other direction we assume 
\[
A = \mathrm{Gram}\big(\oplus_{m=1}^M X^m_1,\ldots,\oplus_{m=1}^M X^m_n\big), \quad X^m_1,\ldots,X^m_n \in \psdcone^{d_m}_+ \  \text{ for } m \in [M].
\]
Set $\lambda_m = \max_{i\in [n]} {\|X^m_i\|^2/ A_{ii}}$, and define the linear form $L$ by
\[
L= \sum_{m=1}^M \lambda_m L_{\mathbf Y^m}, \quad \text{where} \quad \mathbf Y^m = \bX^m / \sqrt{\lambda_m} \quad \text{for all} \quad m \in [M].
\]
We have $L(1)=\sum_m \lambda_m d_m$ and $A=(L(x_ix_j))$, and thus it suffices to show that each matrix tuple $\mathbf Y^m$ belongs to $\smash{\MD(\SAcpsd)}$. 
For this we observe that $\lambda_mA_{ii}\geq \|X^m_i\|^2$. Therefore $\lambda_m A_{ii} I \succeq (X_i^m)^2$, and thus $A_{ii} I \succeq (Y_i^m)^2$, which implies $\sqrt{A_{ii}} Y_i^m - (Y_i^m)^2 \succeq 0$. 
This shows $\smash{\xib{cpsd}{*}(A)} \le L(1)=\sum_m \lambda_m d_m$, and thus $\smash{\xib{cpsd}{*}(A)} \leq \beta$. \qed
\end{proof}

We can say a bit more when the matrix $A$ lies on an extreme ray of the cone $\smash{\CSP^n}$. In the formulation from Proposition~\ref{lem:cpsd*} it suffices to restrict the minimization over factorizations of $A$ involving only one block. However, we know very little about the extreme rays of $\CSP^n$, also in view of the recent 
result that the cone is not closed for large $n$~\cite{Slofstra17,DPP17}. 
\begin{proposition}\label{lem:flatextreme}
If $A$ lies on an extreme ray of the cone $\CSP^n$, then
\[
\xib{cpsd}{*}(A) = {\text{\rm inf}} \Big\{ d \cdot \underset{i \in [n]}{\max} \frac{\|X_{i}\|^2}{A_{ii}} : d \in \N, X_1,\ldots,X_n \in \Hermitian_+^{d}, \, A = \mathrm{Gram}\big(X_1, \ldots, X_n \big)\Big\}.
\]
Moreover, if $\oplus_{m=1}^M X^m_1,\ldots,\oplus_{m=1}^M X^m_n$ is a Gram decomposition of $A$ providing an optimal solution to~\eqref{eq:reformulationxi*} and some block $X^m_i$ has rank $1$, then $\xib{cpsd}{*}(A)=\hcpsd(A)$.
\end{proposition}
\begin{proof}
Let $\beta$ be the infimum in Proposition~\ref{lem:flatextreme}. The inequality $\smash{\xib{cpsd}{*}(A)} \leq \beta$ follows from the reformulation of $\smash{\xib{cpsd}{*}(A)}$ in Proposition~\ref{lem:cpsd*}. To show the reverse inequality we consider a solution
$
\oplus_{m=1}^M X^m_1,\ldots,\oplus_{m=1}^M X^m_n
$
to~\eqref{eq:reformulationxi*}, and 
 set $\lambda_m= \max_i\|X^m_i\|^2/A_{ii}$. We will show $\beta \le \sum_m d_m\lambda_m$. For this define the matrices
$
A_m=\Gram(X^m_1,\cdots,X^m_n),
$
so that $A=\sum_m A_m$.
As $A$ lies on an extreme ray of $\CSP^n$, we must have $A_m = \alpha_m A$ for some $\alpha_m>0$ with $\sum_m\alpha_m=1$.
Hence, since 
\[
A=A_m/\alpha_m=\Gram(X^m_1/\sqrt{\alpha_m}, \cdots, X^m_n/\sqrt{\alpha_m}),
\]
we have $\beta \le d_m\lambda_m/\alpha_m$ for all $m\in [M]$.
It suffices now to use  $\sum_m \alpha_m=1$ to see that 
$\min_m d_m\lambda_m/\alpha_m \leq \sum_m d_m\lambda_m$. So we have shown $\beta \le \min_m d_m\lambda_m/\alpha_m \le \sum_m d_m\lambda_m.$ This  
implies $\beta \le \xib{cpsd}{*}(A)$, and thus equality holds.

Assume now that $\oplus_{m=1}^M X^m_1,\ldots,\oplus_{m=1}^M X^m_n$ is optimal to~\eqref{eq:reformulationxi*} and that there is a block $X_i^m$ of rank $1$.  By Proposition~\ref{lem:cpsd*} we have $\sum_m d_m\lambda_m= \smash{\xib{cpsd}{*}(A)}$. From the argument just made above it follows that 
\[
\xib{cpsd}{*}(A)= \min_m d_m\lambda_m/\alpha_m =\sum_m d_m \lambda_m.
\]
As $\sum_m \alpha_m=1$ this implies $d_m\lambda_m/\alpha_m =\min_m d_m\lambda_m/\alpha_m$ for all $m$; that is, all terms $d_m\lambda_m/\alpha_m$ take the same value $\xib{cpsd}{*}(A)$.
By assumption there exist some  $m\in [M]$ and  $i\in [n]$ for which $X^m_i$ has rank $1$. Then $\|X^m_i\|^2=\langle X^m_i,X^m_i\rangle$, which gives $\lambda_m =\alpha_m$, and thus  $\smash{\xib{cpsd}{*}(A)} = d_m$. 
On the other hand, $\hcpsd(A)\le d_m$ since $\smash{(X^m_i/\sqrt{\alpha_m})_i}$ forms a Gram decomposition of $A$, so equality $\xib{cpsd}{*}(A)=d_m=\hcpsd(A)$ holds. \qed
\end{proof}

\subsection{Additional localizing constraints to improve on \texorpdfstring{$\smash{\xib{cpsd}{*}(A)}$}{xi star}}
\label{cpsd-additional-constraints}

In order to strengthen the bounds we may require nonnegativity over a (truncated) quadratic module generated by a larger set of localizing polynomials for $A$. The following lemma gives one such approach.
\begin{lemma}\label{lem:newloc}
Let $A \in \CSP^n$. For $v\in \R^n$ and $g_v= v^\T Av -\big(\sum_{i=1}^n v_ix_i\big)^2$, the set $\{g_v\}$ is localizing for every Gram factorization by Hermitian positive semidefinite matrices of $A$ (in particular, $\{g_v\}$ is localizing for $A$).
\end{lemma}
\begin{proof}
If $X_1,\ldots,X_n$ is a Gram decomposition of $A$ by Hermitian positive semidefinite matrices, then 
\[
v^\T Av= \Tr\big(\Big(\sum_{i=1}^n v_iX_i\Big)^2\big) \geq \lambda_{\max}\big(\Big(\sum_{i=1}^n v_iX_i\Big)^2\big),
\]
hence $v^\T AvI-(\sum_{i=1}^nv_iX_i)^2\succeq 0$. \qed
\end{proof}

Given a set $V\subseteq \R^n$, we consider the larger set 
\[
\SAVcpsd= \SAcpsd \cup \{g_v: v\in V\}
\]
of localizing polynomials for $A$. For $t \in \N \cup \{\infty,*\}$, denote by $\smash{\xib{cpsd}{t,V}(A)}$ the parameter obtained by replacing in $\smash{\xib{cpsd}{t}(A)}$ the nonnegativity constraint on $\MM_{2t}(\SAcpsd)$ by nonnegativity on the larger set $\smash{\MM_{2t}(\SAVcpsd)}$. We have $\xib{cpsd}{t,\emptyset}(A)=\xib{cpsd}{t}(A)$ and
\[
\xib{cpsd}{t}(A)\le \xib{cpsd}{t,V}(A)\le \hcpsd(A) \quad \text{for all} \quad V \subseteq \R^n.
\]

By scaling invariance, we can add the above constraints for all $v \in \R^n$ by setting $V$ to be the unit sphere $\oS^{n-1}$. Since $\oS^{n-1}$ is a compact metric space, there exists a sequence $V_1 \subseteq V_2 \subseteq \ldots \subseteq \oS^{n-1}$ of finite subsets such that $\bigcup_{k\ge 1} V_k$ is dense in $\oS^{n-1}$. 
Each of the parameters \smash{$\xib{cpsd}{t,V_k}(A)$} involves finitely many localizing constraints, and, as we now show, they converge to the parameter \smash{$\xib{cpsd}{t,\oS^{n-1}}(A)$}.

\begin{proposition}\label{remlocsphere}
Consider a matrix $A\in \CSP^n$. For $t \in \{\infty, *\}$, we have
\[
\lim_{k \to \infty} \xib{cpsd}{t,V_k}(A) = \xib{cpsd}{t,\oS^{n-1}}(A).
\]
\end{proposition}
\begin{proof}
Let $\epsilon > 0$. Since $\bigcup_k V_k$ is dense in $\oS^{n-1}$, there is an integer $k\ge 1$ so that for every $u \in \oS^{n-1}$ there exists a vector $v \in V_k$  satisfying
\begin{equation} \label{eq:vproperty}
\|u-v\|_1 \leq \frac{\varepsilon \lambda_\min(A)}{4 \sqrt{n} \, \max_i A_{ii}} \quad \text{and} \quad \|u-v\|_2 \leq \frac{\varepsilon \lambda_\min(A)}{4\mathrm{Tr}(A^2)^{1/2}}.
\end{equation}
The above Propositions~\ref{prop:lowerbound}~and~\ref{prop:cpsd*} have natural analogues for the programs \smash{$\xib{cpsd}{t,V}(A)$}. These show that  for $t = \infty$ ($t = *$) the parameter $\xib{cpsd}{t,V_k}(A)$ is the infimum over all $\alpha \geq 0$ for which there exist a (finite dimensional) unital $C^*$-algebra $\MA$ with tracial state $\tau$ and $\bX \in \DA(\SAVkcpsd)$ such that $A = \alpha \cdot (\tau(X_iX_j))$.

Below we will show that ${\bf X}' = \sqrt{1-\varepsilon} {\bf X} \in \DA(\SAScpsd)$. 
This implies that the linear form $L \in \R\langle \bx \rangle^*$ defined by $L(p) = \alpha/(1-\varepsilon) \tau(p({\bf X'}))$ is feasible for $\smash{\xib{cpsd}{t,\oS^{n-1}}(A)}$ with objective value $L(1) = \alpha/(1-\varepsilon)$. This shows 
\[
\xib{cpsd}{t,\oS^{n-1}}(A) \le {1\over 1-\epsilon}\ \xib{cpsd}{t,V_k} (A) \le {1\over 1-\epsilon}\ \lim_{k\to\infty} \xib{cpsd}{t,V_k}(A).
\]
Since $\varepsilon >0$ was arbitrary, letting $\epsilon$ tend to 0  completes the proof.

We now show ${\bf X}' = \sqrt{1-\varepsilon} {\bf X} \in \DA(\SAScpsd)$. For this consider the map
\[
f_{\bf X} \colon \oS^{n-1} \to \R, \, v \mapsto \Big\|\sum_{i=1}^n v_i X_i\Big\|^2,
\]
where $\| \cdot \|$ denotes the $C^*$-algebra norm of $\mathcal A$. For $\alpha\in \R$ and $a\in \MA$ with $a^*=a$, we have $\alpha \ge \|a\|$ if and only if $\alpha-a\succeq 0$  in $\MA$, or, equivalently,  $\alpha^2-a^2\succeq 0$ in $\MA$. Since $\bX \in \DA(\SAVkcpsd)$ we have $v^\T A v - f_{\bf X}(v) \geq 0$ for all $v \in V_k$, and hence
\[
v^\T A v - f_{{\bf X}'}(v) 
= v^\T A v \Big( 1 - (1-\varepsilon) \frac{f_{\bf X}(v)}{v^\T A v}\Big)
\geq v^\T A v \big( 1 - (1-\varepsilon) \big) = \varepsilon v^\T A v \geq \varepsilon \lambda_\min(A).
\]
Let $u \in \oS^{n-1}$ and let $v \in V_k$ be such that~\eqref{eq:vproperty} holds. Using Cauchy-Schwarz we have
\begin{align*}
| u^\T A u - v^\T A v | 
&= | (u-v)^\T A (u + v)| 
= |\langle A, (u-v) (u+v)^\T \rangle| \\
&\leq \sqrt{\mathrm{Tr}(A^2)} \sqrt{\mathrm{Tr}((u+v) (u-v)^\T (u-v) (u+v)^\T)}\\
&\leq \sqrt{\mathrm{Tr}(A^2)} \|u-v\|_2 \|u+v\|_2 \leq 2\sqrt{\mathrm{Tr}(A^2)} \|u-v\|_2\\
&\leq 2\sqrt{\mathrm{Tr}(A^2)} \frac{\varepsilon \lambda_\min(A)}{4\sqrt{\mathrm{Tr}(A^2)}}=  \frac{\varepsilon \lambda_\min(A)}{2}.
\end{align*}
Since $\sqrt{A_{ii}} X_i - X_i^2$ is positive in $\mathcal A$, we have that $\sqrt{A_{ii}} -X_i$ is positive in $\mathcal A$ by~\eqref{eq:cx} and~\eqref{eq:x^2}, which implies $\|X_i\| \leq \sqrt{A_{ii}}$. By the reverse triangle inequality we then have
\begin{align*}
|f_{\bf X'}(u) - f_{\bf X'}(v)| &= \Big| \big\|\sum_{i=1}^n u_i X_i'\big\| - \big\|\sum_{i=1}^n v_i X_i'\big\|\Big| \Big( \big\|\sum_{i=1}^n u_i X_i'\big\| + \big\|\sum_{i=1}^n v_i X_i'\big\|\Big)\\
&\leq  \big\|\sum_{i=1}^n (v_i - u_i) X_i'\big\| 2\sqrt{n} \, \max_i \sqrt{A_{ii}} \\
&\leq \Big(\sum_{i=1}^n |v_i - u_i| \|X_i'\|\Big) 2\sqrt{n} \, \max_i \sqrt{A_{ii}}\\
&\leq \|u-v\|_1 2 \sqrt{n} \, \max_i A_{ii} \leq  \frac{\varepsilon \lambda_\min(A)}{4 \sqrt{n} \, \max_i A_{ii}} 2\sqrt{n}\, \max_i A_{ii}= \frac{\varepsilon \lambda_\min(A)}{2}.
\end{align*}
Combining the above inequalities we obtain  that $u^\T A u - f_{{\bf X}'}(u) \geq 0$  for all $\oS^{n-1}$, and hence $u^\T A u - \big(\sum_{i=1}^n u_i X_i'\big)^2$ is positive in $\mathcal A$. Thus we have $\bX' \in \DA(\SAScpsd)$. \qed
\end{proof}

We now discuss two examples where the bounds $\xib{cpsd}{*,V}(A)$ go beyond $\xib{cpsd}{*}(A)$.

\begin{example}\label{exA12}
Consider the matrix 
\begin{equation} \label{eqdefA}
A = \begin{pmatrix}
1 &1/2\\
1/2 & 1
\end{pmatrix}=
 \Gram\Big(
 \begin{pmatrix} 1 & 0 \\ 0 & 0 \end{pmatrix},
 \begin{pmatrix} 1/2 & 1/2 \\ 1/2 & 1/2\end{pmatrix} \Big),
\end{equation}
with $\cpsdr_\C(A) = 2$.
We can also write $A = \mathrm{Gram}(Y_1, Y_2)$, where
\[
Y_1 = \frac{1}{\sqrt{2}}\begin{pmatrix} 1 & 0 & 0\\ 0 & 1 & 0\\ 0 & 0 &0 \end{pmatrix}, 
\quad Y_2 = \frac{1}{\sqrt{2}} \begin{pmatrix} 1 & 0 & 0\\
0 & 0 & 0\\
0 & 0 & 1
\end{pmatrix}.
\]
With $X_i= \sqrt 2 \  Y_i$ we have $I - X_i^2 \succeq 0$ for $i=1,2$. Hence the linear form $L = L_\bX /2$ is feasible for $\smash{\xib{cpsd}{*}}(A)$, which shows that $\smash{\xib{cpsd}{*}(A)} \leq L(1) = 3/2$. In fact, this form $L$ gives an optimal flat solution to $\smash{\xib{cpsd}{2}(A)}$, as we can check using a semidefinite programming solver, so $\smash{\xib{cpsd}{*}}(A) = 3/2$. In passing, we observe that $\smash{\xib{cpsd}{1}}(A) = 4/3$, which coincides with the analytic lower bound~\eqref{eq:lbcpsdanalytic}
(see also Lemma~\ref{lem:lb} below).

For $e = (1,1) \in \R^2$ and $V = \{e\}$, this form $L$ is not feasible for $\smash{\xib{cpsd}{*,V}(A)}$, because for the polynomial $p = 1-3 x_1 - 3x_2$ we have $L(p^*g_ep) = -9/2 < 0$. This means that the localizing constraint 
$\smash{L(p^*g_ep)\ge 0}$ is not redundant: For $t\ge 2$ it cuts off part of the feasibility region of $\smash{\xib{cpsd}{t}(A)}$. Indeed, using a semidefinite programming solver we find an optimal flat solution of $\smash{\xib{cpsd}{3,V}(A)}$ with objective value $(5-\sqrt{3})/2\approx 1.633$, hence 
\[
\xib{cpsd}{*,V}(A) = (5-\sqrt{3})/2 > 3/2 = \xib{cpsd}{*}(A). 
\]
\end{example}

\begin{example} \label{ex:circulantcpsd}
Consider the symmetric circulant matrices
\[
M(\alpha) = 
\begin{pmatrix} 
	1 & \alpha & 0 & 0 & \alpha \\ 
	\alpha & 1 & \alpha & 0 & 0 \\ 
	0 & \alpha & 1 & \alpha & 0 \\ 
	0 & 0 & \alpha  & 1 & \alpha \\ 
	\alpha & 0 & 0 & \alpha & 1
\end{pmatrix}\quad \text{ for } \quad \alpha\in\R.
\]
For $0\le \alpha \le 1/2$ we have $M(\alpha) \in \CSP^5$ with $\hcpsd(M(\alpha)) \leq 5$. To see this we set $\beta=(1+\sqrt{1-4\alpha^2})/2$ and observe that the matrices 
\[
X_i = \mathrm{Diag}(\sqrt{\beta} \, e_i + \sqrt{1-\beta}\, e_{i+1})  \in \psdcone^5_+, \quad i\in [5], \quad  (\text{with }e_6 := e_1),
\]
form a factorization of $M(\alpha)$. As $M(\alpha)$ is supported by a cycle, we have $M(\alpha)\in \CSP^5$ if and only if $M(\alpha)\in \CP^5$~\cite{LP15}. Thus $M(\alpha) \in \CSP^5$ if and only if $0 \leq \alpha \leq 1/2$. 

By using its formulation in Proposition~\ref{lem:cpsd*}, we can use the above factorization to derive the inequality $\xib{cpsd}{*}(M(1/2))\le 5/2$. However, using a semidefinite programming solver we see that  
\[
\xib{cpsd}{2,V}(M(1/2)) = 5,
\]
where $V$ is the set containing the vector $(1,-1,1,-1,1)$ and its cyclic shifts. Hence the bound $\smash{\xib{cpsd}{2,V}(M(1/2))}$ is tight: It certifies $\cpsdr_\C(M(1/2))=5$, while
the other known bounds, the rank bound $\smash{\sqrt{\mathrm{rank}(A)}}$ and the analytic bound~\eqref{eq:lbcpsdanalytic}, only give $\cpsdr_\C(A) \geq 3$. 

We now observe that there exist $0<\epsilon,\delta<1/2$ such that $\cpsdrank_\C(M(\alpha)) = 5$ for all $ \alpha \in [0,\epsilon] \cup [\delta,1/2]$. 
Indeed, this follows from  the fact that $\xib{cpsd}{1}(M(0)) = 5$ (by Lemma~\ref{lem:lb}), the above result that $\smash{\xib{cpsd}{2,V}(M(1/2))} = 5$,  and the lower semicontinuity of $\alpha \mapsto \smash{\xib{cpsd}{2,V}(M(\alpha))}$, which is shown in Lemma~\ref{prop:lowersemi} below.

As the matrices $M(\alpha)$ are nonsingular, the above factorization shows that their cp-rank is equal to 5 for all $\alpha \in [0,1/2]$; whether they all have $\cpsdrank$ equal to 5 is not known.
\end{example}

\subsection{Boosting the bounds}\label{secstronger}

In this section we propose some 
additional constraints that can be added to strengthen the bounds $\smash{\xib{cpsd}{t,V}(A)}$ for finite $t$. These constraints may shrink the feasibility region of $\smash{\xib{cpsd}{t,V}(A)}$ for $t \in \N$, but they are redundant for $t\in \{\infty,*\}$. The latter is shown using the reformulation of the parameters $\smash{\xib{cpsd}{\infty,V}(A)}$ and $\smash{\xib{cpsd}{*,V}(A)}$ in terms of $C^*$-algebras. 

We first mention how to construct  localizing constraints of ``bilinear type'', inspired by the work of Berta, Fawzi and Scholz~\cite{BFS16}. Note that as for localizing constraints, these bilinear constraints can be modeled as semidefinite constraints.
\begin{lemma}\label{lem:bilinear}
Let $A\in \CSP^n$, $t \in \N \cup \{\infty, *\}$, and let $\{g,g'\}$ be localizing for $A$. 
If we add the constraints
\begin{equation}\label{eq:bilinear}
L(p^*gpg')\ge 0 \quad \text{for} \quad p\in \R\ncx \quad \text{with} \quad \deg(p^*gpg')\le 2t
\end{equation}
to $\xib{cpsd}{t,V}(A)$, then we still get a lower bound on $\hcpsd(A)$. However, the constraints~\eqref{eq:bilinear} are redundant for $\smash{\xib{cpsd}{\infty,V}(A)}$ and $\smash{\xib{cpsd}{*,V}(A)}$ when $g,g' \in \smash{\MM(\SAVcpsd)}$.
\end{lemma}

\begin{proof}
Let $\bX\in  (\Hermitian^d_+)^n$ be a Gram decomposition of $A$, and let $L =L_\bX$ be the real part of the trace evaluation at $\bX$.
Then,  $p(\bX)^* g(\bX) p(\bX)\succeq 0$ and $g'(\bX)\succeq 0$, and thus
\[
L(p^*gpg') =\text{Re}( \Tr( p(\bX)^* g(\bX) p(\bX) g'(\bX)))\ge 0.
\]
So by adding the constraints~\eqref{eq:bilinear} we still get a lower bound on $\hcpsd(A)$.

To show that the constraints~\eqref{eq:bilinear} are redundant for  $\xib{cpsd}{t,V}(A)$ and  $\smash{\xib{cpsd}{*,V}(A)}$ when $g,g'\in \smash{\MM(\SAVcpsd)}$, we let $t\in\{\infty,*\}$ and assume $L$ is feasible for $\smash{\xib{cpsd}{t,V}(A)}$. By Theorem~\ref{propLinfinitedim} there exist a unital $C^*$-algebra $\MA$ with tracial state $\tau$ and $\bX\in \smash{\mathcal D(\SAVcpsd)}$ such that $L(p)=L(1) \tau(p(\bX))$ for all $p\in\R\ncx$. Since $g,g' \in \MM(\SAVcpsd)$ we know that 
$g(\bX), g'(\bX)$ are positive elements in $\MA$, so $g(\bX) = a^* a$ and $g'(\bX) = b^* b$ for some $a,b \in \MA$. Then we have 
\begin{align*}
L(p^* g pg) &= L(1) \, \tau(p^*(\bX) \, g(\bX) \, p(\bX) \, g'(\bX) ) \\
&= L(1) \, \tau(p^*(\bX) \, a^* a \, p(\bX) \, b^* b) \\
&=  L(1) \, \tau((a \,  p(\bX) \, b^*)^* a \, p(\bX) \, b^*) \geq 0,
\end{align*}
where we use that $\tau$ is a positive tracial state on $\MA$. \qed
\end{proof}

Second, we show how to use zero entries in $A$ and vectors in the kernel of $A$ to enforce new constraints on $\xib{cpsd}{t,V}(A)$. 

\begin{lemma}\label{lem:kernel}
Let $A\in \CSP^n$ and $t \in \N \cup \{\infty, *\}$. If we add the constraint
\begin{equation} \label{eq:kernel}
L=0  \quad \text{ on } \quad \mathcal I_{2t}\big(\big\{\sum_{i=1}^nv_ix_i: v\in \ker A\big\} \cup \big\{x_ix_j: A_{ij}=0 \big\} \big)
\end{equation}
to $\xib{cpsd}{t,V}(A)$, then we still get a lower bound on $\hcpsd(A)$. Moreover, these constraints are redundant for $\smash{\xib{cpsd}{\infty,V}(A)}$ and $\smash{\xib{cpsd}{*,V}(A)}$.
\end{lemma}

\begin{proof}
Let $\bX\in  (\Hermitian^d_+)^n$ be a Gram factorization of $A$ and let $L_\bX$ be as in~\eqref{eq:lx}. If $Av=0$, then $0=v^\T Av = \Tr((\sum_{i=1}^n v_iX_i)^2)$ and thus  $\sum_{i=1}^nv_iX_i=0$. Hence  $L_\bX((\sum_{I=1}^nv_ix_i)p)=\mathrm{Re}(\Tr((\sum_{i=1}^nv_iX_i)p(\bX)))=0$.
If $A_{ij}=0$, then $\Tr(X_iX_j)=0$, which implies $X_iX_j=0$, since $X_i$ and $X_j$ are positive semidefinite. Hence
$L_\bX(x_ix_ip)=\text{Re}(\Tr(X_iX_jp(\bX)))=0$. Therefore, adding the constraints~\eqref{eq:kernel} still lower bounds $\hcpsd(A)$.

As in the proof of the previous lemma, if $t \in \{\infty,*\}$  and $L$ is feasible for $\smash{\xib{cpsd}{t,V}}(A)$ then, by Theorem~\ref{propLinfinitedim}, there exist a unital $C^*$-algebra $\MA$ with tracial state $\tau$ and $\bX$ in $\smash{\mathcal D(\SAVcpsd)}$ such that $L(p)=L(1) \tau(p(\bX))$ for all $p\in\R\ncx$. Moreover, by Lemma~\ref{remark:von Neumann} we may assume $\tau$ to be faithful. For a vector $v$ in the kernel of $A$ we have $0 = v^\T A v  = L((\sum_i v_i x_i)^2) = L(1) \tau( (\sum_i v_i X_i)^2)$, and hence, since $\tau$ is faithful, $\sum_i v_i X_i = 0$ in $\MA$. It follows that $L(p (\sum_i v_i x_i)) = L(1) \tau(p(\bX) \, 0) = 0$ for all $p \in \R\ncx$.
Analogously, if $A_{ij}=0$, then $L(x_ix_j)=0$ implies $\tau(X_iX_j)=0$ and thus $X_iX_j=0$, since $X_i, X_j$ are positive in $\MA$ and $\tau$ is faithful. This implies $L(p x_i x_j) = 0$ for all $p \in \R\ncx$. This shows that the constraints~\eqref{eq:kernel} are redundant. \qed
\end{proof}
Note that the constraints $L(p \, (\sum_{i=1}^nv_ix_i))=0$  for $p\in \R\ncx_t,$ which are implied by~\eqref{eq:kernel}, are in fact redundant: if $v \in \ker(A)$, then the vector obtained by extending $v$ with zeros belongs to $\ker(M_t(L))$, since $M_t(L)\succeq 0$. Also, for an implementation of $\xib{cpsd}{t}(A)$ with the additional constraints~\eqref{eq:kernel}, it is more efficient to index the moment matrices with a basis for $\R\ncx_{t}$ modulo the ideal $\mathcal I_t\big(\{ \sum_i v_i x_i: v \in \ker(A)\} \cup \{x_i x_j : A_{ij} = 0\}\big)$.

\subsection{Additional properties of the bounds} \label{sec:addproperties}

Here we list some additional properties of the parameters $\xib{cpsd}{t}(A)$ for $t \in \N \cup \{\infty, *\}$. First we state some properties for which the proofs are immediate and thus omitted.
\begin{lemma}\label{lem:properties}
Suppose $A\in\CSP^n$  and $t \in \N \cup \{\infty,*\}$.
\begin{enumerate}
\item[(1)] If $P$ is a permutation matrix, then $\xib{cpsd}{t}(A) = \xib{cpsd}{t}(P^\T A P)$. 
\item[(2)]  If $B$ is a principal submatrix of $A$, then $\xib{cpsd}{t}(B) \leq \xib{cpsd}{t}(A)$. 
\item[(3)]  If $D$ is a positive definite diagonal matrix, then $\xib{cpsd}{t}(A) = \xib{cpsd}{t}(D A D).$
\end{enumerate}
\end{lemma}
We also have the following direct sum property, where the equality follows using the $C^*$-algebra reformulations as given in Proposition~\ref{prop:lowerbound} and Proposition~\ref{prop:cpsd*}.

\begin{lemma}\label{lemoplus}
If $A \in \CSP^n$ and $B \in \CSP^m$, then $\xib{cpsd}{t}(A\oplus B) \leq \xib{cpsd}{t}(A) + \xib{cpsd}{t}(B)$, where equality holds for $t \in \{\infty, *\}$.
\end{lemma}

\begin{proof}
To prove the inequality  we take $L_A$ and $L_B$ feasible for $\smash{\xib{cpsd}{t}}(A)$ and $\smash{\xib{cpsd}{t}}(B)$, and construct a feasible $L$ for $\smash{\xib{cpsd}{t}}(A\oplus B)$ by $L(p({\bf x}, {\bf y})) = L_A(p({\bf x}, {\bf 0})) + L_B(p({\bf 0}, {\bf y}))$. 

Now we show equality  for $t = \infty$ ($t=*$). By Proposition~\ref{prop:lowerbound} (Proposition~\ref{prop:cpsd*}), $\smash{\xib{cpsd}{t}(A\oplus B)}$ is equal to the infimum over all $\alpha \geq 0$ for which there exists a (finite dimensional) unital $C^*$-algebra $\MA$ with tracial state $\tau$ and $(\bX, {\bf Y}) \in \MD_\MA(\SABcpsd)$ such that $A = \alpha \cdot (\tau(X_iX_j))$, $B = \alpha \cdot (\tau(Y_iY_j))$ and $(\tau(X_iY_j))=0$. This implies $\bX\in \smash{\MD_\MA(\SAcpsd)}$ and  $\mathbf Y\in \smash{\MD_\MA(\SBcpsd)}$.
Let $P_A$ be the projection onto the space $\sum_i \mathrm{Im}(X_i)$ and define the linear form  $L_A \in \R\langle \bx \rangle^*$ by $L_A(p) = \alpha \cdot \tau(p(\bX) P_A)$. It follows that $L_A$ is is nonnegative on $\mathcal M(\SAcpsd)$, and 
\[
L_A(x_ix_j) = \alpha \, \tau(x_ix_jP_A) = \alpha \, \tau(x_ix_j) = A_{ij},
\]
so $L_A$ is feasible for $\smash{\xib{cpsd}{\infty}(A)}$ with $L_A(1)=\alpha \tau(P_A)$.
 In the same way we consider the projection $P_B$ onto the space 
$\sum_j  \mathrm{Im}(Y_j)$ and  define a feasible solution $L_B$ for  $\smash{\xib{cpsd}{t}(B)}$ with $L_B(1)=\alpha\tau(P_B)$.  By Lemma~\ref{remark:von Neumann} we may assume $\tau$ to be faithful, so that positivity of $X_i$ and $Y_j$  together with $\tau(X_iY_j) = 0$ implies $X_iY_j = 0$ for all $i$ and $j$, and thus $\sum_i \mathrm{Im}(X_i) \perp \sum_j \mathrm{Im}(Y_j)$. This implies $I \succeq P_A + P_B$ and thus $\tau(P_A+P_B)\le \tau(1)=1$. We have
\[
L_A(1) + L_B(1)  = \alpha \, \tau(P_A) + \alpha\tau(P_B) \leq \alpha \, \tau(1) = \alpha,
\]
so $\xib{cpsd}{t}(A)+\xib{cpsd}{t}(B) \le L_A(1)+L_B(1)\le \alpha$, completing  the proof. \qed
\end{proof}
Note that the $\cpsdrank$ of a matrix satisfies the same properties as those mentioned in the above two 
lemmas, where the inequality in Lemma~\ref{lemoplus} is always an equality: $\cpsd_\C(A~\oplus~B)=\cpsd_\C(A)+\cpsd_\C(B)$~\cite{PSVW16,GdLL17}.

\medskip
The following lemma shows that the first level of our hierarchy is at least as good as the analytic lower bound~\eqref{eq:lbcpsdanalytic} on the cpsd-rank derived in~\cite[Theorem 10]{PSVW16}.
\begin{lemma} \label{lem:lb}
For any non-zero matrix $A \in \CSP^n$ we have
\begin{equation} \label{eq:lbcpsdanalytic}
\xib{cpsd}{1}(A) \geq \frac{\left( \sum_{i=1}^n \sqrt{A_{ii}}\right)^2}{\sum_{i,j=1}^n A_{ij}}.
\end{equation}
\end{lemma}
\begin{proof}
Let $L$ be feasible for $\smash{\xib{cpsd}{1}(A)}$. Since $L$ is nonnegative on $\MM_{2}(\SAcpsd)$, 
it follows that  $L(\sqrt{A_{ii}}x_i-x_i^2)\geq 0$, implying 
$\sqrt{A_{ii}} L(x_i)\ge L(x_i^2)=A_{ii}$ and thus $L(x_i)\ge \sqrt{A_{ii}}$.
Moreover, the matrix $M_1(L)$ is positive semidefinite. By taking the Schur complement with respect to its upper left corner (indexed by $1$) it follows that the matrix $L(1)\cdot A- (L(x_i)L(x_j))$ is positive semidefinite. Hence the sum of its entries is nonnegative, which gives
$L(1)(\sum_{i,j}A_{ij})\ge (\sum_i L(x_i))^2\ge (\sum_i \sqrt{A_{ii}})^2$ and shows the desired inequality. \qed
\end{proof}
As an application of Lemma~\ref{lem:lb}, the first bound $\xib{cpsd}{1}$ is exact for the $k\times k$ identity matrix: $\smash{\xib{cpsd}{1}(I_k)}=\hcpsd(I_k)=k$.
Moreover, by combining this with Lemma~\ref{lem:properties}, it follows that $\smash{\xib{cpsd}{1}(A)}~\ge~k$ if $A$ contains a diagonal positive definite $k\times k$ principal submatrix. A slightly more involved example is given by the $5 \times 5$ circulant matrix $A$ whose entries are given by $A_{ij} = \cos((i-j)4\pi/5)^2$ ($i,j \in [5]$); this matrix was used in~\cite{FGPRT15} to show a separation between the completely positive semidefinite cone and the completely positive cone, and it was shown that $\cpsdrank_\C(A) =2$. The analytic lower bound  of~\cite{PSVW16} also evaluates to $2$, hence Lemma~\ref{lem:lb} shows that our bound is tight on this example. 

\medskip
We now examine further analytic properties of the parameters $\smash{\xib{cpsd}{t}}(\cdot)$. 
For each $r \in \N$, the set of matrices $A\in\CSP^n$ with $\hcpsd(A) \le r$ is closed, which shows that the function $A \mapsto \cpsd_\C(A)$ is lower semicontinuous. We now show that the functions $A \mapsto \smash{\xib{cpsd}{t}(A)}$ have the same property. The other bounds defined in this paper are also lower semicontinuous, with a similar proof.
\begin{lemma} \label{prop:lowersemi}
For every $t \in \N \cup \{\infty\}$ and $V \subseteq \R^n$, the function 
\[
\psdcone^n \to \R \cup \{\infty\}, \, A \mapsto \xib{cpsd}{t,V}(A)
\]
is lower semicontinuous.
\end{lemma}
\begin{proof} 
It suffices to show the result for $t\in \N$, because $\xib{cpsd}{\infty,V}(A)=\sup_t\, \xib{cpsd}{t,V}(A)$, and the pointwise supremum of lower semicontinuous functions is lower semicontinuous. We show that the level sets $\{A \in \psdcone^n: \xib{cpsd}{t,V}(A) \leq r\}$ are closed. For this we consider a sequence $(A_k)_{k\in\N}$ in $\psdcone^n$ converging to $A \in \psdcone^n$ such that \smash{$\xib{cpsd}{t,V}(A_k) \leq r$} for all $k$. We show that \smash{$\xib{cpsd}{t,V}(A) \leq r$}. Let $L_k\in \R\ncx_{2t}^*$ be an optimal solution to $\xib{cpsd}{t,V}(A_k)$. As $L_k(1) \le r$ for all $k$, it follows from Lemma~\ref{lemma:upperboundLw} that there is a pointwise converging subsequence of $(L_k)_k$, still denoted $(L_k)_k$ for simplicity, that has a limit $L\in \R\ncx_{2t}^*$ with $L(1)\leq r$. To complete the proof we show that $L$ is feasible for $\xib{cpsd}{t,V}(A)$. By the pointwise convergence of $L_k$ to $L$, for every $\epsilon >0$, $p \in \R\ncx$, and $i \in [n]$, there exists a $K \in \N$ such that for all $k \geq K$ we have 
\begin{align*}
|L(p^* x_i p) - L_k(p^* x_i p) | &< \min \{1,\frac{\epsilon}{\sqrt{A_{ii}}}\}, \qquad  |L(p^* x_i^2 p) - L_k(p^* x_i^2 p)| < \epsilon, \\
 |\sqrt{A_{ii}} - \sqrt{(A_k)_{ii}}| &< \frac{\epsilon}{L(p^* x_i p) + 1}.
\end{align*}
Hence we have 
\begin{align*}
L(p^*(\sqrt{A_{ii}} x_i - x_i^2) p) &= \sqrt{A_{ii}} \Big(L(p^* x_i p) - L_k(p^* x_i p) + L_k (p^* x_i p) \Big) \\
&\qquad - \Big( L(p^* x_i^2 p) -L_k(p^* x_i^2 p ) + L_k(p^* x_i^2p)\Big) \\
&\geq -2 \epsilon  + \sqrt{A_{ii}} \, L_k (p^* x_i p)  - L_k(p^* x_i^2p) \\
&\geq -3 \epsilon  + \sqrt{(A_k)_{ii}} \, L_k (p^* x_i p)  - L_k(p^* x_i^2p) \\
&= -3 \epsilon + L_k(p^*(\sqrt{(A_k)_{ii}} \, x_i - x_i^2) p)  \geq -3 \epsilon,
\end{align*}
where in the second inequality we use that $0 \leq L_k(p^* x_i p) \leq L(p^* x_i p) + 1$. Letting $\epsilon\rightarrow 0$ gives $L(p^*(\sqrt{A_{ii}}x_i-x_i^2)p)\ge 0$.

Similarly one can show  $L(p^*(v^\T Av - (\sum_i v_i x_i)^2) p) \geq 0$ for $v \in V$, $p \in \R\ncx$. \qed
\end{proof}
If we restrict to completely positive semidefinite matrices with an all-ones diagonal, that is, to $\CSP^n \cap \elliptope_n$,  we can show an even stronger property. Here $\elliptope_n$ is the \emph{elliptope}, which is the set of $n \times n$ positive semidefinite matrices with an all-ones diagonal.
\begin{lemma}\label{lemconvex}
For every $t \in \N \cup \{\infty\}$, the function 
\[
\CSP^n \cap \elliptope_n \rightarrow \R,\, A \mapsto \xib{cpsd}{t}(A)
\]
is convex, and hence continuous on the interior of its domain.
\end{lemma}
\begin{proof}
Let $A,B\in\CSP^n\cap \elliptope_n$ and $0<\lambda<1$. Let $L_A$ and $L_B$ be optimal solutions for $\smash{\xib{cpsd}{t}}(A)$ and $\smash{\xib{cpsd}{t}}(B)$. Since the diagonals of $A$ and $B$ are the same, we have $\SAcpsd=S_B^{\hspace{0.1em}\mathrm{cpsd}}$. So $L=\lambda L_A+(1-\lambda)L_B$ is feasible for $\smash{\xib{cpsd}{t}}(\lambda A+(1-\lambda )B)$, hence 
$
{\xib{cpsd}{t}(\lambda A+(1-\lambda)B)\le 
\lambda L_A(1)+(1-\lambda)L_B(1) = \lambda \xib{cpsd}{t}(A)+} (1-\lambda)\xib{cpsd}{t}(B).
$ \qed
\end{proof}

\begin{example} \label{lem:notcont}
In this example we show that for $t \geq 1$, the function 
\[
\CSP^n \to \R, \, A \mapsto \xib{cpsd}{t}(A)
\]
is not continuous.
For this we consider  the matrices 
\[
A_k = \begin{pmatrix} 1/k & 0 \\ 0 & 1 \end{pmatrix}\in \CSP^2,
\]
with $\cpsdr_\C(A_k) = 2$ for all $k\geq 1$.
As $A_k$ is diagonal positive definite we have  $\xib{cpsd}{t}(A_k) = 2$ for all $t,k\ge 1$, while  $\xib{cpsd}{t}(\lim_{k \rightarrow \infty} A_k) = 1$. This argument extends to $\CSP^n$ with $n > 2$. This example also shows that the first level of the hierarchy $\smash{\xib{cpsd}{1}(\cdot)}$ can be strictly better than the analytic lower bound~\eqref{eq:lbcpsdanalytic} of~\cite{PSVW16}.
\end{example} 

\begin{example} \label{eqxialpha}
In this example 
we determine $\xib{cpsd}{t}(A)$ for all $t \geq 1$ and $A \in \CSP^2$. In view of Lemma~\ref{lem:properties}(3) we only need to find $\smash{\xib{cpsd}{t}(A(\alpha))}$ for $0 \leq \alpha \leq 1$, where
$
A(\alpha)= \bigl(\begin{smallmatrix} 1 & \alpha \\ \alpha & 1\end{smallmatrix}\bigr).
$

The first bound $\smash{\xib{cpsd}{1}}(A(\alpha))$ is equal to the analytic bound $2/(\alpha+1)$ from~\eqref{eq:lbcpsdanalytic}, where the equality follows from the fact that $L$ given by $L(x_i x_j) = A(\alpha)_{ij}$, $L(x_1)=L(x_2)=1$ and $L(1)=2/(\alpha+1)$ is feasible for $\smash{\xib{cpsd}{1}(A(\alpha))}$.

For $t \geq 2$ we show $\xib{cpsd}{t}(A(\alpha)) = 2-\alpha$. By the above this is true for $\alpha = 0$ and $\alpha = 1$, and in Example~\ref{exA12} we show $\smash{\xib{cpsd}{t}(A(1/2))} =3/2$ for $t\ge 2$. The claim then follows since the function $\alpha \mapsto \smash{\xib{cpsd}{t}(A(\alpha))}$ is convex by Lemma~\ref{lemconvex}.
\end{example}

\section{Lower bounds on the completely positive rank} \label{sec:lowercp}

The best current approach for lower bounding the completely positive rank of a matrix is due to Fawzi and Parrilo~\cite{FP16}. Their approach relies on the atomicity of the completely positive rank, that is, the fact that  $\cprank(A)=r$ if and only if $A$ has an atomic decomposition $A=\sum_{k=1}^r v_k v_k^\T$ for nonnegative vectors $v_k$.  In other words, if $\cprank(A)=r$, then $A/r$  can be written as a convex combination of $r$ rank one positive semidefinite matrices $v_k v_k^\T$ that satisfy $0 \leq v_k v_k^\T \leq A$ and $v_k v_k^\T \preceq A$. 
Based on this observation Fawzi and Parrilo define the parameter 
\[
\tau_\mathrm{cp}(A) = \min \Big\{ \alpha  : \alpha \geq 0,\, A \in \alpha \cdot \mathrm{conv} \big\{ R \in \psdcone^n : 0 \leq R \leq A, \,R \preceq A,\, \rank(R) \leq 1\big\}\Big\},
\]
as lower bound for $\cprank(A)$. They also define the semidefinite programming parameter
\begin{align*}
\tau_{\mathrm{cp}}^{\mathrm{sos}}(A) = \mathrm{min} \big\{ \alpha : \; &   \alpha \in \R, \, X \in \psdcone^{n^2},\\[-0.2em]
&\hspace{-0.3em}\begin{pmatrix} \alpha & \text{vec}(A)^\T  \\ \text{vec}(A) & X \end{pmatrix} \succeq 0,\\
& X_{(i,j),(i,j)} \leq A_{ij}^2 \quad \text{for} \quad 1 \leq i,j \leq n, \\
& X_{(i,j),(k,l)} = X_{(i,l),(k,j)} \quad \text{for} \quad 1 \leq i < k \leq n, \; 1 \leq j < l \leq n,\\
& X \preceq A \otimes A\big\},
\end{align*}
as an efficiently computable relaxation of $\tau_\mathrm{cp}(A)$, and they show $\rank(A) \leq \taucpsos(A)$.  Therefore we have 
\[
\rank (A) \le \tau_{\mathrm{cp}}^{\mathrm{sos}}(A) \le \tau_\mathrm{cp}(A)\le \cprank(A).
\]

Instead of the atomic point of view, here we take the matrix factorization perspective, which allows us to obtain bounds by adapting the techniques from Section~\ref{sec:lower bounds on cpsd rank} to the commutative setting. Indeed, we may view a factorization $A =(a_i^{\sf T}a_j)$ 
by nonnegative vectors  as a factorization by diagonal (and thus pairwise commuting) positive semidefinite matrices.

Before presenting the details of our hierarchy of lower bounds, we mention some of our results in order to make the link to the parameters
$\tau_{\mathrm{cp}}^{\mathrm{sos}}(A)$ and $ \tau_\mathrm{cp}(A)$.
The direct analogue of $\{\smash{\xib{cpsd}{t}(A)}\}$ in the commutative setting leads to a hierarchy that does not converge to $\tau_{\mathrm{cp}}(A)$, but we provide two approaches to strengthen it that do converge to $\tau_{\mathrm{cp}}(A)$. The first approach is based on a generalization of the tensor constraints in $\tau_{\mathrm{cp}}^{\mathrm{sos}}(A)$. 
We also provide a computationally more efficient version of these tensor constraints, leading to a hierarchy whose second level is at least as good as $\smash{\tau_{\mathrm{cp}}^\mathrm{sos}(A)}$ while being defined by a smaller semidefinite program. The second approach relies on adding localizing constraints for vectors in the unit sphere as in Section~\ref{cpsd-additional-constraints}.

The following hierarchy is a commutative analogue of  the hierarchy from Section~\ref{sec:lower bounds on cpsd rank}, where we  may now add  the localizing polynomials $A_{ij}-x_ix_j$ for the pairs  $1 \leq i < j \leq n$, which was not possible in the noncommutative setting of the completely positive semidefinite rank. 
For each $t \in \N \cup \{\infty\}$ we consider the semidefinite program
 \begin{align*}
\xib{cp}{t}(A) = \mathrm{min} \big\{ L(1) : \; & L \in \R[x_1,\ldots,x_n]_{2t}^*,\\
&L(x_ix_j) = A_{ij} \quad \text{for} \quad i,j \in [n],\\
&L \geq 0 \quad \text{on} \quad \mathcal M_{2t}(\SAcp) \big\},
\end{align*}
where we set
\[
\SAcp = \big\{\sqrt{A_{ii}}x_i - x_i^2 : i \in [n]\big\} \cup \big\{A_{ij} - x_i x_j : 1 \leq i < j \leq n\big\}.
\]
We additionally define $\xib{cp}{*}(A)$ by adding the constraint $\rank(M(L)) < \infty$ to $\xib{cp}{\infty}(A)$. 
We also consider the strengthening $\smash{\xib{cp}{t,\dagger}(A)}$,  where we add to $\xib{cp}{t}(A)$ the positivity constraints
\begin{equation}\label{eqposcp}
L(gu) \geq 0 \quad \text{for} \quad g \in \{1\} \cup \SAcp  \quad  \text{and} \quad u \in [{\bf x}]_{2t-\deg(g)}
\end{equation}
and the tensor constraints
\begin{equation} \label{eq:tensorconstraints}
(L((ww')^c))_{w,w' \in \langle\bx\rangle_{=l}} \preceq A^{\otimes l} \quad \text{for all  integers } \quad 2 \leq l \leq t,
\end{equation}
which generalize the case $l=2$ used in the relaxation $\tau_\mathrm{cp}^\mathrm{sos}(A)$.
Here, for a word $w \in \langle\bx\rangle$, we denote by $w^c$ the corresponding (commutative) monomial in $[\bx]$.
The tensor  constraints (\ref{eq:tensorconstraints})  involve matrices indexed by the {\em noncommutative} words of length exactly $l$. In 
Section~\ref{sec:tensor} we show a more economical way to rewrite these constraints as
$
(L(mm'))_{m,m' \in [\bx]_{=l}} \preceq Q_l A^{\otimes l} Q_l^\T,
$
thus involving  smaller matrices indexed by {\em commutative} words of degree $l$. 

Note that, as before, we can strengthen the bounds by adding other localizing polynomials to the set $\SAcp$. In particular, we can follow the approach of Section~\ref{cpsd-additional-constraints}. Another possibility is to add localizing constraints specific to the commutative setting: we can add each monomial $u \in \cx$ to $\SAcp$ (see Section~\ref{sec:DJL} for an example). 

The bounds $\xib{cp}{t}(A)$ and $\xib{cp}{t,\dagger}(A)$ are monotonically nondecreasing in $t$ and they are invariant under simultaneously permuting the rows and columns of $A$ and under scaling a row and column of $A$ by a positive number.  In Propositions~\ref{prop:xibcp} and~\ref{lem:comparison} we show 
\[
\taucpsos(A)\le \xib{cp}{t,\dagger}(A)\le  \taucp(A) \quad \text{for} \quad t \geq 2,
\]
and in Proposition~\ref{prop:xi star dagger=taucp} we show the equality $\xib{cp}{*,\dagger}(A) = \taucp(A)$.

\subsection{Comparison to \texorpdfstring{$\tau_\mathrm{cp}^\mathrm{sos}(A)$}{tau cp sos}}
We first show that the semidefinite programs defining $\xib{cp}{t,\dagger}(A)$ are valid relaxations for the completely positive rank. More precisely, we show that they  lower bound $\taucp(A)$. 
\begin{proposition} \label{prop:xibcp}
For $A \in \CP^n$ and $t \in \N \cup \{\infty,*\}$ we have $\xib{cp}{t,\dagger}(A) \leq \taucp(A)$.
\end{proposition}
\begin{proof}
It suffices to show the inequality for $t=*$.
For this consider a decomposition $A=\alpha \, \smash{\sum_{k=1}^r \lambda_k R_k}$, where $\alpha\ge 1$, $\lambda_k>0$, $\sum_{k=1}^r \lambda_k = 1$, $0\le R_k\le A$, $R_k\preceq A$, and $\rank R_k= 1$.
There are nonnegative vectors $v_k$ such that  $R_k=v_k v_k^\T$. Define the linear map $L\in \R\cx^*$ by $L=\alpha\sum_{k=1}^r \lambda_k L_{v_k}$, where $L_{v_k}$ is the evaluation at $v_k$ mapping any polynomial $p\in \R\cx$ to $p(v_k)$.

The equality $(L(x_ix_j))=A$ follows from the identity $A=\alpha\sum_{k=1}^r \lambda_k R_k$. The constraints 
$
\smash{L((\sqrt{A_{ii}} x_i - x_i^2) p^2)} \geq 0
$ 
follow because 
\[
L_{v_k}(\sqrt{A_{ii}} x_i - x_i^2) p^2) = (\sqrt{A_{ii}} (v_k)_i - (v_k)_i^2) p(v_k)^2 \geq 0,
\]
where we use that $(v_k)_i \geq 0$ and  $(v_k)_i^2 = (R_k)_{ii} \leq A_{ii}$ implies $(v_k)_i^2 \leq  (v_k)_i \sqrt{A_{ii}}$. The constraints
$
L((A_{ij} - x_ix_j) p^2) \geq 0
$ 
and
\[
L(gu) \geq 0 \quad \text{for} \quad g \in \{1\} \cup \SAcp \quad  \text{and} \quad u \in [{\bf x}]
\]
follow in a similar way.

It remains to be shown that $X_l \preceq A^{\otimes l}$ for all $l$, where we set $X_l = (L(uv))_{u,v\in \ncx_{=l}}$. Note that $X_1=A$. We adapt the argument used in~\cite{FP16} to show $X_l \preceq A^{\otimes l}$ using induction on $l \geq 2$. Suppose $A^{\otimes (l-1)}\succeq X_{l-1}$. Combining $A-R_k\succeq 0$ and $R_k\succeq 0$ gives $(A-R_k)\otimes R_k^{\otimes (l-1)}\succeq 0$ and thus
$A\otimes R_k^{\otimes (l-1)}\succeq R_k^{\otimes l}$ for each $k$.
Scale by factor $\alpha \lambda_k$ and sum over $k$ to get
\[
A\otimes X_{l-1}=\sum_k \alpha \lambda_k A\otimes R_k^{\otimes (l-1)} \succeq \sum_k \alpha\lambda_k R_k^{\otimes l}= X_l.
\]
Finally, combining with $A^{\otimes (l-1)}-X_{l-1}\succeq 0$ and $A\succeq 0$, we obtain 
\[
A^{\otimes l} =A\otimes (A^{\otimes (l-1)}-X_{l-1})+ A\otimes X_{l-1} \succeq A\otimes X_{l-1}\succeq X_l. \qed
\]
\end{proof}

Now we show that the new parameter $\xib{cp}{2,\dagger}(A)$ is at least as good as  $\tau_\mathrm{cp}^\mathrm{sos}(A)$. Later in Section~\ref{sec: bipartite cp} we will give an example where the inequality is strict. 
\begin{proposition} \label{lem:comparison}
For $A \in \CP^n$ we have
$
\tau_{\mathrm{cp}}^{\mathrm{sos}}(A) \leq \xib{cp}{2,\dagger}(A).
$
\end{proposition}
\begin{proof}
Let $L$ be feasible for $\xib{cp}{2,\dagger}(A)$. We will construct a feasible solution to the program defining $\taucpsos(A)$ with objective value $L(1)$, which implies $\taucpsos(A)\le L(1)$ and thus the desired inequality. For this set $\alpha = L(1)$ and define the symmetric $n^2 \times n^2$ matrix $X$ by 
$
X_{(i,j),(k,l)} =L(x_ix_jx_kx_l)$ for $i,j,k,l \in [n]$.
Then the  matrix
\[
M:=\begin{pmatrix} \alpha & \text{vec}(A)^\T  \\ \text{vec}(A) & X \end{pmatrix}
\]
is positive semidefinite. This follows because $M$ is obtained from the principal submatrix of $M_2(L)$ indexed by the monomials $1$ and $x_ix_j$ ($1\le i\le j\le n$) where the rows/columns indexed by $x_j x_i$ with $1 \leq i < j \leq n$ are duplicates of the rows/columns indexed by $x_i x_j$. 

We have $L((A_{ij} - x_ix_j)x_ix_j) \geq 0$ for all $i,j$: 
For $i \neq j$ this follows using  the constraint  $L((A_{ij} - x_ix_j)u) \geq 0$ with $u = x_ix_j$ (from (\ref{eqposcp})), and for $i = j$ this follows from 
\[
L((A_{ii} -x_i^2) x_i^2) = L((\sqrt{A_{ii}} - x_i)^2 + 2 (\sqrt{A_{ii}} x_i - x_i^2)) \geq 0,
\]
which holds because of~\eqref{eq:x^2}, the constraint $L(p^2) \geq 0$ for $\deg(p)\leq 2$, and the constraint $L(\sqrt{A_{ii}} x_i - x_i^2) \geq 0$. Using $L(x_ix_j) = A_{ij}$, we get
$
X_{(i,j),(i,j)} = L(x_i^2x_j^2) \leq A_{ij}^2.
$
We also have
$
X_{(i,j),(k,l)} = L(x_ix_jx_kx_l) = L(x_ix_lx_kx_j) = X_{(i,l),(k,j)},
$ 
and the constraint $(L(uv))_{u,v \in \langle\bx \rangle_{=2}} \preceq A^{\otimes 2}$ implies $X \preceq A \otimes A$. \qed
\end{proof}

\subsection{Convergence of the basic hierarchy}

We first summarize convergence properties of the hierarchy $\xib{cp}{t}(A)$. Note that unlike in Section~\ref{sec:lower bounds on cpsd rank} where we can only claim the inequality \smash{$\xib{cpsd}{\infty}(A)\le \xib{cpsd}{*}(A)$}, here we can show  the equality \smash{$\xib{cp}{\infty}(A) = \xib{cp}{*}(A)$}. This is because we can use Theorem \ref{thm:eval-comm}, which permits to represent certain  truncated  linear functionals by  finite atomic measures.

\begin{proposition}\label{prop:cpconv}
Let $A \in \CP^n$. For every $t \in \N \cup \{\infty, *\}$ the optimum in $\xib{cp}{t}(A)$ is attained, and $\xib{cp}{t}(A) \to \xib{cp}{\infty}(A) = \xib{cp}{*}(A)$ as $t\to \infty$. If $\xib{cp}{t}(A)$ admits a flat optimal solution, then  $\smash{\xib{cp}{t}(A) = \xib{cp}{\infty}(A)}$.
Moreover, $\smash{\xib{cp}{\infty}(A) = \xib{cp}{*}(A)}$ is the minimum value of $L(1)$ taken over all conic combinations $\smash L$ of evaluations at elements of $\smash{D(\SAcp)}$ satisfying $A = (L(x_ix_j))$.
\end{proposition}

\begin{proof}
We may assume $A\ne 0$. Since $\sqrt{A_{ii}} x_i -x_i^2 \in \SAcp$ for all $i$, using~\eqref{eq:x^2} we obtain  that $\mathrm{Tr}(A) -\sum_i x_i^2 \in \MM_2(\SAcp)$.  By adapting the proof of Proposition~\ref{prop:lowerbound} to the commutative setting, we see that the optimum in $\xib{cp}{t}(A)$ is attained  for  $t \in \N \cup \{\infty\}$,  and $\xib{cp}{t}(A) \to \xib{cp}{\infty}(A)$ as $t\to \infty$.
 
We now show the inequality $\xib{cp}{*}(A)\le \xib{cp}{\infty}(A)$, which implies that equality holds. For this, let $L$ be optimal for $\smash{\xib{cp}{\infty}}(A)$. By Theorem~\ref{thm:eval-comm}, the restriction of $L$ to $\R[\bx]_2$ extends to  a conic combination of evaluations at points in $D(\SAcp)$. It follows that this extension is feasible for $\xib{cp}{*}(A)$ with the same objective value. This shows that $\xib{cp}{*}(A)\le \xib{cp}{\infty}(A)$, that the optimum in $\xib{cp}{*}(A)$ is attained, and that $\smash{\xib{cp}{*}(A)}$ is the minimum of $L(1)$ over all conic combinations $\smash L$ of evaluations at elements of $\smash{D(\SAcp)}$ such that $A = (L(x_ix_j))$. Finally, by Theorem~\ref{thm:flat-comm} we have $\smash{\xib{cp}{t}}(A) = \smash{\xib{cp}{\infty}}(A)$ if $\xib{cp}{t}(A)$ admits a flat optimal solution. \qed
\end{proof}
Next, we give a reformulation for the parameter $\xib{cp}{*}(A)$, which is similar to the formulation of $\tau_\mathrm{cp}(A)$, although it lacks the constraint $R \preceq A$ which is present in $\tau_\mathrm{cp}(A)$. 

\begin{proposition} \label{prop:tauprime}
We have 
\[
\xib{cp}{*}(A) = \min \Big\{ \alpha : \alpha \geq 0,\, A \in \alpha \cdot \mathrm{conv} \big\{ R \in \psdcone^n : 0 \leq R \leq A, \, \rank(R) \le  1\big\}\Big\}.
\]
\end{proposition}
\begin{proof}
This follows directly from the reformulation of $\xib{cp}{*}(A)$ in Proposition~\ref{prop:cpconv} in terms of conic evaluations at points in $D(\SAcp)$ after observing that, for $v \in \R^n$, we have $v \in D(\SAcp)$ if and only if the matrix $R = vv^\T$ satisfies $0 \leq R \leq A$. \qed
\end{proof}

\subsection{Additional constraints and convergence to \texorpdfstring{$\tau_\mathrm{cp}(A)$}{tau cp}} \label{sec:cp additional constraints}

The reformulation of the parameter $\xib{cp}{*}(A)$ in Proposition~\ref{prop:tauprime} differs from $\tau_\mathrm{cp}(A)$ in that the constraint $R\preceq A$ is missing. In order to have a hierarchy converging to $\tau_\mathrm{cp}(A)$ we need to add constraints to enforce that $L$ can be decomposed as a  conic combination of evaluation maps at nonnegative vectors  $v$ satisfying $vv^{\sf T}\preceq A$. Here we present two ways to achieve this goal. First we show that the tensor constraints~\eqref{eq:tensorconstraints}  suffice in the sense that $\xib{cp}{*,\dagger}(A) =\tau_{\mathrm{cp}}(A)$ (note that the constraints~\eqref{eqposcp} are not needed for this result).  However, because of the special form of the tensor constraints we do not know whether $\xib{cp}{t,\dagger}(A)$ admitting a flat optimal solution implies $\smash{\xib{cp}{t,\dagger}(A)} = \smash{\xib{cp}{*,\dagger}(A)}$, and we do not know whether $\smash{\xib{cp}{\infty,\dagger}(A)} = \smash{\xib{cp}{*,\dagger}(A)}$. Second, we adapt the approach of adding additional localizing constraints from Section~\ref{cpsd-additional-constraints} to the commutative setting, where we do show $\smash{\xib{cp}{\infty,\oS^{n-1}}(A)} = \smash{\xib{cp}{*,\oS^{n-1}}(A)} = \tau_{\mathrm{cp}}(A)$. This yields a doubly indexed sequence of semidefinite programs whose optimal values converge to $\tau_{\mathrm{cp}}(A)$.

\begin{proposition} \label{prop:xi star dagger=taucp}
Let $A \in \CP^n$. For every $t \in \N \cup \{\infty\}$ the optimum in $\xib{cp}{t,\dagger}(A)$ is attained. We have $\xib{cp}{t,\dagger}(A) \to \xib{cp}{\infty,\dagger}(A)$ as $t\to\infty$ and $\xib{cp}{*,\dagger}(A) =\tau_{\mathrm{cp}}(A)$.
\end{proposition}

\begin{proof}
The attainment of the optima in $\xib{cp}{t,\dagger}(A)$ for $t \in \N \cup \{ \infty \}$ and the convergence of $\xib{cp}{t,\dagger}(A)$ to $\xib{cp}{\infty,\dagger}(A)$ can be shown in the same way as the analogue statements for $\xib{cp}{t}(A)$ in Proposition~\ref{prop:cpconv}. 

We have seen the inequality $\smash{\xib{cp}{*,\dagger}(A)} \leq \smash{\tau_{\mathrm{cp}}(A)}$ in Proposition~\ref{prop:xibcp}. Now we show the reverse inequality.
Let $L$ be feasible for $\xib{cp}{*,\dagger}(A)$. We will show that $L$ is feasible for $\tau_{\mathrm{cp}}(A)$, which implies 
$\tau_{\mathrm{cp}}(A)\le L(1)$ and thus $\tau_{\mathrm{cp}}(A)\le \smash{\xib{cp}{*,\dagger}(A)}$.

By Proposition~\ref{lem:comparison} and the fact that $\rank(A) \leq \tau_{\mathrm{cp}}^{\mathrm{sos}}(A)$ we have $L(1) > 0$ (where we assume $A\ne 0$).
By Theorem~\ref{propLfinitedimcommutative}, we may write 
\[
L= L(1) \sum_{k=1}^K \lambda_k L_{v_k},
\]
where $\lambda_k>0$, $\sum_k \lambda_k =1$, and $L_{v_k}$ is an evaluation map at a point $v_k \in D(\SAcp)$. We define the matrices $R_k = v_k v_k^\T$, so  that $A = L(1) \sum_{k=1}^K R_k$. The matrices $R_k$ satisfy $0 \leq R_k \leq A$ since $v_k \in D(\SAcp)$. Clearly also $R_k \succeq 0$. It remains to show that $R_k \preceq A$. For this we use the tensor constraints~\eqref{eq:tensorconstraints}. Using that $L$ is a conic combination of evaluation maps, we may rewrite these constraints as 
\[
L(1) \sum_{k=1}^K \lambda_k R_k^{\otimes l} \preceq A^{\otimes l},  
\]
from which it follows that $L(1) \lambda_k R_k^{\otimes l} \preceq A^{\otimes l}$ for all $k\in [K]$. 
Therefore, for all $k\in [K]$ and all vectors $v$ with $v^{\sf T}R_kv>0$ we have 
\begin{equation}\label{eq:conictensor}
L(1) \lambda_k \leq \left(\frac{v^\T A v}{v^\T R_kv}\right)^l \quad \text{for all} \quad  l \in \N.
\end{equation}
Suppose there is a $k$ such that $R_k \not \preceq A$. Then there exists a $v$ such that $\smash{v^\T R_k v > v^\T A v}$. 
As $(v^\T A v) / (v^\T R_kv) < 1$, letting $l$ tend to $\infty$ we obtain $L(1)\lambda_k=0$, reaching a contradiction.
It follows that $R_k \preceq A$ for all $k \in [K]$. \qed
\end{proof}

The second approach for reaching $\tau_\mathrm{cp}(A)$  is based on using the extra localizing constraints from Section~\ref{cpsd-additional-constraints}. For a subset $V\subseteq \oS^{n-1}$,  define $\smash{\xib{cp}{t,V}(A)}$ by replacing the truncated quadratic module $\mathcal M_{2t}(\SAcp)$ in $\smash{\xib{cp}{t}(A)}$ by $\mathcal M_{2t}(\SAVcp)$, where 
\[
\SAVcp= \SAcp \cup \big\{v^\T Av-\Big(\sum_{i=1}^n v_ix_i\Big)^2 : v\in V\big\}.
\]
Proposition~\ref{remlocsphere} can be adapted to the completely positive setting, so that we have a sequence of finite subsets $V_1 \subseteq V_2 \subseteq \ldots \subseteq \oS^{n-1}$ with
$
\smash{\xib{cp}{*,V_k}(A) \to \xib{cp}{*,\oS^{n-1}}(A)}
$
as $k\rightarrow \infty$.
Proposition~\ref{prop:cpconv} still holds when adding extra localizing constraints, so that for any  $k\ge 1$  we have
\[
\lim_{t \to \infty} \xib{cp}{t,V_k}(A) = \xib{cp}{*,V_k}(A).
\] Combined with Proposition~\ref{prop:xitotau} this shows that we have a doubly indexed sequence $\smash{\xib{cp}{t,V_k}}(A)$ of semidefinite programs that converges to  $\tau_\mathrm{cp}(A)$ as $t \to \infty$ and $k \to \infty$. 
\begin{proposition}\label{prop:xitotau}
For $A \in \CP^n$ we have $\xib{cp}{*,\oS^{n-1}}(A) = \tau_{\mathrm{cp}}(A)$.
\end{proposition}
\begin{proof}
The proof is the same as the proof of Proposition~\ref{prop:tauprime}, with the following additional observation: Given a vector $u \in \R^n$, we have $u \in D(S_{A,\oS^{n-1}}^{\hspace{0.1em}\mathrm{cp}})$ only if $uu^\T \preceq A$. The latter follows from the additional localizing constraints: for each $v \in \R^n$ we have 
\[
0 \leq v^\T A v - \Big(\sum_i v_i u_i \Big)^2 = v^{\T} ( A - uu^\T ) v. \qed
\]
\end{proof}

\subsection{More efficient tensor constraints}
\label{sec:tensor}
Here we show that for any integer $l\ge 2$ the constraint $A^{\otimes l} -(L((ww')^c))_{w,w'\in \ncx_{=l}}\succeq 0$, used in the definition of $\xib{cp}{t,+}(A)$, can be reformulated in a more economical way using matrices indexed by {\em commutative} monomials in $\cx_{=l}$ instead of noncommutative words in $\ncx_{=l}$.  For this we exploit the symmetry in the matrices $\smash{A^{\otimes l}}$ and $(L((ww')^c))_{w,w'\in \ncx_{=l}}$ for $L \in \R[\bx]_{2l}^*$. Recall  that for a  word $w \in \langle\bx\rangle$, we let $w^c$ denote the corresponding  (commutative) monomial in $[\bx]$.

Define the matrix $Q_l \in \R^{\cx_{=l} \times \ncx_{=l}}$ by
\begin{equation}\label{eqQl}
(Q_l)_{m,w} = \begin{cases}
1/d_m  & \text{ if } w^c = m,\\
0 & \text{ otherwise,}
\end{cases}
\end{equation}
where, for $m = x_1^{\alpha_1} \cdots x_n^{\alpha_n} \in [\bx]_{=l}$, we define the multinomial coefficient
\begin{equation}\label{eqm}
d_m = \big|\big\{w\in \ncx_{=l}: w^c = m\big\}\big| = \frac{l!}{\alpha_1! \cdots \alpha_n!}.
\end{equation}

\begin{lemma}\label{lemQ}
For $L \in \R\cx_{2l}^*$ we have 
\[
Q_l (L((ww')^c))_{w,w'\in \ncx_{=l}} Q_l^\T = (L(mm'))_{m,m'\in \cx_{=l}}.
\]
\end{lemma}
\begin{proof}
For $m,m'\in \cx_{l}$, the $(m,m')$-entry of the left hand side is equal to
\begin{align*}
\sum_{w,w'\in \ncx_{=l}} Q_{mw}Q_{m'w'}L((ww')^c) &= \sum_{\underset{w^c = m}{w \in \ncx_{=l}}} \sum_{\underset{(w')^c = m'}{w' \in \ncx_{=l}}} \frac{L((ww')^c)}{d_md_{m'}} = L(mm').\qed
\end{align*}
\end{proof}
The symmetric group $S_l$ acts on $\langle \bx \rangle_{=l}$ by $(x_{i_1} \cdots x_{i_l})^\sigma = x_{i_{\sigma(1)}} \cdots x_{i_{\sigma(l)}}$ for $\sigma\in S_l$. Let 
\begin{equation}\label{eqP}
P = \frac{1}{l!} \sum_{\sigma \in S_l} P_\sigma,
\end{equation}
where, for any $\sigma\in S_l$, $P_\sigma \in \smash{\R^{\ncx_{=l} \times \ncx_{=l}}}$ is the permutation matrix defined by
\[
(P_\sigma)_{w,w'} = \begin{cases} 1 & \text{if } w^\sigma = w',\\ 0 & \text{otherwise}.\end{cases}
\]
A matrix $M \in \smash{\R^{\ncx_{=l} \times \ncx_{=l}}}$ is said to be {\em $S_l$-invariant} if $P^\sigma M = M P^\sigma$ for all $\sigma \in S_l$.
\begin{lemma} \label{leminvariant}
If $M \in \smash{\R^{\ncx_{=l} \times \ncx_{=l}}}$ is symmetric and $S_l$-invariant, then 
\[
M\succeq 0 \quad \Longleftrightarrow \quad Q_l M Q_l^\T \succeq 0.
\]
\end{lemma}
\begin{proof}
The implication $M \succeq 0 \Longrightarrow Q_l M Q_l^\T \succeq 0$ is immediate. For the other implication we need a preliminary fact. Consider the diagonal matrix $D \in \smash{\R^{\cx_{=l}\times \cx_{=l}}}$ with $D_{mm}= d_m$ for $m \in \cx_{=l}$.
We claim that $Q_l^\T D Q_l = P$, the matrix in~\eqref{eqP}. Indeed, for any $w,w'\in \ncx_{=l}$, we have 
\begin{align*}
(Q_l^\T D Q_l)_{ww'} &= \sum_{m\in \cx_{=l}} (Q_l)_{mw}(Q_l)_{mw'}D_{mm} = \begin{cases} 1/d_m & \text{if } w^c = (w')^c=m,\\ 0 & \text{otherwise}\end{cases}\\
&= \frac{|\{\sigma\in S_l: w^\sigma=w'\}|}{l!} = P_{ww'}.
\end{align*}
Suppose $Q_l M Q_l^\T \succeq 0$, and let $\lambda$ be an eigenvalue of $M$ with eigenvector $z$. Since $MP=PM$, we may assume $Pz=z$, for otherwise we can replace $z$ by $P z$, which is still an eigenvector of $M$ with eigenvalue $\lambda$. We may also assume $z$ to be a unit vector. Then $\lambda \ge 0$ can be shown using the identity $Q_l^\T D Q_l=P$ as follows:
\[
\lambda = z^\T M z = z^\T P M P z = z^\T (Q_l^\T D Q_l) M(Q_l^\T D Q_l)z = (D Q_l z)^\T (Q_l M Q_l^\T) D Q_l z \geq 0. \hspace{-0.1cm}\qed
\]
\end{proof}

We can now derive our symmetry reduction result:
\begin{proposition} \label{prop:tensor}
For $L \in \R\cx_{2l}^*$ we have 
\[
A^{\otimes l}-(L((ww')^c))_{w,w'\in \ncx_{=l}}\succeq 0 \quad \Longleftrightarrow \quad Q_l A^{\otimes l}Q_l^\T - (L(mm'))_{m,m'\in \cx_{=l}}\succeq 0.
\]
\end{proposition}
\begin{proof}
For any $w,w'\in \ncx_{=l}$ we have $(P_\sigma A^{\otimes l} P_\sigma^\T)_{w,w'} = A^{\otimes l}_{w^\sigma, (w')^\sigma} = A^{\otimes l}_{w,w'}$ and 
\[
(P_\sigma (L((uu')^c))_{u,u'\in \ncx_{=l}} P_\sigma^*)_{w,w'} = L((w^\sigma(w')^\sigma)^c) = L((ww')^c).
\]
This  shows that the matrix $A^{\otimes l}-(L((ww')^c))_{w,w'\in \ncx_{=l}}$ is $S_l$-invariant. Hence the claimed result follows by using Lemma~\ref{lemQ} and Lemma~\ref{leminvariant}. \qed
\end{proof}

\subsection{Computational examples}

\subsubsection{Bipartite matrices} \label{sec: bipartite cp}

Consider the $(p+q)\times (p+q)$ matrices
\[
P(a,b) = \begin{pmatrix} (a+q) I_p & J_{p,q} \\ J_{q,p} & (b+p) I_q \end{pmatrix}, \quad a,b \in \R_+,
\]
where $J_{p,q}$ denotes the all-ones matrix of size $p \times q$. We have $P(a,b)=P(0,0)+D$ for some nonnegative diagonal matrix $D$. As can be easily verified, $P(0,0)$ is completely positive with $\cprank(P(0,0))=pq$, so $P(a,b)$ is completely positive with $pq \leq \cprank(P(a,b)) \leq pq + p + q$.

For $p=2$ and $q=3$ we have $\cprank(P(a,b))=6$ for all $a,b \ge 0$, which follows from the fact that $5 \times 5$ completely positive matrices with at least one zero entry have $\cprank$ at most $6$; see~\cite[Theorem~3.12]{BSM03}. Fawzi and Parrilo~\cite{FP16} show that $\tau_{\text{cp}}^{\mathrm{sos}}(P(0,0)) = 6$, and give a subregion of $[0,1]^2$ where $5 < \tau_{\text{cp}}^{\mathrm{sos}}(P(a,b)) < 6$. The next lemma shows the bound $\xib{cp}{2,\dagger}(P(a,b))$ is tight for all $a,b \geq 0$ and therefore strictly improves on $\taucp^{\mathrm{sos}}$ in this region.
\begin{lemma} \label{lem:lowercpfawziexample}
For $a,b \geq 0$ we have $\xib{cp}{2,\dagger}(P(a,b)) \geq pq$. 
\end{lemma}
\begin{proof}
Let $L$ be feasible for $\xib{cp}{2,\dagger}(P(a,b))$ and let
\[
B = \begin{pmatrix} \alpha & c^\T \\ c & X \end{pmatrix}
\]
be the principal submatrix of $M_2(L)$ where the rows and columns are indexed by 
\[
\{1\} \cup \{x_ix_j : 1 \leq i \leq p, \, p+1 \leq j \in p+q\}.
\]
It follows that $c$ is the all ones vector $c = \mathbf{1}$. Moreover, if $P(a,b)_{ij} = 0$ for some $i\ne j$, then the constraints $L(x_ix_ju) \geq 0$ and $L((P(a,b)_{ij} - x_ix_j)u) \geq 0$ imply $L(x_i x_j u) = 0$ for all $u \in [\bx]_2$. Hence, $X_{x_ix_j,x_kx_l} = L(x_i x_j x_k x_l) = 0$ whenever $x_ix_j \neq x_k x_l$. It follows that $X$ is a diagonal matrix. We write
\[
B = \begin{pmatrix} \alpha & \mathbf{1}^\T \\ \mathbf{1} & \mathrm{Diag}(z_1, \ldots, z_{pq}) \end{pmatrix}.
\]
Since $\begin{pmatrix} 1 & - \mathbf{1}^\T \\ -\mathbf{1} & J \end{pmatrix} \succeq 0$ we have 
\[
0 \leq \mathrm{Tr}\left(\begin{pmatrix} \alpha & \mathbf{1}^\T \\ \mathbf{1} & \mathrm{Diag}(z_1, \ldots, z_{pq}) \end{pmatrix} \begin{pmatrix} 1 & - \mathbf{1}^\T \\ -\mathbf{1} & J \end{pmatrix}\right) = \alpha - 2 pq + \sum_{k = 1}^{pq} z_k.
\]
Finally, by the constraints $L((P(a,b)_{ij} - x_i x_j) u) \geq 0$ (with $i \in [p], j \in p+[q]$ and $u = x_i x_j$) and  $L(x_i x_j) = P(a,b)_{ij}$ we obtain $z_k \leq 1$ for all $k \in [pq]$. Combined with the above inequality, it follows that 
\[
L(1) = \alpha \geq 2pq - \sum_{k=1}^{pq} z_k \geq pq,
\]
and hence $\xib{cp}{2,\dagger}(P(a,b)) \geq pq$. \qed
\end{proof}

\subsubsection{Examples related to the DJL-conjecture} \label{sec:DJL}

The Drew-Johnson-Loewy conjecture~\cite{DJL94} states that the maximal $\cprank$ of an $n~\times~n$ completely positive matrix is equal to $\lfloor n^2/4 \rfloor$. Recently this conjecture has been disproven for $n=7,8,9,10,11$ in~\cite{BSU14} and for all $n \geq 12$ in~\cite{BSU15} (interestingly, it remains open for $n=6$). Here we study our bounds on the examples of~\cite{BSU14}. Although our bounds are not tight for the $\cprank$, they are non-trivial and as such may be of interest for future comparisons. For numerical stability reasons we have evaluated our bounds on scaled versions of the matrices from~\cite{BSU14}, so that the diagonal entries become equal to $1$. The matrices $\tilde M_7$, $\tilde M_8$ and $\tilde M_9$ correspond to the matrices $\tilde M$ in Examples 1,2,3 of~\cite{BSU14}, and $M_7$, $M_{11}$ correspond to the matrices $M$ in Examples 1 and 4. The column $\xib{cp}{2,\dagger}(\cdot) + x_i x_j$ corresponds to the bound $\xib{cp}{2,\dagger}(\cdot)$ where we replace $\SAcp$ by $\SAcp \cup \{ x_i x_j : 1 \leq i < j \leq n\}$.

\begin{table}[h]
\caption{Examples from~\cite{BSU14} with various bounds on their cp-rank}
\centering
\begin{tabular}{ l l l l l l l l l}
\hline\noalign{\smallskip}
Example & $\cprank(\cdot)$ & $\lfloor \frac{n^2}{4}\rfloor$ & $\rank(\cdot)$ & $\xib{cp}{1}(\cdot)$ &  $\xib{cp}{2}(\cdot)$ & $\xib{cp}{2,\dagger}(\cdot)$ & $\xib{cp}{2,\dagger}(\cdot) + x_i x_j$ & $\xib{cp}{3,\dagger}(\cdot)$ \\
\noalign{\smallskip}\hline\noalign{\smallskip}
  $M_7$ & $14$ & $12$  & $7$ & $2.64$ & $4.21$ & $7.21$ & $9.75$ & $10.50$ \\
  $\widetilde{M}_7$ & $14$ & $12$ & $7$ & $2.58$ & $4.66$ & $8.43$ & $9.53$  & $10.50$ \\ 
  $\widetilde{M}_8$ &  $18$ & $16$ & $8$ & $3.23$ & $5.45$ & $8.74$ & $10.41$ & $13.82$ \\ 
  $\widetilde{M}_9$ & $26$ & $20$ & $9$ & $3.39$ & $5.71$ & $11.60$ & $13.74$ & $17.74$ \\
  $M_{11}$ & $32$ & $30$ & $11$ & $4.32$ & $7.46$ & $20.76$ & $21.84$ & -- \\
  \noalign{\smallskip}\hline
\end{tabular}

\end{table}

\section{Lower bounds on the nonnegative rank}\label{sec:lowernnr}

In this section we adapt the techniques for the cp-rank from Section~\ref{sec:lowercp} to the asymmetric setting of the nonnegative rank. We now view a factorization $A = (a_i^\T b_j)_{i \in [m], j \in [n]}$ by nonnegative vectors as a factorization by positive semidefinite diagonal matrices, by writing $A_{ij} = \Tr(X_i X_{m+j})$, with $X_i =\Diag(a_i)$ and $X_{m+j} = \Diag(b_j)$. Note that we can view this as a ``partial matrix'' setting, where for the symmetric matrix $(\Tr(X_iX_k))_{i,k\in [m+n]}$ of size $m+n$, only the off-diagonal entries at the  positions $(i,m+j)$ for $i\in [m], j\in [n]$ are specified. 

This asymmetry requires rescaling the factors in order  to get upper bounds on their maximal eigenvalues, which is needed to ensure the Archimedean property for the selected localizing polynomials. For this we use the well-known fact that for any $A \in \R_+^{m \times n}$ there exists a factorization $A=(\Tr(X_iX_{m+j}))$ by diagonal nonnegative matrices of size $\rank_+(A)$, such that
\[
\lambda_\max(X_i), \lambda_\max(X_{m+j}) \leq \sqrt{A_\max}  \quad \text{for all} \quad i \in [m], j \in [n],
\]
where $A_\max := \max_{i,j} A_{ij}$. To see this, observe that for any rank one matrix $R = u v^\T$ with $0 \leq R \leq A$, one may assume $0 \leq u_i, v_j \leq \sqrt{A_\max}$ for all $i,j$.
Hence, the set 
\[
\SAplus = \big\{\sqrt{A_\max}x_i - x_i^2 : i \in [m+n]\big\} \cup \big\{A_{ij} - x_i x_{m+j} : i \in [m], j \in [n] \big\}
\]
is localizing for $A$; that is, there exists a minimal factorization $\bX$ of $A$ with $\bX \in \mathcal D(S_A^+)$.

\medskip
Given  $A\in\R^{m\times n}_{\ge 0}$, for each $t \in \N \cup \{\infty\}$ we consider the semidefinite program
\begin{align*}
\xib{+}{t}(A) = \mathrm{min} \big\{ L(1) : \; & L \in \R[x_1,\ldots,x_{m+n}]_{2t}^*,\\
&L(x_ix_{m+j}) = A_{ij} \quad \text{for} \quad i \in [m], j \in [n],\\
&L \geq 0 \quad \text{on} \quad \mathcal M_{2t}(\SAplus) \big\}.
\end{align*}
Moreover, define $\xib{+}{*}(A)$ by adding the constraint $\rank(M(L)) < \infty$ to the program defining $\xib{+}{\infty}(A)$. It it easy to check that $\xib{+}{t}(A)\le \xib{+}{\infty}(A)\le \xib{+}{*}(A)\le \rank_+(A)$ for $t \in \N$.

Denote by $\xib{+}{t,\dagger}(A)$ the strengthening of $\xib{+}{t}(A)$ where we add the positivity constraints
\begin{equation}\label{eqposrank+}
L(gu) \geq 0 \quad \text{for} \quad g \in \{1\} \cup \SAplus \quad  \text{and} \quad u \in [{\bf x}]_{2t-\deg(g)}.
\end{equation}
Note that these extra constraints can help for finite $t$, but  are redundant for $t \in \{\infty, *\}$.

\subsection{Comparison to other bounds}

As in the previous section, we compare our bounds to the bounds by Fawzi and Parrilo~\cite{FP16}. They introduce the following parameter $\tau_+(A)$ as analogue of the bound $\tau_\mathrm{cp}(A)$ for the nonnegative rank:
\[
\tau_+(A) = \min \Big\{ \alpha : \alpha\ge 0,\,A \in\alpha \cdot \mathrm{conv} \big\{ R \in \R^{m \times n}: 0 \leq R \leq A, \, \rank(R) \le   1\big\}\Big\},
\]
and  the  analogue $\tau_+^\mathrm{sos}(A)$ of the bound $\tau_{\mathrm{cp}}^{\mathrm{sos}}(A)$ for the nonnegative rank:
\begin{align*}
\tau_{+}^{\mathrm{sos}}(A) = \mathrm{inf} \big\{ \alpha : \; &  X \in \R^{mn \times mn}, \, \alpha \in \R,\\[-0.1em]
&\hspace{-0.3em}\begin{pmatrix} \alpha & \text{vec}(A)^\T  \\ \text{vec}(A) & X \end{pmatrix} \succeq 0, \\
& X_{(i,j),(i,j)} \leq A_{ij}^2 \quad \text{for} \quad 1 \leq i \leq m, 1 \leq  j \leq n, \\
& X_{(i,j),(k,l)} = X_{(i,l),(k,j)} \quad \text{for} \quad 1 \leq i < k \leq m, \; 1 \leq j < l \leq n \big\}.
\end{align*}

First we give the analogue of Proposition~\ref{prop:cpconv}, whose proof we omit since it is very similar.

\begin{proposition}\label{prop:+conv}
Let $A \in \R_+^{m \times n}$. For every $t \in \N \cup \{\infty, *\}$ the optimum in $\xib{+}{t}(A)$ is attained, and $\xib{+}{t}(A) \to \xib{+}{\infty}(A) = \xib{+}{*}(A)$ as $t\to\infty$. If $\xib{+}{t}(A)$ admits a flat optimal solution, then $\smash{\xib{+}{t}(A) = \xib{+}{*}(A)}$.
Moreover, $\smash{\xib{+}{\infty}(A) = \xib{+}{*}(A)}$ is the minimum of $L(1)$ over all conic combinations $\smash L$ of trace evaluations at elements of $\smash{D(\SAplus)}$ satisfying $A =( L(x_ix_{m+j}))$.
\end{proposition}

Now we observe that the parameters 
$\xib{+}{\infty}(A)$ and $\xib{+}{*}(A)$ coincide with $\tau_+(A)$, so that we have a sequence of semidefinite programs converging to $\tau_+(A)$.
\begin{proposition} \label{prop:tau+}
For any $A \in \R_{\geq 0}^{m \times n}$, we have $\xib{+}{\infty}(A) = \xib{+}{*}(A) = \tau_+(A).$
\end{proposition}
\begin{proof}
The discussion at the beginning of Section~\ref{sec:lowernnr} shows that for any rank one matrix $R$
satisfying $0 \leq R \leq A$ we may assume that $R=uv^\T$ with $(u,v)\in\R^m_+ \times \R^n_+$ and $ u_i,v_j \leq \sqrt{A_{\max}}$ for $i\in [m],j\in [n]$. Hence, $\tau_+(A)$ can be written as 
\begin{align*}
\min \Big\{\alpha : \alpha \geq 0,\, A \in \alpha& \cdot \mathrm{conv} \big\{ uv^\T \colon u \in \Big[0, \sqrt{A_\max}\Big]^m, v \in \Big[0, \sqrt{A_\max}\Big]^n,\,  uv^\T \leq A \big\} \Big\} \\
&=\min \Big\{ \alpha : \alpha \geq 0,\, A \in \alpha \cdot \mathrm{conv}\big\{uv^\T: (u,v) \in  D(\SAplus)\big\} \Big\}.
\end{align*}
The equality $\xib{+}{\infty}(A) = \xib{+}{*}(A)=\tau_+(A)$ now follows from the reformulation of $\xib{+}{*}(A)$ in Proposition~\ref{prop:+conv} in terms of conic evaluations, after noting that for $(u,v)$ in $ \R^m\times\R^n$ we have $(u,v)\in  D(\SAplus)$ if and only if the matrix $R=uv^\T$ satisfies $0\le R\le A$. \qed
\end{proof}

Analogously  to the case of the completely positive rank we have the following proposition. The proof is similar to that of Proposition 4.2, considering now for $M$ the principal submatrix of $M_2(L)$ indexed by the monomials 1 and $x_ix_{m+j}$ for $i\in [m]$ and $j\in [n]$.
\begin{proposition} \label{lem:xi+tausos}
If $A$ is a nonnegative matrix, then $\xib{+}{2,\dagger}(A) \geq \tau_{+}^{\mathrm{sos}}(A)$. 
\end{proposition}

In the remainder of this section we recall how $\tau_+(A)$ and $\tau_{+}^{\mathrm{sos}}(A)$ compare to other bounds in the literature. These  bounds can be divided into two categories: combinatorial lower bounds and norm-based lower bounds. The following diagram from~\cite{FP16} summarizes how $\tau_+^{\mathrm{sos}}(A)$ and $\tau_+(A)$ relate to the combinatorial lower bounds
\[
\begin{array}{rcccccl}
                          &      & \tau_+^{\mathrm{sos}}(A) & \leq & \tau_+(A) & \leq & \rank_+(A)\\
                          &      &  \rotatebox[origin=c]{90}{{\large  $\leq$}} &  &    \rotatebox[origin=c]{90}{{\large  $\leq$}} &  &  \quad \rotatebox[origin=c]{90}{{\large  $\leq$}} \\
\mathrm{fool}(A) = \omega(\mathrm{RG}(A)) & \leq & \overline{\vartheta}(\mathrm{RG}(A)) & \leq & \chi_{\mathrm{frac}}(\mathrm{RG}(A)) & \leq & \chi(\mathrm{RG}(A)) = \rank_B(A).
\end{array}
\]
Here  $\mathrm{RG}(A)$ is the {\em rectangular graph}, with  $V = \{(i,j)\in [m]\times [n]: A_{ij} > 0\}$ as vertex set and  $E = \{ ((i,j),(k,l)): A_{il} A_{kj}= 0\}$ as edge set. The coloring number of $\mathrm{RG}(A)$ coincides with  the well known \emph{rectangle covering number} (also denoted $\rank_B(A)$), which was used, e.g.,  in~\cite{FMPTdW12} to show that the extension complexity of the correlation polytope is exponential. The clique number of $\mathrm{RG}(A)$ is also known as the {\em fooling set number} (see, e.g.,~\cite{FKPT}). Observe that the above combinatorial lower bounds only depend on the sparsity pattern of the matrix $A$, and that they are all equal to one for a strictly positive matrix.

Fawzi and Parrilo~\cite{FP16} have furthermore shown that the bound  $\tau_+(A)$ is at least as good as norm-based lower bounds:
\[
\tau_+(A) = \underset{\substack{\mathcal N \text{ monotone and}  \\ \text{  positively homogeneous}}}{\sup} \frac{\mathcal N^*(A)}{\mathcal N(A)}. 
\]
Here, a function $\mathcal N: \R^{m \times n}_{+} \rightarrow \R_{+}$ is \emph{positively homogeneous} if $\mathcal N(\lambda A) = \lambda \mathcal N(A)$ for all $\lambda \geq 0$ and \emph{monotone} if $\mathcal N(A) \leq \mathcal N(B)$ for $A \leq B$, and $\mathcal N^*(A)$ is defined as 
\begin{align*}
\mathcal N^*(A) = \max \{ L(A) : &\ L:\R^{m\times n}\rightarrow \R  \text { linear and } L(X) \leq 1 
\text{ for all } X \in \R^{m \times n}_{+} \\  & \text{ with } \rank(X) \leq 1
 \text{ and } \mathcal N(X) \leq 1\}.
\end{align*}
These bounds are called {\em norm-based}  since norms often provide valid functions $\mathcal N$. For example, when $\mathcal N$ is the  $\ell_\infty$-norm, Rothvo{\ss}~\cite{Roth14} used the corresponding lower bound to show that the matching polytope has exponential extension complexity. 

When $\mathcal N$ is the Frobenius norm: $\mathcal N(A) = (\sum_{i,j} A_{ij}^2)^{1/2}$, the parameter $\mathcal N^*(A)$ is known as the \emph{nonnegative nuclear norm}. In~\cite{FP15} it is denoted by $\nu_+(A)$, shown to satisfy $\rank_+(A)\geq\left(\nu_+(A)/||A||_F\right)^2$, and reformulated as 
\begin{align}
\hspace{-0.18cm} \nu_+(A) &= \min \big\{ \sum_i \lambda_i : A = \sum_{i} \lambda_i u_i v_i^\T, \, (\lambda_i,u_i,v_i) \in \R^{1+m+n}_{+}, 
\, ||u_i||_2 = ||v_i||_2 = 1 \big\} \label{eq:nuplus} \\ 
& = \max \big\{ \langle A, W \rangle : W \in \R^{m \times n}, \, \bigl(\begin{smallmatrix} I & -W \\ -W^\T & I\end{smallmatrix}\bigr) \text{ is copositive} \big\}. \label{eq:nupluscopo}
\end{align}
where the cone of copositive matrices is the dual of the cone of completely positive matrices. Fawzi and Parrilo~\cite{FP15} use the copositive formulation~\eqref{eq:nupluscopo} to provide bounds $\nu_+^{[k]}(A)$ ($k\ge 0$),  based on inner approximations of the copositive cone from~\cite{Par00},  which converge to $\nu_+(A)$ from below.
We now observe that by Theorem~\ref{thm:eval-comm} the atomic formulation of $\nu_+(A)$ from~\eqref{eq:nuplus} can be seen as a moment optimization problem:
\[
\nu_+(A) = \min \int_{V(S)} 1 \, d \mu(x) \quad \text{s.t.} \quad A_{ij} = \int_{V(S)} x_i x_{m+j} \, d \mu(x) \quad \text{for}\quad i\in [m], j\in [n].
\]
Here, the optimization variable $\mu$ is required to be a Borel measure on the variety $V(S)$, where 
\[
S=\textstyle{\{\sum_{i=1}^mx_i^2-1, \ \sum_{j=1}^n x_{m+j}^2-1\}}.
\]
(The same observation is made in~\cite{TS15} for the real nuclear norm of a symmetric $3$-tensor and in~\cite{Nie16} for symmetric odd-dimensional tensors.) For $t \in \N \cup \{\infty\}$, let $\mu_t(A)$ denote the parameter defined analogously to $\xib{+}{t}(A)$, where we replace the condition $L\ge 0$ on $\mathcal M_{2t}(S_A^+)$ by  $L\ge 0$ on $\mathcal M_{2t}(\{x_1,\ldots, x_{m+n}\})$ and $L=0$ on $\mathcal I_{2t}(S)$, and let $\mu_*(A)$ be obtained by adding the constraint $\rank(M(L)) < \infty$ to  $\mu_\infty(A)$. We have $\mu_t(A) \rightarrow \mu_\infty(A) = \mu_*(A) = \nu_+(A)$ by Theorem~\ref{thm:eval-comm} and (a non-normalized analogue of) Theorem~\ref{thm:convcomm}. One can show that $\mu_1(A)$ with the additional constraints $L(u) \geq 0$ for all $u \in \cx_2$, is at least as good as $\nu_+^{[0]}(A)$. It is not clear how the hierarchies $\mu_t(A)$ and $\nu_+^{[k]}(A)$ compare in general.

\subsection{Computational examples}
We illustrate the performance of our approach  by comparing our lower bounds $\xib{+}{2,\dagger}$ and $\xib{+}{3,\dagger}$ to the lower bounds $\tau_+$ and $\tau_+^{\mathrm{sos}}$ on the two examples considered in~\cite{FP16}.

\subsubsection{All nonnegative $2 \times 2$ matrices}

For $A(\alpha) = \bigl(\begin{smallmatrix} 1 & 1 \\ 1 & \alpha\end{smallmatrix}\bigr)$, Fawzi and Parrilo~\cite{FP16} show that
\[
\tau_+(A(\alpha)) = 2-\alpha \quad \text{and} \quad \tau_+^{\mathrm{sos}}(A(\alpha)) = \frac{2}{1+\alpha} \quad \text{for all} \quad 0 \leq \alpha \leq 1.
\]
Since the parameters $\tau_+(A)$ and $\tau_+^{\mathrm{sos}}(A)$ are invariant under scaling and permuting rows and columns of $A$, one can use the identity 
\[
\begin{pmatrix} 1 & 1 \\ 1 & \alpha \end{pmatrix} =  \begin{pmatrix}  1 & 0 \\ 0 & \alpha \end{pmatrix}\begin{pmatrix} 1 & 1 \\ 1 & 1/\alpha \end{pmatrix}\begin{pmatrix}  0 & 1 \\ 1 & 0 \end{pmatrix}
\]	
to see this describes the parameters for all nonnegative $2 \times 2$ matrices. 
By using a semidefinite programming solver for $\alpha = k/100$, $k \in [100]$, we see that $\xib{+}{2}(A(\alpha))$ coincides with $\tau_+(A(\alpha))$.

\subsubsection{The nested rectangles problem}

In this section we consider the nested rectangles problem as described in~\cite[Section~2.7.2]{FP16} (see also~\cite{MSvS03}), which asks for which $a, b$ there exists a triangle $T$ such that $R(a,b) \subseteq T \subseteq P$, where $R(a,b) = [-a,a] \times [-b,b]$ and $P = [-1,1]^2$.

The nonnegative rank relates not only to the extension complexity of a polytope~\cite{Yan91}, but also to extended formulations of nested pairs~\cite{BFPS15,GG12}. An \emph{extended formulation} of a pair of polytopes $P_1\subseteq P_2 \subseteq \R^d$ is a (possibly) higher dimensional polytope $K$ whose projection $\pi(K)$ is nested between $P_1$ and $P_2$. Let us suppose $\pi(K)= \{ x \in \R^d : y \in \R_+^k, \, (x,y) \in K\}$ and $K= \{(x,y): Ex+Fy = g,\, y \in \R^k_+\}$, then $k$ is the \emph{size of the extended formulation}, and the smallest such $k$ is called the \emph{extension complexity} of the pair $(P_1, P_2)$. It is known (cf.~\cite[Theorem~1]{BFPS15}) that the extension complexity of the pair $(P_1,P_2)$, where 
\[
P_1 = \mathrm{conv}(\{v_1, \ldots, v_n\}) \quad \text{and} \quad P_2 = \{x : a_i^\T x \leq b_i \text{ for } i \in [m]\},
\]
is equal to the nonnegative rank of the generalized slack matrix $S_{P_1,P_2} \in \R^{m \times n}$, defined by
\[
(S_{P_1,P_2})_{ij} = b_j - a_j^\T v_i \quad \text{for} \quad i\in [m], j\in [n].
\]
Any nonnegative matrix is the slack matrix of some nested pair of polytopes~\cite[Lemma~4.1]{GPT13} (see also~\cite{GG12}).

Applying this to the pair  $(R(a,b), P)$, one immediately sees that there exists a polytope $K$ with at most three facets whose projection $T = \pi(K)\subseteq \R^2$ satisfies $R(a,b) \subseteq T \subseteq P$ if and only if the pair $(R(a,b), P)$ admits an extended formulation of size $3$. For $a,b>0$, the polytope $T$ has to be $2$ dimensional, therefore $K$ has to be at least $2$ dimensional as well; it follows that $K$ and $T$ have to be triangles. 
Hence there exists a triangle $T$ such that $R(a,b) \subseteq T \subseteq P$ if and only if the nonnegative rank of the slack matrix $S(a,b) := S_{R(a,b),P}$ is equal to $3$. One can verify that
\[
S(a,b) = \begin{pmatrix} 1-a & 1+a & 1-b & 1+b \\  1+a & 1-a & 1-b & 1+b \\1+a & 1-a & 1+b &1-b \\ 1-a & 1+a & 1+b & 1-b \end{pmatrix}.
\]
Such a triangle exists if and only if $(1+a)(1+b) \leq 2$ (see~\cite[Proposition~4]{FP16} for a proof sketch). To test the quality of their bound, Fawzi and Parrilo~\cite{FP16} compute $\tau_+^{\mathrm{sos}}(S(a,b))$  for different values of $a$ and $b$. In doing so they determine the region where  $\tau_+^{\mathrm{sos}}(S(a,b))>3$. We do the same for the bounds $\xib{+}{1,\dagger}(S(a,b)), \xib{+}{2,\dagger}(S(a,b))$ and $\xib{+}{3,\dagger}(S(a,b))$, see Figure~\ref{fignestedrectangles}. The results show that $\xib{+}{2,\dagger}(S(a,b))$ strictly improves upon the bound $\tau_+^{\mathrm{sos}}(S(a,b))$, and that $\xib{+}{3,\dagger}(S(a,b))$ is again a strict improvement over $\xib{+}{2,\dagger}(S(a,b))$.
\begin{figure}
\centering
\begin{tikzpicture}
    \begin{axis}[enlargelimits=false,axis on top,xlabel =$a$, ylabel = $b$, 
    		xtick={0,0.2,0.4,0.6,0.8,1},ytick={0,0.2,0.4,0.6,0.8,1},
		xticklabel style={tick align = outside},
		yticklabel style={tick align = outside},
		axis equal image
		]
				
       \addplot graphics
       [xmin=0,xmax=1,ymin=0,ymax=1]
       {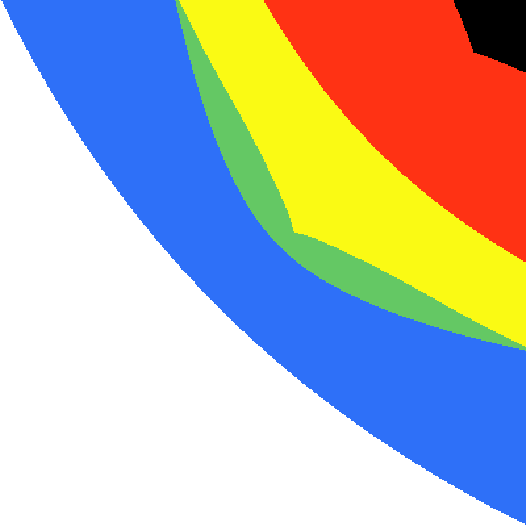};
    \end{axis}
\end{tikzpicture}
\caption{The colored region corresponds to $\rank_+(S(a,b)) = 4$. The top right region (black) corresponds to $\xib{+}{1,\dagger}(S(a,b)) >3$, the two top right regions (black and red) together correspond to $\tau_+^\mathrm{sos}(S(a,b)) > 3$, the three top right regions (black, red and yellow) to $\xib{+}{2,\dagger}(S(a,b))>3$, and the four top right regions (black, red, yellow, and green) to $\xib{+}{3,\dagger}(S(a,b))>3$}
\label{fignestedrectangles}
\end{figure}

\section{Lower bounds on the positive semidefinite rank} \label{sec:psdrank}

The positive semidefinite rank can be seen as an {\em asymmetric} version of the completely positive semidefinite rank.
Hence, as was the case in the previous section for the nonnegative rank, we need to select suitable factors in a minimal factorization in order to be able to bound their maximum eigenvalues and obtain a localizing set of polynomials leading to an Archimedean quadratic module.

For this we can follow, e.g., the approach in~\cite[Lemma~5]{LWdW17} to rescale a  factorization and claim that,
for any $A \in \R^{m\times n}_+$ with psd-rank$_\C(A) = d$, there exists a factorization $A =( \langle X_i, X_{m+j}\rangle)$ by  matrices $X_1, \ldots, X_{m+n} \in \Hermitian_{+}^d$ such that $\sum_{i=1}^m X_i = I$ and $\mathrm{Tr}(X_{m+j}) = \sum_i A_{ij}$ for all $j\in [n]$. Indeed, starting from any  factorization $X_i,X_{m+j}$ in $\Hermitian^d_+$ of $A$, we may replace $X_i$ by $X^{-1/2}X_iX^{-1/2}$ and $X_{m+j}$ by $X^{1/2}X_{m+j}X^{1/2}$, where $X:=\sum_{i=1}^m X_i$ is positive definite (by minimality of $d$). This argument shows that the set of polynomials 
\[
\SApsd = \big\{x_i - x_i^2 : i \in [m]\big\} \cup \big\{\Big(\sum_{i=1}^m A_{ij}\Big) x_{m+j} - x_{m+j}^2 : j \in [n] \big\}
\]
is localizing for $A$; that is, there is {\em at least one} minimal factorization $\bX$ of $A$ such that $g(\bX)\succeq 0$ for all polynomials $g\in \SApsd$. Moreover, for the same minimal factorization $\bX$ of $A$ we have $p(\bX)  (1-\sum_{i=1}^m X_i) = 0$ for all $p \in \R\ncx$. 

Given $A\in\R^{m\times n}_{+}$, for each $t \in \N \cup \{\infty\}$ we consider the semidefinite program
\begin{align*}
\xib{psd}{t}(A) = \mathrm{min} \big\{ L(1) : \; & L \in \R\langle x_1,\ldots,x_{m+n}\rangle_{2t}^*,\\
&L(x_ix_{m+j}) = A_{ij} \quad \text{for} \quad i \in [m], j \in [n],\\[-0.2em]
&L \geq 0 \quad \text{on} \quad \mathcal M_{2t}(\SApsd), \\ 
&L = 0 \quad \text{on} \quad \mathcal I_{2t}(1-\textstyle{\sum_{i=1}^m x_i}) \big\}.
\end{align*}
We additionally define $\xib{psd}{*}(A)$ by adding the constraint $\rank(M(L)) < \infty$ to the program defining $\smash{\xib{psd}{\infty}(A)}$ (and considering the infimum instead of the minimum, since we do not know if the infimum is attained in $\smash{\xib{psd}{*}(A)}$).
By the above discussion it follows that 
the parameter $\smash{\xib{psd}{*}(A)}$ is a lower bound on psd-rank$_\C(A)$
and we have
\[
\xib{psd}{1}(A)\le \ldots \le \xib{psd}{t}(A)\le \ldots \le \xib{psd}{\infty}(A)\le \xib{psd}{*}(A)\le \psdrank_\C(A).
\] 
Note that, in contrast to the previous bounds, the parameter $\xib{psd}{t}(A)$ is not invariant under rescaling the rows of $A$ or under taking the transpose of $A$ (see Section~\ref{sec: examples psd rank polygons}).

It follows from the construction of $\SApsd$ and Equation~\eqref{eq:x^2} that the quadratic module $\MM(\SApsd)$ is Archimedean, and hence the following analogue of Proposition~\ref{prop:lowerbound} can be shown.
\begin{proposition} 
Let $A \in \R^{m \times n}_+$. For each $t \in \N \cup \{\infty\}$, the optimum in $\xib{psd}{t}(A)$ is attained, and we have
\[
\lim_{t \to \infty} \xib{psd}{t}(A) = \xib{psd}{\infty}(A).
\]
Moreover, $\smash{\xib{psd}{\infty}(A)}$ is equal to the infimum over all  
 $\alpha \geq 0$ for which there exists a unital $C^*$-algebra $\MA$ with tracial state $\tau$ and $\bX \in \mathcal D_{\mathcal A} (\SApsd) \cap \mathcal V_\mathcal{A}(1-\textstyle{\sum_{i=1}^m x_i})$ such that $A = \alpha \cdot (\tau(X_iX_{m+j}))_{i\in [m],j\in [n]}$.
\end{proposition}

\subsection{Comparison to other bounds}

In~\cite{LWdW17} the following bound on the complex positive semidefinite rank was derived: 
\begin{equation} \label{eq:lowerbound LWdW}
\psdrank_\C(A) \geq \sum_{i=1}^m \max_{j \in [n]} \frac{ A_{ij}}{\sum_i A_{ij}}.
\end{equation}
If a feasible linear form $L$ to $\xib{psd}{t}(A)$ satisfies the inequalities $L(x_i( \sum_i A_{ij} - x_{m+j}))~\geq~0$ for all $i \in [m], j \in [n]$, then $L(1)$ is at least the above lower bound. Indeed, the inequalities give 
\[
L(x_i) \geq \max_{j \in [n]} \, \frac{L(x_i x_{m+j})}{\sum_i A_{ij}} = \max_{j \in [n]} \, \frac{A_{ij}}{ \sum_i A_{ij}}. 
\]
and hence 
\[
L(1) = \sum_{i = 1}^m L(x_i) \geq \sum_{i=1}^m \max_{j \in [n]} \frac{ A_{ij}}{\sum_i A_{ij}}.
\]
The inequalities $L(x_i( \sum_i A_{ij} - x_{m+j})) \geq 0$ are easily seen to be valid for trace evaluations at points of $\MD(\SApsd)$. More importantly, as in Lemma~\ref{lem:bilinear}, these inequalities are satisfied by feasible linear forms to the programs $\xib{psd}{\infty}(A)$ and $\xib{psd}{*}(A)$. Hence, $\xib{psd}{\infty}(A)$ and $\xib{psd}{*}(A)$ are at least as good as the lower bound~\eqref{eq:lowerbound LWdW}.

In~\cite{LWdW17} two other fidelity based lower bounds on the psd-rank were defined; we do not know how they compare to $\xib{psd}{t}(A)$. 

\subsection{Computational examples}

In this section we apply our bounds to some (small) examples taken from  the literature, namely $3\times 3$ circulant matrices and slack matrices of small polygons.

\subsubsection{Nonnegative circulant matrices of size $3$}

We consider the nonnegative circulant matrices of size $3$ which are, up to scaling, of the form 
\[
M(b,c) = \begin{pmatrix} 1 & b & c \\ c & 1 & b \\ b & c & 1 \end{pmatrix} \quad \text{with} \quad b,c \geq 0.
\]
If $b=1=c$, then $\rank(M(b,c)) = \psdrank_\R(M(b,c)) = \psdrank_\C(M(b,c)) = 1$.
Otherwise we have $\rank(M(b,c))\ge 2$, which implies $\psdrank_\K(M(b,c)) \geq 2$ for $\K \in \{\R,\C\}$. In~\cite[Example~2.7]{FGPRT15} it is shown that
\[
\psdrank_\R(M(b,c)) \leq 2 \quad \Longleftrightarrow \quad 1 +b^2 +c^2 \leq 2(b + c + bc).
\]
Hence, if $b$ and $c$ do not satisfy the above relation then $\psdrank_\R(M(b,c))=3$.

\begin{figure}
\centering
\begin{tikzpicture}[scale=1]
    \begin{axis}[enlargelimits=false,axis on top,xlabel =$b$, ylabel = $c$, 
    		xtick={0,0.5,1,1.5,2,2.5,3,3.5,4},ytick={0,0.5,1,1.5,2,2.5,3,3.5,4},
		xticklabel style={tick align = outside},
		yticklabel style={tick align = outside},
		axis equal image]
				
       \addplot graphics
       [xmin=0,xmax=4,ymin=0,ymax=4]
       {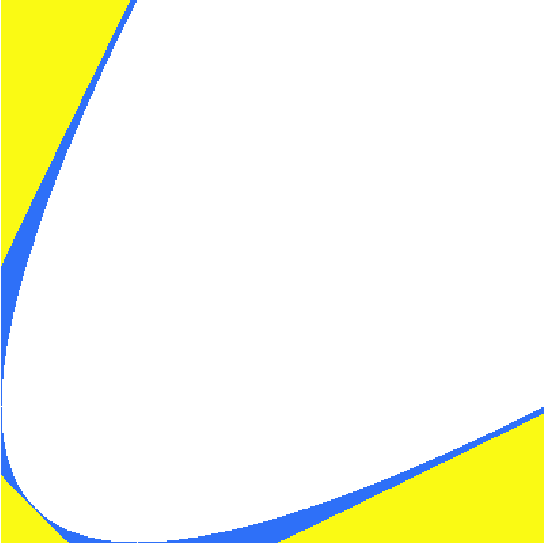};
    \end{axis}
\end{tikzpicture}
\caption{The colored region corresponds to the values $(b,c)$ for which $\psdrank_\R(M(b,c)) =3$; the outer region (yellow) shows the values of $(b,c)$ for which $\xib{psd}{2}(M(b,c))>2$.}
\label{fig:psdcirc}
\end{figure}

To see how good our lower bounds are for this example, we use a semidefinite programming solver to compute $\xib{psd}{2}(M(b,c))$ for $(b,c) \in [0,4]^2$ (with stepsize $0.01$). In Figure~\ref{fig:psdcirc} we see that the bound $\xib{psd}{2}(M(b,c))$ certifies that $\psdrank_\R(M(b,c)) =\psdrank_\C(M(b,c))=3$ for most values $(b,c)$ where $\psdrank_\R(M(b,c))=3$.

\subsubsection{Polygons} \label{sec: examples psd rank polygons} 
Here we consider the slack matrices of two polygons in the plane, where the bounds are sharp (after rounding) and illustrate the dependence on scaling the rows or taking the transpose.  
We consider the quadrilateral $Q$ with vertices 
$(0,0),(0,1),(1,0),(2,2)$, and the regular hexagon $H$, whose slack matrices are given by
\[
S_Q = \begin{pmatrix} 0 &0 &2 &2 \\ 1 & 0 & 0& 3 \\ 0 & 1 & 3 & 0 \\ 2 &2 &0 & 0\end{pmatrix}, \qquad 
S_H = \begin{pmatrix} 0 & 1& 2& 2& 1& 0 \\ 0& 0& 1&2 &2& 1 \\1 & 0 & 0 & 1 & 2 & 2 \\ 2& 1& 0& 0& 1& 2\\ 2& 2& 1& 0& 0 &1 \\ 1& 2& 2& 1& 0 &0 \end{pmatrix}.
\]
Our lower bounds on the $\psdrank_\C$ are not invariant under taking the transpose, indeed numerically we have $\xib{psd}{2}(S_Q) \approx 2.266$ and $\xib{psd}{2}(S_Q^\T) \approx 2.5$. The slack matrix $S_Q$ has $\psdrank_\R(S_Q) = 3$ (a corollary of~\cite[Theorem 4.3]{GRT13}) and therefore both bounds certify $\psdrank_\C(S_Q) = 3 = \psdrank_\R(S_Q)$. 

Secondly, our bounds are not invariant under rescaling the rows of a nonnegative matrix. Numerically we have $\xib{psd}{2}(S_H) \approx 1.99$ while $\xib{psd}{2}(DS_H) \approx 2.12$, where $D = \mathrm{Diag}(2,2,1,1,1,1)$. The bound $\xib{psd}{2}(DS_H)$ is in fact tight (after rounding) for the complex positive semidefinite rank of $DS_H$ and hence of $S_H$: in~\cite{GGS16} it is shown that $\psdrank_\C(S_H) = 3$.

\section{Discussion and future work}

In this work we provide a unified approach for the four matrix factorizations obtained by considering (a)symmetric factorizations by nonnegative vectors and positive semidefinite matrices. Our methods can be extended to the nonnegative tensor rank, which is defined as  the smallest integer $d$ for which a $k$-tensor $A \in \R_+^{n_1 \times \cdots \times n_k}$ can be written as $A = \sum_{l=1}^d u_{1,l}\otimes \cdots \otimes u_{k,l}$ for nonnegative vectors $u_{j,l} \in \R_+^{n_j}$. The approach from Section~\ref{sec:lowernnr} for $\rank_+$ can be extended to obtain a hierarchy of lower bounds on the nonnegative tensor rank. For instance, if  $A$ is a 3-tensor, the analogous bound $\xib{+}{t}(A)$ is obtained by minimizing $L(1)$ over $L\in \R[x_{1},\ldots,x_{n_1+n_2+n_3}]^*$ such that $L(x_{i_1}x_{n_1+i_2}x_{n_1+n_2+i_3})=A_{i_1i_2i_3}$ (for $i_1\in [n_1],i_2\in [n_2],i_3\in [n_3]$), using as localizing polynomials in $S_A^+$ the polynomials $\sqrt[3]{A_\max}x_i-x_i^2$
and $A_{i_1i_2i_3}- x_{i_1}x_{n_1+i_2}x_{n_1+n_2+i_3}$.
As in the matrix case one can compare to the bounds $\tau_+(A)$ and $\tau_+^{\mathrm{sos}}(A)$ from~\cite{FP16}. One can show $\xib{+}{*}(A)=\tau_+(A)$, and one can show $\smash{\xib{+}{3,\dagger}(A)} \ge \tau_+^{\mathrm{sos}}(A)$ after adding the conditions
$L(x_{i_1}x_{n_1+i_2}x_{n_1+n_2+i_3}(A_{i_1i_2i_3}- x_{i_1}x_{n_1+i_2}x_{n_1+n_2+i_3}))\ge 0$ to $\xib{+}{3}(A)$.

Testing membership in the completely positive cone and the completely positive semidefinite cone  is another important problem, to which our hierarchies can also be applied. It follows from the proof of Proposition~\ref{prop:cpconv} that if $A$ is not completely positive then, for some order $t$, the program  $\xib{cp}{t}(A)$ is infeasible or its optimum value is larger than the Caratheodory bound on the cp-rank (which is similar to an earlier result in~\cite{Nie14}).
In the noncommutative setting the situation is  more complicated: If $\xib{cpsd}{*}(A)$ is feasible, then $A\in \CSP$,  and if $A\not\in \CSPvN$, then $\xib{cpsd}{\infty}(A)$ is infeasible (Propositions~\ref{prop:lowerbound} and~\ref{prop:cpsd*}). 
Here $\CSPvN$ is the cone defined in~\cite{BLP17} consisting  of the matrices admitting a factorization in a von Neumann algebra with a trace. By Lemma~\ref{remark:von Neumann},  $\CSPvN$ can equivalently be characterized as the set of matrices of the form $\alpha \, (\tau(a_ia_j))$ for some $C^*$-algebra $\mathcal A$ with tracial state $\tau$, positive elements $a_1,\ldots,a_n\in\mathcal A$ and $\alpha\in\R_+$.

Our lower bounds are on the \emph{complex} version of the (completely) positive semidefinite rank. As far as we are aware, the existing lower bounds (except for the dimension counting rank lower bound) are also on the complex (completely) positive semidefinite rank. It would be interesting to find a lower bound on the \emph{real} (completely) positive semidefinite rank that can go beyond the complex (completely) positive semidefinite rank. 

We conclude with some open questions regarding applications of lower bounds on matrix factorization ranks. First, as was shown in~\cite{PSVW16,GdLL17,PV17}, completely positive semidefinite matrices whose $\cpsdrank_\C$ is larger than their size do exist, but currently we do not know how to construct small examples for which this holds. Hence, a concrete question: Does there exist a $5 \times 5$ completely positive semidefinite matrix whose $\cpsdrank_\C$ is at least $6$? 
Second, as we mentioned before, the asymmetric setting corresponds to (semidefinite) extension complexity of polytopes. Rothvo{\ss}' result~\cite{Roth14} (indirectly) shows that the parameter $\xib{+}{\infty}$ is exponential (in the number of nodes of the graph) for the slack matrix of the matching polytope. Can this result also be shown directly using the dual formulation of $\xib{+}{\infty}$, that is, by a sum-of-squares certificate? If so, could one extend the argument to the noncommutative setting (which would show a lower bound on the semidefinite extension complexity)?

\begin{acknowledgements}
The authors would like to thank Sabine Burgdorf for helpful discussions and an anonymous referee for suggestions that helped improve the presentation. 
\end{acknowledgements}

\appendix

\section{Commutative and tracial polynomial optimization} \label{sec:background}

In this appendix we discuss known convergence and flatness results for commutative and tracial polynomial optimization. We present these results in such a way that they can be directly used for our hierarchies of lower bounds on matrix factorization ranks. Although the commutative case was developed first, here we treat the commutative and tracial cases together. For the reader's convenience we provide all proofs by working on the ``moment side''; that is, relying on properties of linear functionals rather than using real algebraic results on sums of squares.
Tracial optimization is an adaptation of eigenvalue optimization as developed in~\cite{NPA10}, but here we only discuss the commutative and tracial cases, as these are most relevant to our work.

\subsection{Flat extensions and representations of linear forms}\label{sec:flat ext}

The optimization variables in the optimization problems considered in this paper are linear forms on spaces of (noncommutative) polynomials. To study the properties of the bounds obtained through these optimization problems we need to study properties and representations of (flat) linear forms on polynomial spaces.

\medskip
In Section~\ref{sec:prelim} the key examples of symmetric tracial linear functionals on $\R\ncx_{2t}$ are trace evaluations on a (finite dimensional) $C^*$-algebra. In this section we present some results that provide conditions under which, conversely, a symmetric tracial linear map on $\R\ncx_{2t}$ ($t \in \N \cup \{\infty\}$) that is nonnegative on $\MM(S)$ and zero on $\mathcal I(T)$ arises from trace evaluations at elements in the intersection of the $C^*$-algebraic analogs of the matrix positivity domain of $S$ and the matrix ideal  of $T$.
In Theorems~\ref{propLinfinitedim} and~\ref{propLfinitedim} we consider the case $t= \infty$ and in Theorem~\ref{propextension} we consider the case $t \in \N$. Results like these can for instance be used to link the linear forms arising in the limiting optimization problems of our hierarchies to matrix factorization ranks. 
 
The proofs of Theorems~\ref{propLinfinitedim} and~\ref{propLfinitedim} use a classical Gelfand--Naimark--Segal (GNS) construction. In these proofs it  will also 
 be convenient to work with the concept of the null space of a linear functional $L\in\R\ncx_{2t}^*$, which is defined as the vector space 
\[
N_t(L) = \big\{ p \in \R\langle \bx \rangle_t : L(qp) = 0 \text{ for } q\in \R\langle \bx \rangle_t\big\}.
\]
We use the notation $N(L)=N_\infty(L)$ for the nontruncated null space. Recall that $M_t(L)$ is the moment matrix associated to $L$, its rows and columns are indexed by words in $\ncx_t$, and its entries are given by $M_t(L)_{w,w'} = L(w^* w')$ for $w,w' \in \ncx_t$. The null space of $L$ can therefore be identified with the kernel of $M_t(L)$: A polynomial $p=\sum_{w}c_w w$ belongs to $N_t(L)$ if and only if its coefficient vector $(c_w)$ belongs to the kernel of $M_t(L)$.

In Section~\ref{sec:prelim} we defined a linear functional $L \in \R\ncx_{2t}^*$ to be $\delta$-flat based on the rank stabilization property~\eqref{eq:flat} of its moment matrix: $\rank(M_t(L)) = \rank(M_{t-\delta}(L))$. This definition can be reformulated in terms of a decomposition of the corresponding polynomial space using the null space: the form $L$ is $\delta$-flat if and only if 
\[ 
\R\langle \bx \rangle_t = \R\langle \bx \rangle_{t-\delta} + N_t(L).
\]
Recall that $L$ is said to be flat if it is $\delta$-flat for some $\delta \ge 1$. Finally, in the nontruncated case ($t=\infty$) $L$ was called flat if $\rank(M(L))<\infty$. We can now see that $\rank(M(L))<\infty$ if and only if there exists an integer $s \in \N$ such that $\R\ncx = \R\ncx_s + N(L)$.

\medskip

Theorem~\ref{propLinfinitedim} below is implicit in several works (see, e.g.,~\cite{NPA12,BKP16}). Here we assume that $\mathcal M(S) + \mathcal I(T)$ is Archimedean, which we recall means that there exists a scalar $R>0$ such that 
\[
R-\sum_{i=1}^n x_i^2\in \mathcal M(S)+ \mathcal I(T).
\]
\begin{theorem}
\label{propLinfinitedim}
Let $S\subseteq \mathrm{Sym}\, \R\ncx$ and $T \subseteq \R\ncx$ with $\MM(S)+ \mathcal I(T)$ Archimedean. Given a linear form $L\in \R\ncx^*$, the following are equivalent:
\begin{itemize}
\item[(1)] $L$ is symmetric, tracial, nonnegative on $\MM(S)$, zero on $\mathcal I(T)$, and $L(1) = 1$;
\item[(2)] there is a unital $C^*$-algebra $\mathcal A$ with tracial state $\tau$ and ${\bf X} \in \DAS \cap \mathcal V_{\mathcal A}(T)$ with
\begin{equation}\label{eqLA}
L(p)=\tau(p({\bf X})) \quad \text{for all} \quad p\in \R\ncx.
\end{equation}
\end{itemize}
\end{theorem}
\begin{proof}
We first prove the easy direction $(2) \Rightarrow (1)$: We have
\[
L(p^*) = \tau(p^*(\bX)) = \tau(p(\bX)^*) = \overline{\tau(p(\bX))} = \overline{L(p)} = L(p),
\]
where we use that $\tau$ is Hermitian and $X_i^* = X_i$ for $i \in [n]$. Moreover, $L$ is tracial since $\tau$ is tracial. In addition, for $g \in S \cup \{1\}$ and $p \in \R\langle \bx\rangle$ we have 
\[
L(p^*gp) = \tau(p^*(\bX) g(\bX) p(\bX)) = \tau(p(\bX)^* g(\bX) p(\bX)) \geq 0,
\]
since $g({\bf X})$ is positive in $\mathcal A$ as $\bX\in \MD_\MA(S)$ and $\tau$ is positive. Similarly $L(hq) = \tau(h(\bX) q(\bX)) = 0$ for all $h \in T$, since $\bf X\in \mathcal V_{\mathcal A}(T)$.

We show $(1) \Rightarrow (2)$ by applying a GNS construction. Consider the quotient vector space $\R\ncx/N(L)$, and denote the class of $p$ in $\R\ncx/N(L)$ by $\overline p$. We can equip this quotient with the inner product $\langle \overline p,\overline q\rangle =L(p^*q)$ for $p,q\in\R\ncx$, so that the completion $\Hilbert$ of $\R\ncx/N(L)$ is a separable Hilbert space. As $N(L)$ is a left ideal in $\R\ncx$, the operator
\begin{equation}\label{eqXi}
X_i \colon \R\ncx/N(L) \rightarrow  \R\ncx/N(L), \, \overline p  \mapsto \overline{x_ip}
\end{equation}
is well defined. We have 
\[
\langle X_i\,\overline p,\overline q\rangle = L((x_ip)^*q) = L(p^*x_iq)=\langle \overline p,X_i\overline q\rangle  \quad \text{for all} \quad p,q \in \R\langle {\bf x} \rangle,
\]
so the $X_i$ are self-adjoint. Since $g \in S \cup \{1\}$ is symmetric and $\langle \overline p, g(\bX) \overline p\rangle = \langle \overline p,\overline{gp}\rangle = L(p^* g p)\ge 0$ for all $p$ we have $g({\bf X}) \succeq 0$. By the Archimedean condition, there exists an $R > 0$ such that $R-\sum_{i=1}^nx_i^2\in \mathcal M(S)+\MI(T)$. Using  $R-x_i^2= (R-\sum_{j=1}^nx_j^2)+\sum_{j\ne i}x_j^2 \in \mathcal M(S)+\MI(T)$ we get
\[
\langle X_i\overline p,X_i\overline p\rangle = L(p^*x_i^2p)\le R\cdot L(p^*p)=R\langle \overline p,\overline p\rangle \quad \text{for all} \quad p \in \R\langle {\bf x} \rangle.
\]
So each $X_i$ extends to a bounded self-adjoint operator, also denoted $X_i$, on the Hilbert space $\Hilbert$ such that $g({\bf X})$ is positive for all $g \in S \cup \{1\}$. Moreover, we have
$
\langle \overline f, h(\bX) \overline 1 \rangle = L(f^* h) = 0$ for all $f \in \R\ncx, h \in T$. 

The operators $X_i \in \mathcal B(\Hilbert)$ extend to self-adjoint operators in $\mathcal B(\C \otimes_\R \Hilbert)$, where $\C \otimes_\R \Hilbert$ is the complexification of $\Hilbert$.
Let $\mathcal A$ be the unital $C^*$-algebra obtained by taking the operator norm closure of $\R\langle \bX \rangle \subseteq \mathcal B(\C \otimes_\R \Hilbert)$. It follows that $\bX \in \mathcal D_{\mathcal A}(S) \cap \mathcal V_{\mathcal A}(T)$. 

Define the state $\tau$ on $\mathcal A$ by
$\tau(a) = \langle \overline{1}, a\overline{1} \rangle$ for $a \in \mathcal A$. For all $p,q \in \R\ncx$ we have
\begin{equation} \label{eqtracial}
\tau(p(\bX) q(\bX)) = \langle \overline{1}, p(\bX) q(\bX)\overline{1} \rangle = \langle \overline{1}, \overline{pq} \rangle = L(pq),
\end{equation}
so that the restriction of $\tau$ to $\R\langle \bX\rangle$ is tracial. Since $\R\langle \bX \rangle$ is dense in $\mathcal A$ in the operator norm, this implies $\tau$ is tracial.

To conclude the proof, observe that~\eqref{eqLA} follows from~\eqref{eqtracial} by taking $q=1$.  \qed
\end{proof}

The next result can be seen as a finite dimensional analogue of the above result, where we do not need 
$\mathcal M(S) +\mathcal I(T)$ to be Archimedean, but instead we assume the rank of $M(L)$ to be finite (i.e., $L$ to be flat). In addition to the Gelfand--Naimark--Segal construction, the proof uses Artin--Wedderburn theory. For the unconstrained case the proof of this result can be found in~\cite{BK12}, and in~\cite{BKP16,KP16} this result is extended to the constrained case.
\begin{theorem}
\label{propLfinitedim}
For $S\subseteq \mathrm{Sym}\, \R\ncx$, $T \subseteq \R\ncx$, and 
$L\in \R\ncx^*$, the following are equivalent:
\begin{enumerate}
\item[(1)] $L$ is a symmetric, tracial, linear form with $L(1) =1$ that is nonnegative on $\MM(S)$, zero on $\mathcal I(T)$, and has $\mathrm{rank}(M(L)) < \infty$;
\item[(2)] there is a finite dimensional $C^*$-algebra $\mathcal A$ with a tracial state $\tau$, and ${\bf X} \in \smash{\DAS} \cap \mathcal V_{\mathcal A}(T)$ satisfying equation~\smash{\eqref{eqLA}};
\item[(3)] $L$ is a convex combination of normalized trace evaluations at points in $\smash{\MD(S)} \cap \mathcal V(T)$.
\end{enumerate}
\end{theorem}
\begin{proof}
((1) $\Rightarrow$ (2)) Here we can follow the proof of Theorem~\ref{propLinfinitedim}, with the extra observation that the
condition $\rank (M(L))<\infty$ implies that the quotient space $\R\ncx/N(L)$ is finite dimensional. Since $\R\ncx/N(L)$ is finite dimensional the multiplication operators are bounded, 
and the constructed $C^*$-algebra $\mathcal A$ is finite dimensional.

((2) $\Rightarrow$ (3)) 
By Artin-Wedderburn theory there exists a $*$-isomorphism
\[
\varphi \colon \mathcal A \to \bigoplus_{m=1}^M \C^{d_m \times d_m}\quad \text{ for some } \ M\in \N,\ d_1,\ldots,d_M\in \N.
\]
Define the $*$-homomorphisms $\varphi_m \colon \mathcal A \to \C^{d_m \times d_m}$ for $m\in[M]$ by $\varphi = \oplus_{m=1}^M \varphi_m$. Then, for each $m\in [M]$,  the map 
$\C^{d_m \times d_m} \to \C$ defined by $X \mapsto \tau(\varphi_m^{-1}(X))$ is a positive tracial linear form, and hence it is a nonnegative multiple $\lambda_m \mathrm{tr}(\cdot)$ of the normalized matrix trace (since, for a full matrix algebra, the normalized trace is the unique tracial state). Then we have  $\tau(a) = \sum_m \lambda_m\, \mathrm{tr}(\varphi_m(a))$ for all $a\in\MA$. So $\tau(\cdot)=\sum_m\lambda_m \mathrm{tr}(\cdot)$  for nonnegative scalars $\lambda_m$ with $\sum_m \lambda_m = L(1) = 1$. By defining the matrices $X_i^m = \varphi_m(X_i)$ for $m\in [M]$, we get
\[
L(p) = \tau(p(X_1,\ldots,X_n)) = \sum_{m=1}^M \lambda_m\, \mathrm{tr}(p(X_1^m, \ldots, X_n^m)) \quad \text{ for all }\quad  p\in \R\ncx.
\]
Since $\varphi_m$ is a $*$-homomorphism we have $g(X_1^m,\ldots,X_n^m) \succeq 0$ for all $g \in  S \cup \{1\}$ and also $h(X_1^m,\ldots,X_n^m) = 0$ for all $h \in T$, which shows $(X_1^m,\ldots,X_n^m) \in \mathcal D(S) \cap \mathcal V(T)$.

((3) $\Rightarrow$ (1)) If $L$ is a conic combination of trace evaluations at elements from $\DS \cap \mathcal V(T)$, then $L$ is symmetric, tracial, nonnegative on $\mathcal M(S)$, zero on $\mathcal I(T)$, and satisfies $\rank(M(L)) < \infty$ because the moment matrix of any trace evaluation has finite rank. \qed
\end{proof}

The previous two theorems were about linear functionals defined on the full space of noncommutative polynomials. The following result claims that a {\em flat} linear functional on a truncated polynomial space can be extended to a flat linear functional on the full space of polynomials while preserving the same positivity properties. 
It is due to Curto and Fialkow~\cite{CF96} in the commutative case and extensions to the noncommutative case can be found  in~\cite{NPA10} (for eigenvalue optimization) and~\cite{BK12} (for trace optimization). 
\begin{restatable}{theorem}{respropextension}\label{propextension}
Let $1 \leq \delta \leq t < \infty$, $S \subseteq \mathrm{Sym} \, \R\ncx_{2\delta}$, and $T \subseteq \R\ncx_{2\delta}$. Suppose $L\in \R\ncx_{2t}^*$ is symmetric, tracial, $\delta$-flat, nonnegative on $\MM_{2t}(S)$, and zero on $\mathcal I_{2t}(T)$. Then
$L$ extends to a symmetric, tracial, linear form on $\R\ncx$ that is nonnegative on $\mathcal M(S)$, zero on $\mathcal I(T)$, and whose moment matrix has finite rank.
\end{restatable}
\begin{proof}
Let  $W\subseteq \ncx_{t-\delta}$  index a maximum nonsingular submatrix of $M_{t-\delta}(L)$, and let $\mathrm{span}(W)$ be the linear space spanned by $W$. We have the vector space direct sum 
\begin{equation} \label{eq:basecase}
\R\ncx_t=\mathrm{span}(W)\oplus N_t(L).
\end{equation}
That is, for each $u \in \ncx_t$ there exists a unique $r_u \in \mathrm{span}(W)$ such that $u - r_u \in N_t(L)$. 

We first construct the (unique) symmetric flat extension $\hat L \in \R\ncx_{2t+2}$ of $L$. For this we set $\hat L(p) = L(p)$ for $\deg(p) \leq 2t$, and we set
\[
\hat L(u^* x_i v) = L(u^* x_i r_v) \quad \text{and} \quad \hat L((x_i u)^* x_j v) = L((x_i r_u)^* x_j r_v)
\]
for all $i,j \in [n]$ and $u,v \in \ncx$ with $|u|=|v| = t$. One can verify that $\hat L$ is symmetric and satisfies $x_i (u - r_u) \in N_{t+1}(\hat L)$ for all $i \in [n]$ and $u \in \R\ncx_t$, from which it follows that $\hat L$ is $2$-flat.

We also have $(u-r_u)x_i \in N_{t+1}(\hat L)$ for all $i \in [n]$ and $u \in \R\ncx_t$: Since $\hat L$ is $2$-flat, we have $(u-r_u)x_i \in N_{t+1}(\hat L)$ if and only if $\hat L(p (u-r_u) x_i) = 0$ for all $p \in \R\ncx_{t-1}$. By using  $\deg(x_ip) \leq t$, $L$ is tracial, and $u-r_u \in N_t(L)$, we get 
$\hat L(p(u-r_u) x_i) = L(p(u-r_u)x_i)=L(x_i p(u-r_u)) = 0$.

By consecutively using $(v-r_v)x_j \in N_{t+1}(\hat L)$, symmetry of $\hat L$, $x_i (u-r_u) \in N_{t+1}(\hat L)$, and again symmetry of $\hat L$, we see that
\begin{equation}\label{eq:example1} 
\hat L((x_i u)^* v x_j) = \hat L((x_i u)^* r_v x_j) = \hat L((r_v x_j)^* x_i u) = \hat L((r_v x_j)^* x_i r_u) = \hat L((x_i r_u)^* r_v x_j),
\end{equation}
and in an analogous way one can show
\begin{equation}\label{eq:example3}
\hat L((u x_i)^* x_j v ) = \hat L(( r_u x_i)^* x_j r_v ).  
\end{equation}

We can now show that $\hat L$ is tracial. We do this by showing that $\hat L(w x_j) = \hat L(x_j w)$ for all $w$ with $\deg(w) \leq 2t+1$. Notice that when $\deg(w) \leq 2t-1$ the statement follows from the fact that $\hat L$ is an extension of $L$. Suppose $w = u^* v$ with $\deg(u) = t+1$ and $\deg(v) \leq t$. We write $u = x_i u'$, and we let $r_{u'},r_v \in \R\ncx_{t-1}$ be such that $u' - r_{u'}, v-r_v \in N_t(L)$. We then have 
\begin{align*}
\hat L(wx_j) = \hat{L}(u^*vx_j) &= \hat L((x_i u')^* v x_j) \\
&= \hat L((x_i r_{u'})^* r_v x_j) & \text{ by }~\eqref{eq:example1} \\
&= L((x_i r_{u'})^* r_v x_j) &\text{ since } \deg( x_i r_{u'} r_v x_j) \leq 2t \\
&= L( (r_{u'} x_j)^* x_i r_v) &\text{ since } L \text{ is tracial}\\
&= \smash{\hat L}( (r_{u'} x_j)^* x_i r_v) &\text{ since } \deg((r_{u'} x_j)^* x_i r_v) \leq 2t\\
&= \smash{\hat L}((u'x_j)^* x_i v) & \text{ by }~\eqref{eq:example3}\\ 
&= \smash{\hat L}(x_j w).
\end{align*}
It follows $\hat L$ is a symmetric tracial flat extension of $L$, and $\rank(M(\hat L)) = \rank(M(L))$.

Next, we iterate the above procedure to extend $L$ to a symmetric tracial linear functional $\hat L \in \R\ncx^*$. It remains to show that $\hat L$ is nonnegative on $\MM(S)$ and zero on $\mathcal I(T)$. For this we make two observations:
\begin{itemize}
\item[(i)] $\mathcal I(N_t(L)) \subseteq N(\hat L)$. 
\item[(ii)] $\R\ncx = \mathrm{span}(W) \oplus \mathcal I(N_t(L))$.
\end{itemize}
For (i) 
we use the (easy to check) fact that 
$
N_t(L) = \mathrm{span}( \{u-r_u: u \in \ncx_t\}). 
$
Then it suffices to show that $w(u-r_u)\in N(\hat L)$ for all $w\in \ncx$, which can be done using induction on $|w|$.
From (i) one easily deduces that $ \mathrm{span}(W)\cap N(\hat L)=\{0\}$, so  we have the direct sum  $ \mathrm{span}(W) \oplus \mathcal I(N_t(L))$. 
The claim (ii)  follows  using induction on the length of $w \in \ncx$: The base case $w \in \ncx_t$ follows from~\eqref{eq:basecase}. Let $w = x_i v \in \ncx$ and assume $v \in \mathrm{span}(W) \oplus \mathcal I(N_t(L))$, that is, $v= r_v + q_v$ where $r_v \in \mathrm{span}(W)$ and $q_v \in \mathcal I(N_t(L))$. We have $x_i v = x_i r_v + x_i q_v$ so it suffices to show $x_i r_v, x_i q_v \in \mathrm{span}(W) \oplus \mathcal I(N_t(L))$. Clearly $x_i q_v \in \mathcal I(N_t(L))$, since $q_v \in \mathcal I(N_t(L))$. Also, observe that $x_i r_v \in \R\ncx_t$ and therefore $x_i r_v \in \mathrm{span}(W) \oplus \mathcal I(N_t(L))$ by~\eqref{eq:basecase}. 

We conclude the proof by showing that $\hat L$ is nonnegative on $\MM(S)$ and zero on $\mathcal I(T)$. Let $g \in \MM(S)$, $h \in \mathcal I(T)$, and $p \in \R\ncx$. For $p \in \R\ncx$ we extend the definition of $r_p$ so that $r_p \in \mathrm{span}(W)$ and $p -r_p \in \mathcal I(N_t(L))$, which is possible by (ii). Then,
\[
\hat L(p^* g p) \overset{\mathrm{(i)}}{=} \hat L(p^* g r_p) = \hat L(r_p^* g p) \overset{\mathrm{(i)}}{=} \hat L(r_p^* g r_p) = L(r_p^* g r_p) \geq 0, 
\]
\[
\hat L(p^* h) = \hat L(h^* p) \overset{\mathrm{(i)}}{=} \hat L(h^* r_p) = \hat L(r_p h) = L(r_p h) = 0, 
\]
where we use  $\deg(r_p^*gr_p)\le 2(t-\delta)+2\delta=2t$ and $\deg(r_ph)\le (t-\delta)+ 2\delta \le 2t$. \qed
\end{proof}

Combining Theorems~\ref{propLfinitedim} and~\ref{propextension} gives the following result, which shows that a flat linear form can be extended to a conic combination of trace evaluation maps. It was first proven in~\cite[Proposition~6.1]{KP16} (and in~\cite{BK12} for the unconstrained case).
\begin{restatable}{corollary}{resthmflatnonc}
\label{thm:flat-nonc}
Let $1 \leq \delta \leq t < \infty$, $S \subseteq \mathrm{Sym}\, \R\ncx_{2\delta}$, and $T \in \R\ncx_{2\delta}$. If $L \in \R\langle \bx \rangle^*_{2t}$ is  symmetric, tracial, $\delta$-flat, nonnegative on $\smash{\mathcal M_{2t}(S)}$, and zero on $\mathcal I_{2t}(T)$, then it extends to a conic combination of trace evaluations at elements of $\smash{\DS \cap \mathcal V(T)}$.
\end{restatable}

\subsection{Specialization to the commutative setting}\label{sec:comm}

The material from Appendix~\ref{sec:flat ext} can be adapted to the commutative setting. Throughout $\cx$ denotes the set of monomials in $x_1,\ldots,x_n$, i.e., the commutative analog of $\ncx$.

The moment matrix $M_t(L)$ of a linear form $L\in \R\cx_{2t}^*$ is now indexed by the monomials in $\cx_{t}$, where we set $M_t(L)_{w,w'}=L(ww')$ for $w,w'\in \cx_t$. Due to the commutativity of the variables, this matrix is smaller and more entries are now required to be equal. For instance, the $(x_2x_1,x_3x_4)$-entry of $M_2(L)$ is equal to its  $(x_3x_1,x_2x_4)$-entry, which does not hold in general in the noncommutative case. 

Given $a \in \R^n$, the {\em evaluation map} at $a$ is the linear map $L_a\in \R\cx^*$ defined by
\[
L_a(p)= p(a_1,\ldots,a_n) \quad \text{for all} \quad p\in \R\cx.
\]
We can view $L_a$ as a trace evaluation at scalar matrices. Moreover, we can view a trace evaluation map at a tuple of pairwise commuting matrices as a conic combination of evaluation maps at scalars by simultaneously diagonalizing the matrices. 

The  quadratic module $\mathcal M(S)$ and the ideal $\mathcal I(T)$ have immediate specializations to the commutative setting. We recall that in the commutative setting the (scalar) positivity domain and scalar variety of  sets $S,T\subseteq \R\cx$ are given by
\begin{equation}\label{eqDS}
\CDS = \big\{a \in \R^n: g(a)\ge 0 \text{ for } g\in S\big\} \text{, } \quad V(T) = \big\{a \in \R^n: h(a) = 0 \text{ for } h \in T\big\}.\footnote{Note that in the commutative setting we could avoid using the variety since $V(T)=D(\pm T)$. However, in the noncommutative setting, the polynomials in $T$ need not be symmetric in which case the quadratic module $\mathcal D(\pm T)$ would not be well defined.}
\end{equation}

We first give the commutative analogue of Theorem~\ref{propLinfinitedim}, where we give an additional integral representation in point (3). The equivalence of points (1) and (3) is proved in~\cite{Pu93} based on Putinar's Positivstellensatz. Here we give a direct proof on the ``moment side'' using the Gelfand representation.
\begin{restatable}{theorem}{respropLinfinitedimcomm}
\label{propLinfinitedimcomm}
Let $S,T \subseteq \R[\bx]$ with $\MM(S) + \mathcal I(T)$ Archimedean. For $L\in \R[\bx]^*$, the following are equivalent:
\begin{enumerate}
\item[(1)] $L$ is  nonnegative on $\MM(S)$, zero on $\mathcal I(T)$, and $L(1) = 1$;
\item[(2)] there exists a unital commutative $C^*$-algebra $\mathcal A$ with a state $\tau$ and ${\bf X} \in \DAS \cap \mathcal V_{\mathcal A}(T)$ such that $L(p)=\tau(p({\bf X}))$ for all $p\in \R[\bx]$;
\item[(3)] there exists a probability measure $\mu$ on $D(S) \cap V(T)$ such that  
\[
L(p) = \int_{D(S) \cap V(T)} p(x) \, d\mu(x) \quad \text{for all} \quad p\in \R[\bx].
\]
\end{enumerate}
\end{restatable}
\begin{proof}
((1) $\Rightarrow$ (2)) 
This is the commutative analogue of the  implication (1) $\Rightarrow$ (2) in  Theorem~\ref{propLinfinitedim} (observing in addition that the operators $X_i$ in~\eqref{eqXi} pairwise commute so that the constructed $C^*$-algebra $\MA$ is commutative).

((2) $\Rightarrow$ (3)) Let  $\widehat{\mathcal A}$ denote the set of unital $*$-homomorphisms $\mathcal A \to \C$, known as  the \emph{spectrum} of $\MA$. We equip $\smash{\widehat{\mathcal A}}$ with  the  weak-$^*$ topology, so that it  is compact as a result of $\mathcal A$ being unital
(see, e.g.,~\cite[II.2.1.4]{Blackadar06}). The {\em Gelfand representation} is the $*$-isomorphism
\[
\Gamma \colon \mathcal A \to \mathcal C(\widehat{\mathcal A}), \quad \Gamma(a)(\phi) = \phi(a) \quad \text{for} \quad a\in \MA,\ \phi\in \widehat{\mathcal A},
\]
where $\mathcal C(\widehat{\mathcal A})$ is the set of complex-valued continuous functions on $\widehat{\mathcal A}$.
Since $\Gamma$ is an isomorphism, the state $\tau$ on $\MA$ induces a state  $\tau'$ on  $\mathcal C(\smash{\widehat{\mathcal A}})$ defined by $\tau'(\Gamma(a))=\tau(a)$ for $a\in \MA$.
By the Riesz representation theorem (see, e.g.,~\cite[Theorem~2.14]{Rudin87}) there is a Radon measure $\nu$ on  $\smash{\widehat{\mathcal A}}$ such that 
\[
\tau'(\Gamma(a))  = \int_{\widehat{\mathcal A}} \Gamma(a)(\phi) \, d\nu(\phi) \quad \text{for all} \quad a \in \mathcal A.
\]
We then have
\begin{align*}
L(p) &= \tau(p(\bX)) 
= \tau'(\Gamma(p(\bX)))
= \int_{\widehat{\mathcal A}} \Gamma(p(\bX))(\phi) \, d\nu(\phi) = \int_{\widehat{\mathcal A}} \phi(p(\bX)) \, d\nu(\phi)\\
&= \int_{\widehat{\mathcal A}} p(\phi(X_1),\ldots,\phi(X_n)) \, d\nu(\phi) = \int_{\widehat{\mathcal A}} p(f(\phi)) \, d\nu(\phi) = \int_{\R^n} p(x) \, d\mu(x),
\end{align*}
where $f \colon \widehat{\mathcal A} \to \R^n$ is defined by $\phi \mapsto (\phi(X_1),\ldots,\phi(X_n)),
$
and where $\mu = f_*\nu$ is the pushforward measure of $\nu$ by $f$; that is, $\mu(B) = \nu(f^{-1}(B))$ for measurable $B \subseteq \R^n$.

Since $\bX \in \DAS$, we have $g(\bX) \succeq 0$ for all $g \in S$, hence $\Gamma(g(\bX))$ is a positive element of 
$\mathcal C(\smash{\widehat{\mathcal A}})$, implying 
$
g(\phi(X_1), \ldots, \phi(X_n)) = \phi(g(\bX)) =\Gamma(g(\bX))(\phi) \geq 0.
$
Similarly we see $h(\phi(X_1), \ldots, \phi(X_n)) = 0$ for all $h \in T$. So, the range of $f$ is contained in $D(S) \cap V(T)$, $\mu$ is a probability measure on $D(S) \cap V(T)$ since $L(1)=1$, and we have $L(p) = \int_{D(S) \cap V(T)} p(x) \, d\mu(x)$ for all $p \in \R[\bx]$.

((3) $\Rightarrow$ (1)) This is immediate. \qed
\end{proof}
Note that the more common proof for the implication (1) $\Rightarrow$ (3) in Theorem~\ref{propLinfinitedimcomm} relies on Putinar's Positivstellensatz~\cite{Pu93}: if $L$ satisfies (1) then $L(p)\ge 0$ for all polynomials $p$ nonnegative on $D(S)\cap V(T)$ (since $p+\epsilon \in \mathcal M(S) +\mathcal I(T)$ for any $\epsilon>0$), and thus $L$ has a representing measure $\mu$ as in (3) by the Riesz-Haviland theorem~\cite{Hav36}.

The following is the commutative analogue of Theorem~\ref{propLfinitedim}.
\begin{restatable}{theorem}{respropLfinitedimcommutative}
\label{propLfinitedimcommutative}
For $S\subseteq \R\cx$, $T \subseteq \R\cx$, and $L\in \R\cx^*$, the following are equivalent:
\begin{enumerate}
\item[(1)] $L$ is nonnegative on $\MM(S)$, zero on $\mathcal I(T)$, has $\mathrm{rank}(M(L)) < \infty$, and $L(1)=1$;
\item[(2)] there is a finite dimensional commutative $C^*$-algebra $\mathcal A$ 
with a state $\tau$, and ${\bf X} \in \mathcal D_{\mathcal A}(S) \cap  \mathcal V_{\mathcal A}(T)$ 
such that $L(p)=\tau(p({\bf X}))$ for all $p\in \R\cx$;
\item[(3)] $L$ is a convex combination of 
evaluations at points in $D(S) \cap V(T)$.
\end{enumerate}
\end{restatable}

\begin{proof}
((1) $\Rightarrow$ (2)) We indicate how to derive this claim from its noncommutative analogue. 
For this denote the commutative version of $p \in \R\langle {\bf x}\rangle$ by $p^c \in \R[{\bf x}]$.
For any $g\in S$ and $h \in T$, select symmetric polynomials $g',h' \in \R\ncx$ with $(g')^c = g$ and $(h')^c = h$, and set
\[
S' = \big\{ g' : g \in S \big\}\subseteq\R\ncx \ \text{ and }\  T' = \big\{ h' : h \in T \big\} \cup \big\{x_ix_j - x_j x_i \in \R\ncx: i, j \in [n], \, i \neq j\big\}\subseteq\R\ncx.
\]
Define the linear form $L' \in \R\langle {\bf x} \rangle^*$ by $L'(p) = L(p^c)$ for $p\in \R\ncx$. Then $L'$  is symmetric, tracial,  nonnegative on $\mathcal M(S')$, zero on $\mathcal \MI(T')$, and satisfies $\rank M(L')=\rank M(L)<\infty$.
Following the proof of the implication (1) $\Rightarrow$ (2) in Theorem~\ref{propLinfinitedim}, we see that the operators $X_1,\ldots,X_n$ pairwise commute (since $\bX\in \mathcal V_\MA(T')$ and $T'$ contains all $x_ix_j-x_jx_i$) and thus the constructed $C^*$-algebra $\MA$ is finite dimensional and commutative.

((2) $\Rightarrow$ (3))
Here we follow the proof of this implication in Theorem~\ref{propLfinitedim} and  observe that since $\MA$ is finite dimensional and commutative, it is $*$-isomorphic to an algebra of diagonal matrices ($d_m=1$ for all $m\in [M]$), which gives directly the desired result.

((3) $\Rightarrow$ (1)) is easy. \qed
\end{proof}

The next result, due to  Curto and Fialkow~\cite{CF96}, is the commutative analogue of Corollary~\ref{thm:flat-nonc}.
\begin{restatable}{theorem}{resflatcomm}
\label{thm:flat-comm}
Let $1 \leq \delta \leq t < \infty$ and $S,T \subseteq \R[\bx]_{2\delta}$. If $L\in \R\cx_{2t}^*$ is $\delta$-flat, nonnegative on $\mathcal M_{2t}(S)$, and zero on $\mathcal I_{2t}(T)$, then $L$ extends to a conic combination of evaluation maps at points in $D(S) \cap V(T)$. 
\end{restatable}
\begin{proof}
Here too we derive the result from its noncommutative analogue in Corollary~\ref{thm:flat-nonc}.
As in the above proof for the implication  (1) $\Longrightarrow$ (2) in Theorem~\ref{propLfinitedimcommutative},
define the sets $S',T'\subseteq \R\ncx$ and 
 the linear form $L' \in \R\langle {\bf x} \rangle_{2t}^*$ by $L'(p) = L(p^c)$ for $p\in \R\ncx_{2t}$. Then $L'$  is symmetric, tracial, nonnegative on $\mathcal M_{2t}(S')$,  zero on $\mathcal I_{2t}(T')$, and $\delta$-flat. 
By Corollary~\ref{thm:flat-nonc}, $L'$ is a conic combination of trace evaluation maps at elements of $\mathcal D(S') \cap \mathcal V(T')$. 
It suffices now to observe that  such a trace evaluation $L_\bX$ is a conic combination of (scalar) evaluations at elements of $D(S) \cap V(T)$.
Indeed, 
as $\bX\in \mathcal V(T')$,  the matrices $X_1,\ldots,X_n$ pairwise commute and thus can be assumed to be diagonal.
Since $\bX\in \MD(S') \cap \mathcal V(T')$, we have $g(\bX)\succeq 0$ for $g'\in S'$ and $h'(\bX) = 0$ for $h' \in T'$. This implies 
$g((X_1)_{jj},\ldots,(X_n)_{jj})\ge 0$ and $h((X_1)_{jj},\ldots,(X_n)_{jj}) = 0$ for all $g\in S$, $h \in T$, and $j\in [d]$. Thus $L_\bX = \sum_j L_{r_j}$, where  $r_j = ((X_1)_{jj},\ldots,(X_n)_{jj}) \in D(S) \cap V(T)$. \qed
\end{proof}

Unlike in the noncommutative setting, here we also have the following result, which permits to express any linear functional $L$ nonnegative on an Archimedean quadratic module as a conic combination of evaluations at points, when restricting  $L$ to polynomials of bounded degree.
\begin{restatable}{theorem}{resthmevalcomm}
\label{thm:eval-comm}
Let $S,T \subseteq \R\cx$ such that $\mathcal M(S) + \mathcal I(T)$ is Archimedean. 
If $L\in \R\cx^*$ is nonnegative on $\mathcal M(S)$ and zero on $\mathcal I(T)$, then for any integer $k\in \N$ the  restriction of $L$  to  $\R\cx_{k}$ extends to a conic combination of  evaluations at points in $\CDS \cap V(T)$.
\end{restatable}
\begin{proof}
By Theorem~\ref{propLinfinitedimcomm} there exists a  probability measure $\mu$ on $D(S)$ such that  
\[
L(p) = L(1) \int_{D(S) \cap V(T)} p(x) \, d\mu(x) \quad \text{for all} \quad p\in \R[\bx].
\]
A general version of Tchakaloff's theorem, as explained in~\cite{BT}, shows that there  exist $r\in \N$,   scalars $\lambda_1,\ldots,\lambda_r>0$ and points $x_1,\ldots,x_r \in D(S)$ such that
\[
\int_{D(S) \cap V(T)} p(x) \, d\mu(x) = \sum_{i=1}^r \lambda_i p(x_i) \quad \text{for all} \quad p\in \R[\bx]_k.
\]
Hence the restriction of $L$ to $\R[\bx]_k$ extends to a conic combination of evaluations at points in $D(S)$. \qed
\end{proof}

\subsection{Commutative and tracial polynomial optimization}\label{sec:hierarchies}
We briefly discuss here the basic polynomial optimization problems in the commutative and tracial settings. We recall how to design hierarchies of semidefinite programming based bounds and we give their main convergence properties. 
The classical commutative polynomial optimization problem asks to minimize a polynomial $f\in \R\cx$ over a feasible region of the form $D(S)$ as defined in~\eqref{eqDS}:
\[
\fmin = \mathrm{inf}_{a\in \CDS}f(a) = \mathrm{inf}\big\{ f(a) : a \in \R^n, \, g(a)\ge 0 \text{ for } g\in S\big\}.
\]
In tracial polynomial optimization, given $f\in \mathrm{Sym}\, \R\ncx$, this is modified to minimizing $\mathrm{tr}(f({\bf X}))$ over a feasible region of the form $\DS$ as in~\eqref{eqDSnc}:
\[
f_*^\mathrm{tr} = \mathrm{inf}_{{\bf X} \in \DS} \mathrm{tr}(f({\bf X})) = \mathrm{inf}\big\{ \mathrm{tr}(f({\bf X})) : d\in \N,\, {\bf X} \in (\Hermitian^d)^n, \, g({\bf X}) \succeq 0 \text{ for } g\in S\big\},
\]
where the infimum does not change if we replace $\Hermitian^d$ by $\psdcone^d$. Commutative polynomial optimization is recovered by restricting to $1 \times 1$ matrices.

For the commutative case, Lasserre~\cite{Las01} and Parrilo~\cite{Par00} have proposed hierarchies of semidefinite programming relaxations based on sums of squares of polynomials and the dual theory of moments. This approach has been extended to eigenvalue optimization~\cite{NPA10,NPA12} and later to tracial optimization~\cite{BCKP13,KP16}. The starting point in deriving these relaxations is to reformulate the above problems as minimizing $L(f)$ over all normalized trace evaluation maps $L$ at points in $\CDS$ or $\DS$, and then to express computationally tractable properties satisfied by such maps $L$.

For $S \cup \{f\} \subseteq \R[{\bf x}]$ and $\lceil\deg(f)/2\rceil \leq t \leq \infty$, recall the (truncated) quadratic module $\mathcal M_{2t}(S)$ 
 \[
\MM_{2t}(S) =\mathrm{cone}\big\{gp^2: p\in \R[{\bf x}], \ g\in S\cup\{1\},\ \deg(gp^2)\le 2t\big\},
\]
which we use  to formulate the following semidefinite programming  lower bound on $\fmin$:
\[
f_t =\mathrm{inf}_{}\big\{L(f) : L\in \R\cx_{2t}^*,\, L(1)=1,\, L\ge 0  \text{ on }  \MM_{2t}(S)\big\}.
\]
For $t\in\N$ we have $f_t\le f_\infty\le f_*$. 

In the same way, for $S \cup \{f\} \subseteq \mathrm{Sym} \, \R\langle {\bf x} \rangle$  and $t$ such that $\lceil\deg(f)/2\rceil \leq t \leq \infty$, we have the following semidefinite programming lower bound on $f_*^\mathrm{tr}$: 
\[
f_t^\mathrm{tr} =\mathrm{inf}_{}\big\{L(f) : L\in \R\langle {\bf x}\rangle_{2t}^* \text{ tracial and symmetric},\, L(1)=1,\, L \ge 0 \text{ on } \MM_{2t}(S)\big\},
\]
where we now use definition~\eqref{eq:quadratic module} for $\MM_{2t}(S)$.

The next theorem from~\cite{Las01} gives fundamental convergence properties for the commutative case; see also, e.g.,~\cite{Las09,Lau09} for a detailed exposition. 
\begin{restatable}{theorem}{resthmconvcomm}\label{thm:convcomm}
Let $1 \leq \delta \leq t < \infty$ and  $S \cup \{f\} \subseteq \R\cx_{2\delta}$ with $D(S)\ne \emptyset$.
\begin{itemize}
\item[(i)] If  $\MM(S)$ is Archimedean, then $f_t \to f_\infty$ as $t \to \infty$,  
the optimal values in $f_\infty$ and $\fmin$ are attained,  and $f_\infty = \fmin$.
\item[(ii)] If $f_t$ admits an optimal solution $L$ that is $\delta$-flat, then $L$ is a convex combination of evaluation maps at global minimizers of $f$ in $\CDS$, and $f_t=f_\infty =\fmin$.
\end{itemize}
\end{restatable}
\begin{proof}
(i) By repeating the first part of the proof of Theorem~\ref{thm:convnc} in the commutative setting we see that $f_t \to f_\infty$ and that the optimum is attained in $f_\infty$. Let $L$ be optimal for $f_\infty$ and let $k$ be greater than $\mathrm{deg}(f)$ and $\mathrm{deg}(g)$ for $g \in S$. By Theorem~\ref{thm:eval-comm}, the restriction of $L$ to $\R[\bx]_k$ extends to  a conic combination of evaluations at points in $D(S)$. It follows that this extension if feasible for $f_*$ with the same objective value, which shows $f_\infty = f_*$.

(ii) This follows in the same way as the proof of Theorem~\ref{thm:convnc}(ii) below, where, instead of using Corollary~\ref{thm:flat-nonc}, we now use its commutative analogue, Theorem~\ref{thm:flat-comm}. \qed
\end{proof}

To discuss convergence for the tracial case we need one more optimization problem: 
\[
f_\mathrm{II_1}^\mathrm{tr} = \mathrm{inf} \big\{ \tau(f({\bf X})) : {\bf X} \in \DAS, \, \mathcal A \text{ is a unital $C^*$-algebra with tracial state } \tau\big\}.
\]
This problem can be seen as an infinite dimensional analogue of $f_*^{\mathrm{tr}}$: if we restrict to finite dimensional $C^*$-algebras in the definition of $f_{\mathrm{II_1}}^{\mathrm{tr}}$, then we recover the parameter $f_*^{\mathrm{tr}}$ (use Theorem~\ref{propLfinitedim} to see this). 
Moreover, as we see in Theorem~\ref{thm:convnc}(ii) below, equality $f_*^{\mathrm{tr}} = f_{\mathrm{II_1}}^{\mathrm{tr}}$ holds if some flatness condition is satisfied.
Whether $f_\mathrm{II_1}^\mathrm{tr} = f_*^\mathrm{tr}$ is true in general is related to Connes' embedding conjecture (see~\cite{KS08,KP16,BKP16}).

Above we defined the parameter $f_\mathrm{II_1}^\mathrm{tr}$ using $C^*$-algebras. However, the following lemma
shows that we get the same optimal value if we restrict to $\MA$ being a von Neumann algebra of type~$\mathrm{II_1}$ with separable predual, which is the more common way of defining the parameter  $f_\mathrm{II_1}^\mathrm{tr}$ as is done in~\cite{KP16} (and justifies the notation).
We omit the proof of this lemma which relies on a GNS construction and  algebraic manipulations, standard  for algebraists.

\begin{restatable}{lemma}{resremarkvonNeumann}
\label{remark:von Neumann}
Let $\mathcal A$ be a $C^*$-algebra with tracial state $\tau$ and $a_1,\ldots,a_n \in \mathcal A$. There exists a von Neumann algebra $\mathcal F$ of type $\mathrm{II_1}$ with separable predual, a faithful normal tracial state $\phi$, and elements $b_1,\ldots,b_n \in \mathcal F$, so that for every $p \in \R\langle{\bf x}\rangle$ we have 
\[
\tau(p(a_1,\ldots,a_n)) = \phi(p(b_1,\ldots,b_n)) \quad \text{ and }
\]
\[
p(a_1,\ldots, a_n) \text{ is positive} \quad \iff \quad p(b_1,\ldots,b_n) \text{ is positive}.
\]
\end{restatable}

For all $t \in \N$ we have
$$
f_t^\mathrm{tr} \leq \ftr_{\infty} \le f_\mathrm{II_1}^\mathrm{tr} \leq f_\mathrm{*}^\mathrm{tr},
$$ 
where the last inequality follows by considering for $\MA$ the full matrix algebra $\C^{d\times d}$. 
The next theorem from~\cite{KP16} summarizes convergence properties for these parameters, its proof uses  Lemma \ref{lemma:upperboundLw} below.

\begin{restatable}{theorem}{resthmconvnc}\label{thm:convnc}
Let $1 \leq \delta \leq t < \infty$ and $S\cup\{f\}\subseteq \mathrm{Sym}\, \R\ncx_{2\delta}$ with $\MD(S)\ne \emptyset$.
\begin{itemize}
\item[(i)] If $\MM(S)$ is Archimedean, then $f_t^{\text{tr}} \to f_\infty^\mathrm{tr}$ as $t \to \infty$, and the optimal values in $\ftr_{\infty}$ and $f_\mathrm{II_1}^\mathrm{tr}$ are attained and equal.
\item[(ii)] If $f_t^{\text{tr}}$ has an optimal solution $L$ that is $\delta$-flat, then $L$ is a convex combination of normalized trace evaluations at matrix tuples in $\DS$, and $f_t^{\text{tr}}=f_\infty^{\text{tr}}=f_\mathrm{II_1}^\mathrm{tr} =f_*^\mathrm{tr}$.
\end{itemize}
\end{restatable}
\begin{proof}
We first show  (i). As  $\MM(S)$ is Archimedean,   $R-\sum_{i=1}^nx_i^2\in \MM_{2d}(S)$ for some $R>0$ and $d\in \N$.
Since the bounds $\ftr_t$ are monotone nondecreasing in $t$ and upper bounded by $\ftr_\infty$, the limit $\lim_{t\rightarrow \infty} \ftr_t$ exists and it is at most $\ftr_\infty$. 

Fix $\epsilon>0$. For  $t\in \N$ let $L_t$ be a feasible solution to the program defining $\ftr_t$ with value $L_t(f)\le \ftr_t+\epsilon$.
As $L_t(1)=1$ for all $t$ we can apply Lemma~\ref{lemma:upperboundLw} below and conclude that the sequence $(L_t)_t$ has a convergent subsequence. 
 Let $L\in \R\ncx^*$ be the pointwise limit. One can easily check that $L$ is feasible for  $\ftr_\infty$. Hence we have 
 $\ftr_\infty \le L(f)\le \lim_{t\to\infty} \ftr_t +\epsilon\le \ftr_\infty +\epsilon$. Letting $\epsilon\to 0$ we obtain that $\ftr_\infty=\lim_{t\to\infty}\ftr_t$ and 
 $L$ is optimal for  $\ftr_\infty$.

Next, since $L$ is symmetric, tracial, and nonnegative on $\MM(S)$, we can apply Theorem~\ref{propLinfinitedim} to obtain a feasible solution $(\MA,\tau,\bX)$ to $f_\mathrm{II_1}^\mathrm{tr}$ satisfying~\eqref{eqLA} with objective value $L(f)$. This shows $\ftr_\infty= f_\mathrm{II_1}^{\mathrm{tr}}$ and that the optima are attained in $\ftr_\infty$ and  $\ftr_\mathrm{II_1}$. 

Finally, part (ii) is derived as follows. If $L$ is an optimal solution of $\ftr_t$ that is $\delta$-flat, then, by Corollary~\ref{thm:flat-nonc}, it has an extension $\hat L\in \R\ncx^*$ that is a conic combination of trace evaluations at elements of $\DS$. This shows $\smash{\ftr_*} \le \smash{\hat L(f)} = L(f)$, and thus the chain of equalities 
$\ftr_t=\ftr_\infty= \ftr_*=\ftr_{\Pi_1}$ holds.  
\end{proof}

We conclude with the following technical lemma, based on the Banach-Alaoglu theorem. It is a well known crucial tool for proving  the asymptotic convergence result from Theorem~\ref{thm:convnc}(i) and it is used at other places in the paper.

\begin{restatable}{lemma}{reslemmaupperboundLw}\label{lemma:upperboundLw}
Let $S \subseteq \mathrm{Sym}\, \R\ncx$, $T \subseteq \R\ncx$, 
and assume  $R-(x_1^2 + \cdots + x_n^2) \in \MM_{2d}(S)+\mathcal I_{2d}(T)$  for some $d\in \N$ and $R>0$.
 For $t\in \N$ assume $L_t \in \smash{\R\langle \bx \rangle_{2t}^*}$ is tracial, nonnegative on $\mathcal M_{2t}(S)$ and zero on $\MI_{2t}(T)$. 
Then we have
$\smash{|L_t(w)|\le R^{|w|/2} L_t(1)}$ for all $w\in \ncx_{2t-2d+2}$. In addition, if $\sup_t \, L_t(1) < \infty$, then  $\smash{\{L_t\}}_t$ has a pointwise converging subsequence in $\smash{\R\langle \bx \rangle^*}$.
\end{restatable}
\begin{proof}
We first use induction on $|w|$ to show that $L_t(w^*w)\le R^{|w|}L_t(1)$ for all $w\in \ncx_{t-d+1}$. 
For this, assume $L_t(w^*w)\le R^{|w|}L_t(1)$ and $|w|\le t-d$. Then we have
$$L_t((x_iw)^*x_iw) =L_t(w^*(x_i^2-R)w)+R \cdot L_t(w^*w) \le R\cdot R^{|w|}L_t(1)=R^{|x_iw|}L_t(1).$$
For the inequality we use the fact that $L_t(w^*(x_i^2-R)w)\le 0$ since $w^*(R-x_i^2)w$ can be written as the sum of a polynomial in $\mathcal M_{2t}(S)+\mathcal I_{2t}(T)$ and a sum of commutators of degree at most $2t$, which follows using the following identity:
$w^*qhw=ww^*qh+[w^*qh,w].$
Next we write any $w\in\ncx_{2(t-d+1)}$ as $w=w_1^*w_2$ with $w_1,w_2\in \ncx_{t-d+1}$ and use the positive semidefiniteness of the principal submatrix of $M_t(L_t)$ indexed by $\{w_1,w_2\}$ to get
\[
L_t(w)^2 = L_t(w_1^*w_2)^2\le  L_t(w_1^*w_1)L_t(w_2^*w_2) \le  R^{|w_1|+|w_2|}L_t(1)^2=R^{|w|}L_t(1)^2.
\]
This shows the first claim.

Suppose $c:=\sup_t \, L_t(1) < \infty$. For each $t \in \N$,  consider the linear functional $\hat L_t\in\R\ncx^*$ defined by 
$\hat L_t(w)=L_t(w)$ if $|w|\le 2t-2d+2$ and $\hat L_t(w)=0$ otherwise.
Then the vector $(\hat L_t(w)/(c R^{|w|/2}))_{w \in \langle{\bf x}\rangle}$ lies in the supremum norm unit ball of $\smash{\R^{\langle {\bf x} \rangle}}$, which is compact in the weak$*$ topology by the Banach--Alaoglu theorem. It follows that the sequence $(\hat L_t)_t$ has a pointwise converging subsequence and thus the same holds for the sequence $(L_t)_t$. \qed
\end{proof}






\end{document}